\documentclass[reqno, 12pt]{amsart}
\usepackage[letterpaper,hmargin=1in,vmargin=1in]{geometry}
\usepackage{amsmath,amssymb,amsthm}
\usepackage{verbatim,epsfig,graphics}
\usepackage{graphicx}
\usepackage{epstopdf}
\usepackage{color}
\newtheorem{theorem}{Theorem}
\newtheorem{proposition}[theorem]{Proposition}
\newtheorem{corollary}[theorem]{Corollary}
\newtheorem{lemma}[theorem]{Lemma}

\theoremstyle{remark} \newtheorem{remark}{Remark}
\newcommand{\ip}[2]{\langle #1 , #2 \rangle}

\newcommand{\mc}{\mathcal}
\newcommand{\rr}{\mathbb{R}}
\newcommand{\nn}{\mathbb{N}}
\newcommand{\cc}{\mathbb{C}}

\newcommand{\zz}{\mathbb{Z}}
\newcommand{\sph}{\mathbb{S}}

\newcommand{\la}{\lambda}
\newcommand{\eps}{\epsilon}

\newcommand{\pl}{\partial}
\newcommand{\x}{\times}

\newcommand{\til}{\widetilde}
\newcommand{\bbar}{\overline}

\newcommand{\cjd}{\rangle}
\newcommand{\cjg}{\langle}

\newcommand{\demi}{\frac{1}{2}}
\newcommand{\ndemi}{\frac{n}{2}}
\newcommand{\tra}{\textrm{Tr}}
\newcommand{\Ima}{\textrm{Im}}

\newcommand{\zf}{\textrm{zf}}
\newcommand{\bfo}{\textrm{bf}_0}
\newcommand{\rbo}{\textrm{rb}_0}
\newcommand{\lbo}{\textrm{lb}_0}
\newcommand{\lb}{\textrm{lb}}
\newcommand{\rb}{\textrm{rb}}
\newcommand{\bfa}{\textrm{bf}}
\newcommand{\bfc}{\textrm{bf}}
\newcommand{\sca}{\textrm{sc}}

\newcommand\Id{\operatorname{Id}}

\newcommand\RR{\mathbb{R}}

\newcommand\extunion{\overline{\cup}}

\newcommand\MMksc{M^2_{k, \sca}}

\renewcommand\Re{\operatorname{Re}}

\newcommand\Mbar{M}

\begin{document}
\title[Resolvent at low energy and Riesz transform]{Low energy resolvent for the Hodge Laplacian: Applications to Riesz transform, Sobolev estimates and analytic torsion}
\author{Colin Guillarmou}
\address{DMA, ENS, 45 rue d'Ulm, 75005 Paris.}     
\email{cguillar@dma.ens.fr}
\author{David A. Sher}
\address{Department of Mathematics, University of Michigan, 2074 East Hall, 530 Church Street, Ann Arbor, MI 48109-1043.}
\email{dsher@umich.edu}
\subjclass[2000]{}
%

\begin{abstract}
On an asymptotically conic manifold $(M,g)$, we analyze the asymptotics of the integral 
kernel of the resolvent $R_q(k):=(\Delta_q+k^2)^{-1}$ of the Hodge Laplacian $\Delta_q$ on $q$-forms as the spectral parameter $k$ approaches zero, assuming that $0$ is not a resonance.  
The first application we give is an $L^p$ Sobolev estimate for $d+\delta$ and $\Delta_q$.
Then we obtain a complete characterization of the range of $p>1$ 
for which the Riesz transform for $q$-forms $T_q=(d+\delta)\Delta_q^{-1/2}$ is bounded on $L^p$.
Finally, we obtain an asymptotic formula for the analytic torsion of a family of smooth compact Riemannian manifolds $(\Omega_\eps,g_\eps)$ degenerating to a compact manifold $(\Omega_0,g_0)$ with a conic singularity as $\eps\to 0$.
\end{abstract}
\maketitle

\section{Introduction}

Let $(M,g)$ be an $n$-dimensional \emph{asymptotically conic} manifold with cross-section 
a closed Riemannian manifold $(N,h_0)$. Such a manifold is the interior of a smooth compact manifold $\bbar{M}$ with boundary $\pl\bbar{M}=N$, equipped with a complete smooth metric $g$ with the following property: there exists a smooth boundary defining function $x\in C^{\infty}(\bbar{M})$ (i.e. $\{x=0\}=N$ and $dx|_N$ does not vanish) such that near $x=0$, 
the metric $g$ can be written in the form
\[g= \frac{dx^2}{x^4}+\frac{h(x)}{x^2}\]
with $h(x)$ a smooth family of metrics on $N$ such that $h(0)=h_0$. Notice that, setting $x=r^{-1}$, a neighbourhood of 
$\pl\bbar{M}$ equipped with $g$ is asymptotic to the metric cone $(\rr_r^+\x N , dr^2+r^2h_0)$ as $r\to \infty$. 
In the special case where $N=\sph^{n-1}$ with the usual metric (or is a disjoint union of copies of $\sph^{n-1}$), we say that $(M,g)$ is \emph{asymptotically Euclidean}.
For technical purposes, we assume
\begin{equation}\label{assumhx}
h(x)-h_0=\mc{O}(x^{n_0}), \quad n_0\geq 3
\end{equation}
and we say that $(M,g)$ is asymptotically conic to order $n_0$. 

Let $d$ be the exterior derivative acting on differential forms and $\delta$ its formal adjoint.  
The Hodge Laplacian on $q$-forms is defined by $\Delta_q=d\delta+\delta d$ and its spectrum is 
$[0,\infty)$. For $\Re(k)>0$, the resolvent $R_q(k)=(\Delta_q+k^2)^{-1}$ is well defined as a bounded operator on 
$L^2(M,dg)$. In this article, we analyze the behaviour of this operator as $k>0$ goes to $0$ by using a parameter-dependent pseudo-differential calculus adapted to the geometry, which was introduced by the first author in collaboration with Hassell \cite{gh1}. The pseudo-differential calculus $\Psi^*_k(M)$ of \cite{gh1}  is recalled in Section \ref{pdcle} below. It is described through Schwartz kernels of operators: a $k$-dependent family of operators $A(k)$ lies in the calculus when it has a Schwartz kernel $A(k;z,z')$ which is a polyhomogeneous conormal distribution on a manifold $\MMksc$ obtained by a sequence of blow-ups from $[0,1]_k\x \bbar{M}_z \x\bbar{M}_{z'}$ (see Section \ref{mmksc} for the definition of $\MMksc$). More informally, this means that the kernel $A(k;z,z')$ has full asymptotic expansions as $k\to 0$, $z\to y\in \pl\bbar{M}$, $z'\to y'\in \pl\bbar{M}$ under certain regimes of convergence.
  
In the construction of the parametrix for the low-energy resolvent $R(k)$, we make two assumptions. 
The first assumption is that the operator $\Delta_q$ has no \emph{zero-resonances}, which means that 
\begin{equation}\label{Int_noresonance}
\ker_{x^{-1}L^2}(\Delta_q)=\ker_{L^2}(\Delta_q).
\end{equation}
In our geometric setting, it turns out that this condition is equivalent to 
\begin{equation}\label{kerLr}
\ker_{L^2}(\Delta_q)=\ker_{L^r}(\Delta_q) , \quad \forall r\in\Big[2,\frac{2n}{(n-2)}\Big],
\end{equation}
where $L^r=L^r(M,dg)$.

Our second assumption involves the spectrum of the Hodge Laplacian $\Delta_N=d_N\delta_N+\delta_Nd_N$ acting on the form bundle $\oplus_{p=0}^{n-1}\Lambda^{p}(N)$ of the cross-section $(N,h_0)$, 
where $d_N$ is the exterior derivative on $N$ and $\delta_N$ its formal adjoint (with respect to $h_0$): we assume that
\begin{equation}\label{0notindroot}
\left\{\begin{array}{lll}
|q-\ndemi|\leq 1/2 \Rightarrow 1-(\ndemi-q)^2\notin{\rm Sp}_{\Lambda^q}(\Delta_N|_{{\rm Im}\,d_N}), \\
H^{q}(N)=0 \textrm{ if }q=\ndemi-1,\\
H^{q-1}(N)=0 \textrm{ if }q=\ndemi+1
\end{array}\right.
\end{equation}
where ${\rm Sp}_{\Lambda^p}(\Delta_N)$ denotes the spectrum of $\Delta_N$ acting on $p$-forms on $N$ 
and $H^q(N)=\ker_{\Lambda^q}(\Delta_N)$ is the $q$-th de Rham cohomology of $N$.

\begin{theorem}\label{introresthm} 
Let $(M,g)$ be asymptotically conic to order $n_0\geq 3$ and assume \eqref{Int_noresonance} and \eqref{0notindroot}. Then there exists $k_0>0$ such that the resolvent 
$R_q(k) = (\Delta_q + k^2)^{-1}$ is a pseudo-differential operator in the calculus $\Psi_k^*(M)$ for $k\leq k_0$. 
\end{theorem}
A more precise statement, including the orders of the operator and describing the polyhomogeneity of the Schwartz kernel at the various boundary hypersurfaces of $\MMksc$, is given in Theorem  \ref{thm:nullspace}. As Theorem \ref{thm:nullspace} is proved by a parametrix construction, it also gives the explicit leading-order asymptotic terms of the kernel at all faces.

\begin{remark} Zero-resonances can appear only for degrees $q$ such that $|q-n/2|\leq 1$, and they are absent under certain assumptions on the bottom of the spectrum of $\Delta_N$ on forms of degree $r\in[n/2-1,n/2]$; for one such assumption, see Lemma \ref{wittconseq} when $n$ is odd and Remark \ref{noresonanceneven} when $n$ is even.
In fact, assumption \eqref{Int_noresonance} could likely be removed, but the parametrix construction would be much more technically involved, similar to the work \cite{gh2} for Schr\"odinger operators on functions. However it should be noticed that from the analysis of \cite{gh2}, there are likely some cases with zero-resonances where the resolvent is not a pseudo-differential operator in the calculus $\Psi_k^*(M)$.
Assumption \eqref{0notindroot} is likely not necessary either, but the construction would be more complicated - in fact, quite similar to the analysis of the resolvent on functions in dimension $n=2$ done in \cite[Section 4]{s1}.
We finally mention that assumptions \eqref{Int_noresonance}  and \eqref{0notindroot} are always satisfied on asymptotically Euclidean manifolds of dimension $n\geq 3$ (ie. when $(N,h_0)$ is a disjoint union of canonical 
spheres $(\sph^{n-1},d\theta)$).
\end{remark}

\textbf{Application to Sobolev estimates.}
We first give a Sobolev inequality which follows from the resolvent description. 
\begin{theorem}\label{mainthSob}
Let $(M,g)$ be asymptotically conic to order $n_0\geq 3$, let $p=2n/(n+2)$ and define the conjugate exponent $p'=2n/(n-2)$. Assume \eqref{0notindroot} and that $\ker_{L^2}(\Delta_q)=\ker_{L^r}(\Delta_q)$ for all $r\in [p,p']$.
Then there exists $C>0$ such that for all $q$-forms $u,v\in C_0^\infty(M;\Lambda^q)$ 
\begin{equation}\label{Sobolevest} 
||({\rm Id}-\Pi_{\ker(\Delta_q)})v||_{L^{p'}}\leq C||(d+\delta)v||_{L^2}, \quad 
||({\rm Id}-\Pi_{\ker(\Delta_q)})u||_{L^{p'}}\leq C||\Delta_q u||_{L^p},
\end{equation}
where $\Pi_{\ker(\Delta_q)}$ is the orthogonal projector on $\ker_{L^2}(\Delta_q)$ in $L^2$.
\end{theorem}
Of course these inequalities extend by continuity to $u,v$ in appropriate Sobolev spaces (see Theorem \ref{SobolevTh}). 
The conditions \eqref{0notindroot} and $\ker_{L^2}(\Delta_q)=\ker_{L^r}(\Delta_q)$ for all $r\in [p,p']$ are satisfied when $N$ is a disjoint union of canonical spheres.
Uniform Sobolev estimates (for $\Delta-\lambda$) were recently proved for functions in the same geometric setting by the first author and Hassell; see \cite{GH3}. For differential forms, Li \cite{Li} proves some Sobolev estimates of the same form for $d+\delta$ on complete manifolds under some curvature conditions (non-negativity of some curvature tensor). \\

\textbf{Application to Riesz transform on forms.}
The Riesz transform acting on functions on a complete Riemannian manifold 
is defined by $T_0=d\Delta^{-1/2}$ and is bounded from the space of $L^2$ functions to the space of $L^2$ $1$-forms. 
It is a classical question in harmonic analysis (asked for instance by Strichartz \cite{St}) to understand for which $p$ the map $T_0$ is bounded on $L^p$. We refer for instance to Section 1.3 of the paper \cite{ACDH} by Auscher-Coulhon-Duong-Hoffman for a quite complete list of results in the geometric setting. For instance, Bakry \cite{B} proved that $T_0$ is bounded on any $L^p$ for $p\in(1,\infty)$ if $(M,g)$ is a complete manifold with non-negative Ricci curvature,  and Coulhon-Duong \cite{CD} obtained the quite general result stating that $T_0$ is bounded on $L^p$ for $p\in(1,2]$ when the volume of balls satisfies the doubling property and the heat kernel satisfies Gaussian upper estimates. On the other hand, for $p>2$, there exist simple examples where $T_0$ is not bounded on $L^p$. For instance,
it is shown by Carron-Coulhon-Hassell \cite{CCH} that an $n$-dimensional manifold with two ends isometric to 
$\rr^n\setminus B(0,R)$ has $T_0$ bounded on $L^p$ if and only if $p\in(1,n)$; this result has been generalized significantly by Devyver \cite{De}.

As in \cite{St}, we define the Riesz transform on $q$-forms as the operator taking $q$-forms on $M$ to a 
direct sum of $(q-1)$ and $(q+1)$-forms on $M$ defined by 
\[T_q=D\Delta_q^{-1/2} \,\, \textrm{ where }D:=d+\delta\]
(to make sense of $T_qf\in L^2$ when $f\in L^2$, we can consider weak limits of 
$D(\Delta_q+\eps)^{-1/2}f$ as $\eps\to 0^+$; see the beginning of Section \ref{secRT}). 

In this work we consider the sharp range of $p$ for which  $T_q$  is bounded on an asymptotically conic manifold 
with cross section $(N,h_0)$. The answer turns out to be quite complicated, and it can be expressed in terms of both topological and spectral data:
first the cohomology of $\bbar{M}$, then the small eigenvalues of the Laplacian $\Delta_N$ on the cross section, and finally the rate of decay of $L^2$ harmonic $q$-forms on $M$. To state the result we introduce the following indices
related to the Laplacian $\Delta_N$ on forms on $N$:  
writing ${\rm Sp}_{\Lambda^p}(\Delta_N|_H)$ for the spectrum of the Laplacian $\Delta_N$ on $p$-forms
restricted to a vector space $H\subset L^2(N,\Lambda^p(N))$, we define for $p=q-1,q,q+1$
\[\begin{gathered}
\la_p:=\min({\rm Sp}_{\Lambda^p}(\Delta_N|_{{\rm Im}\,d_N})), \,\, \mu_p:=\min({\rm Sp}_{\Lambda^p}(\Delta_N|_{({\rm Im}\,\delta_N)^\perp})),\,\,  \gamma_p:=\min({\rm Sp}_{\Lambda^p}(\Delta_N|_{({\rm Im}\,d_N)^\perp})),\\
\gamma'_q:=\left\{\begin{array}{ll}
\gamma_q & \textrm{ if }q\geq n/2-1\\
\la_{q+1} & \textrm{ if }q< n/2-1
\end{array}\right. ,  \quad 
\mu'_q=\left\{\begin{array}{ll}
\mu_q & \textrm{ if }q\leq n/2+1\\
\la_{q} & \textrm{ if }q> n/2-1.
\end{array}\right.
\end{gathered}\]
To state the result as smoothly as possible, we make an extra assumption in the Introduction which will be removed later in the paper: we assume that $\la_q>1-|q-n/2|^2$. Then under this assumption we define
\begin{equation}\label{defnu0intro}
\begin{split}
\nu_0:= &\min\Big(\sqrt{(\ndemi-q)^2+\la_q}-1,\, \sqrt{(\ndemi-q-1)^2+\gamma_q}, \,
\sqrt{(\ndemi-q+1)^2+\mu_{q-1}}
\, \Big)\\
\nu_D:= & \min\Big(\sqrt{(\ndemi-q)^2+\la_q}+1, \,\sqrt{(\ndemi-q-1)^2+\gamma'_q}, \,\sqrt{(\ndemi-q+1)^2+\mu'_{q-1}}\, \Big).
\end{split}
\end{equation} 
Then we prove 
\begin{theorem}\label{first} 
Let $(M,g)$ be asymptotically conic to order $n_0\geq 3$  with cross-section $(N,h_0)$ and assume that 
$\la_q>1-|q-n/2|^2$. Let $\nu_0,\nu_D$ be  be the indices defined by \eqref{defnu0intro} from the spectrum of $\Delta_N$ on $(q-1)$ and $q$-forms. Assume that \eqref{0notindroot} and \eqref{Int_noresonance} hold and finally, define 
\begin{equation}\label{nukerintro} 
\nu_{\ker}:= \min \Big(\nu_0+2 ,\, \max \{ \nu\geq \nu_0; \ker_{L^2}(\Delta_q)\subset x^{\nu+\ndemi-1}L^\infty\}\Big).
\end{equation}
Then the Riesz transform $T_q$ is bounded on $L^p$ if 
\begin{equation}\label{intervalsuff} 
\frac{n}{n-(n/2+1-\nu_{\ker})_+}<p< \frac{n}{(n/2-\nu_0)_+}.
\end{equation}
To obtain the exact interval of $L^p$ boundedness, we make the extra assumption 
that $n_0>\nu_0+2$ if $\nu_0<n/2$. Under this additional assumption, we have:\\
\noindent\textbf{Case $1$.} If $q<n/2-1$ and the natural map $H^{q+1}(\bbar{M},\pl \bbar{M})\to H^{q+1}(\bbar{M})$ in cohomology is not injective, or if 
$q>n/2+1$ and the natural map $H^{n-q+1}(\bbar{M},\pl \bbar{M})\to H^{n-q+1}(\bbar{M})$ in cohomology is not injective, then the Riesz transform 
$T_q=D\Delta_q^{-1/2}$ on $q$-forms is bounded on $L^p$ if and only if \eqref{intervalsuff} holds.\\
\textbf{Case $2$.} In all other cases, $T_q$ is bounded on $L^p$ if and only if 
\[  \frac{n}{n-(n/2+1-\nu_{\ker})_+}<p< \frac{n}{(n/2-\min(\nu_D,\nu_{\ker}))_+}.\]
\end{theorem}
\begin{remark}
The same result holds without assuming $\la_q>1-|q-n/2|^2$, but the indices $\nu_0$ and $\nu_D$ need to be defined slightly differently (ie. they need to be defined by \eqref{defnu0} and \eqref{defindices}). The general case is 
in Theorem \ref{mainthRT}.
\end{remark}

\begin{remark} For $q=0$, $H^1(\bbar{M},\pl \bbar{M})=0$ if and only if $M$ has one end, and when $H^1(\bbar{M},\pl \bbar{M})\not= 0$ the map $H^1(\bbar{M},\pl \bbar{M})\to H^{1}(\bbar{M})$ is never injective. In particular, since for $q=0$ one has 
$\nu_0=n/2-1$ and $\ker_{L^2}(\Delta_q)=0$, we recover 
Theorem 1.5 in \cite{gh1} about the Riesz transform on functions by applying Case 1 of Theorem \ref{first}. We also recover Theorem 1.4 of \cite{gh1} by applying Case 2 of Theorem \ref{first} since
$\nu_D=\sqrt{(\ndemi-1)^2+\lambda_1}$ where $\lambda_1$ is the first non-zero eigenvalue of the Laplacian $\Delta_{N}$ on functions (or equivalently exact $1$-forms).
\end{remark}

Using that $\nu_0\geq |q-n/2|-1$ when $|q-n/2|>1$, we obtain the first corollary, which is weaker than Theorem \ref{first} in the sense that it does {\bf not} give the sharp range of $p$ for $L^p$ boundedness of $T_q$, but it has the advantage of being stated only in terms of the degree $q$:
\begin{corollary}
Let  $q$ satisfy $|q-n/2|>1$ and let $(M,g)$ be asymptotically conic to order $n_0\geq 3$ with cross-section $(N,h_0)$. Assume that $\ker_{x^{-1}L^2}(\Delta_q)=\ker_{L^2}(\Delta_q)=0$; then 
the Riesz transform $T_q$ is always bounded on $L^p$ if 
\begin{equation}\label{firstboundcor}  
\frac{n}{n/2+|n/2-q|}<p< \frac{n}{n/2+1-|n/2-q|}.
\end{equation}
If $\ker_{x^{-1}L^2}(\Delta_q)=\ker_{L^2}(\Delta_q)\not=0$, then $T_q$ is always bounded on $L^p$ if 
\[  \min\Big(2,\, \frac{n}{n/2-2+|n/2-q|}\Big)<p< \frac{n}{n/2+1-|n/2-q|}.\]
\end{corollary}
\begin{remark}
By Theorem \ref{first}, we see that the lower bound in \eqref{firstboundcor} is sharp if the cohomology $H^{q}(N)$ is non-trivial. This contrast with the case of Riesz transform on functions where the lower exponent of boundedness 
in a very general case is $1$, see \cite{CD}. It is not unlikely that such a (in general non-sharp) result could be extended to a more general setting, such as manifolds with volumes of large balls being comparable to those of Euclidean balls of dimension $n$, and satisfying some bounds on the curvature tensor as well as some Sobolev inequality for $D$ as in Theorem \ref{mainthSob} (see \cite{De} for the case $q=0$).\end{remark}

When $N=\sph^{n-1}$ is the sphere with curvature $+1$, one has $\nu_0=n/2-1$ (see \eqref{caseSn}), and the map $H^{q+1}(\bbar{M},\pl \bbar{M})\to H^{q+1}(\bbar{M})$ is always injective for $0<q<n$ since $H^q(\sph^{n-1})=0$. Thus Theorem \ref{first} applied to asymptotically Euclidean manifolds with $n\geq 3$ and $n_0>n/2+1$ gives
\begin{corollary}\label{euclidean}
Let $(M,g)$ be asymptotically Euclidean to order $n_0>n/2+1$, with dimension $n\geq 3$.
Then $\nu_{\ker}\in\{\ndemi-1,\ndemi,\ndemi+1\}$, and we have:\\ 
\textbf{Case $1$.} If $q=0$ or $q=n$, the Riesz transform $T_q=D\Delta_q^{-1/2}$ on $q$-forms is bounded on $L^p$ if and only if 
\[\begin{array}{ll}
1<p<n , & \textrm{if }\,\,\,Ê\ker(H^1(\bbar{M},\pl \bbar{M})\to H^1(\bbar{M}))\not=0, \\
1<p<\infty, & \textrm{if }\,\,\, \ker(H^1(\bbar{M},\pl \bbar{M})\to H^1(\bbar{M}))=0.
\end{array}\]
\textbf{Case $2$.} If $q\notin \{0,n\}$, the Riesz transform $T_q=D\Delta_q^{-1/2}$ on $q$-forms is bounded on $L^p$ if and only if 
\[\begin{array}{ll}
 \frac{n}{n-2}<p<n , & \textrm{ if } \nu_{\ker}=\ndemi-1, \\
\frac{n}{n-1}<p<\infty, & \textrm{ if } \nu_{\ker}=\ndemi,\\
1<p<\infty, & \textrm{ if } \nu_{\ker}=\ndemi+1.
\end{array}\]
\end{corollary}

Although our geometric situation is quite restrictive in terms of the structure near infinity, 
there seem to be only very limited results about the Riesz transform for forms in the literature, and even Corollary \ref{euclidean} did not seem to be known (in fact, Theorem \ref{first} answers an open problem asked by Carron-Coulhon-Hassell \cite[Sec. 8]{CCH}).
There are a few previously known results: Bakry \cite{B} proved $L^p$ boundedness of $T_q$ on manifolds such that a curvature term appearing in the Weitzenbock formula is non-negative, Auscher-McIntosh-Russ \cite{AMcR} proved boundedness of Riesz transforms for forms on Hardy spaces for manifolds with volume measure satisfying the doubling property, while M\"uller-Peloso-Ricci \cite{MPR} obtained boundedness on $L^p$ for all $p$ in the case of the Heisenberg group.\\
  
\textbf{Conic degeneration and analytic torsion.}
We now apply Theorem \ref{introresthm} to investigate the behaviour of the analytic torsion under conic degeneration. The degeneration we discuss was originally proposed by Degeratu and Mazzeo as a means of analyzing elliptic operators on the more general class of iterated cone-edge spaces \cite{ma2}. The objective is to generalize theorems such as the Cheeger-M\"uller theorem to singular spaces by analyzing the behavior of the quantities involved in the smooth analogues as a family of smooth manifolds degenerates to a singular manifold. With this objective in mind, in \cite{s2}, the behaviour of the determinant of the Laplacian is investigated under conic degeneration. Here, we generalize this work to investigate the behaviour of the analytic torsion.

Let $(M,g)$ be a smooth asymptotically conic manifold with cross section $(N,h_0)$ which is exactly conic outside the compact manifold with boundary $\hat W:=\{x\geq 1\}$, which means that $g=dx^2/x^4+h_0/x^2$ in $\{x<1\}$.
As in the work of the second author \cite{s2}, we define a family of smooth compact manifolds $(\Omega_\epsilon,g_\eps)$ which degenerate to a  manifold $(\Omega_0,g_0)$ with an exact conic singularity as follows: assume that there is a compact set $K\subset \Omega_0$ where $g_0$ is smooth such that 
\[(\Omega_0\setminus K,g_0) \textrm{ is isometric to }((0,1)_r\x N, dr^2+r^2h_0).\]
For each $\eps>0$ small, let $M_\eps:=\{z\in M; x(z)\geq \eps\}$, and consider the manifold $\Omega_\eps=K\sqcup M_\eps$ obtained by gluing $K$ with $M_\eps$ along $\pl K\simeq \pl M_\eps \simeq N$. The obtained manifold is smooth and the metric  
$g_0$ on $K$ glues smoothly with the metric $\eps^2 g$ defined on $M_\eps$, giving a metric on $\Omega_\eps$ which we denote by $g_\eps$.
As $\epsilon$ goes to zero, $\Omega_\epsilon$ approaches $\Omega_0$ in the Gromov-Hausdorff sense; see \cite{s2} for more details concerning the geometry. 

We first need to make some assumptions on the cross-section $N$. Since analytic torsion is trivial in even dimensions, we assume that $n$ is odd. We then say that $N$ satisfies the \emph{modified Witt condition} if 
\begin{equation}
\label{Wittcond} 
\Big[0,\frac{3}{4}\Big]\cap {\rm Sp}_{\Lambda^{(n-1)/2}}(\Delta_{N})=\emptyset.
\end{equation}
Usually $M$ is said to be Witt if $0\notin{\rm Sp}_{\Lambda^{(n-1)/2}}(\Delta_N)$, so the modified Witt condition is slightly stronger. As we will see, the modified Witt condition rules out zero-resonances for $\Delta_q$ on $M$ for all $q\in[0,n]$, which allows us to apply Theorem \ref{introresthm} to obtain the microlocal description of the resolvent $R_q(k)$ on $M$ near $k=0$.

We can define a determinant of the Laplacian for any form degree on any compact manifold $\Omega$ by the method of Ray-Singer \cite{RS}, using a spectral zeta function for the Laplacians $\Delta_q^\Omega$ acting on $q$-forms:
\[ \zeta^\Omega_q(s):=\sum_{\la_j\in{\rm Sp}(\Delta_q), \la_j>0}\la_j^{-s} \,\, \textrm{ for }\Re(s)>n/2, \quad \log(\det(\Delta^\Omega_q)):=-\zeta_q'(0),\]
where $\zeta_q'(0)$ is obtained by meromorphic extension of $\zeta^\Omega_q(s)$ in $s\in \cc$. The analytic torsion is then defined by 
\[ \log T(\Omega):= \frac{1}{2}\sum_{q=0}^n(-1)^{q+1}q\log(\det\Delta^\Omega_q).\]

By the work of Cheeger \cite{ch2}, Dar \cite{Da} and Mooers \cite{mo}, the objects above are also well-defined on a compact manifold with conical singularities $\Omega_0$, under the Witt condition $0\notin{\rm Sp}_{\Lambda^{(n-1)/2}}(\Delta_N)$ (unless stated otherwise, we will always use the Friedrichs extension at the conic points when a choice of self-adjoint extension is necessary). The analytic torsion on manifolds with conic singularities has been the object of a considerable amount of recent study; see, for example, \cite{le,mv,ve} and the references therein.

We can also define analogous objects on $M$. Although $M$ is non-compact and its Laplacian on $q$-forms $\Delta_q^M$ 
has continuous spectrum in $\rr^+$, 
we can define a renormalized determinant of $\Delta^M_q$ and thus a renormalized analytic torsion, under the assumption that $N$ satisfies the modified Witt condition. To define the renormalized determinant, one uses a renormalized trace of the heat kernel defined as follows: 
\begin{equation}\label{renormtrace} {^R}\tra(e^{-t\Delta^M_q}):={\rm FP}_{\eps\to 0}\int_{\{x\geq\eps\}}\tra(H^M_q(t;z,z))dg(z)\end{equation}
where $H^M_q(t; z,z')$ is the heat kernel for $q$-forms on $M$ at time $t$ and ${\rm FP}$ means finite part in the sense of Hadamard.
The determinant of $\Delta_q^M$ is then defined as usual by $\log(\det \Delta_q^M ):=-\pl_s\zeta^M_q(0)$ through the zeta function
\[\zeta_q^{M}(s)=\frac{1}{\Gamma(s)}\int_0^{\infty}
({^R}\tra(e^{-t\Delta_q^M})-\dim \ker_{L^2}\Delta_q^M)t^{s-1}\ dt,\]
once we have shown that ${^R}\tra(e^{-t\Delta^M_q})$ has expansions in powers of $t$ and $\log(t)$ as $t\to 0$ and $t\to +\infty$.
 We may therefore define a renormalized analytic torsion on $M$ by
\[\log(T(M)):=\frac{1}{2}\sum_{q=0}^n(-1)^{q+1}q\log(\det\Delta_q^M).\]

It turns out, by the result of Ann\'e-Takahashi \cite{at}, that as $\eps\in(0,\eps_0]$ goes to $0$,
there are a finite number of small non-zero eigenvalues of $\Delta_q^{\Omega_\epsilon}$ which converge to $0$: for each $q$, if $N_q$ denotes the number of such eigenvalues, we have $N_0=0$ and  
\begin{equation}\label{numbersmall}
N_q=\dim\ker(\Delta_q^{\Omega_0})+\dim\ker_{L^2}(\Delta_q^M)-\dim\ker(\Delta_q^{\Omega_{\eps_0}}).
\end{equation}
This is a topological invariant, since all these kernel dimensions can be expressed in terms of 
cohomologies of $M$, $\Omega_0$ and $\Omega_{\eps_0}$.
Let the small eigenvalues themselves be $\mu^q_i(\epsilon)$, $i=1$ to $N_q$. 
Our theorem is the following, whose proof is given in Section \ref{section:proof}:

\begin{theorem}\label{conictorsion} 
Let $\Omega_{\epsilon}$ be defined as above and assume that $N$ satisfies the modified Witt condition. As $\epsilon\rightarrow 0$, for each $q$ between 0 and $n$,
\[\log\det\Delta^{\Omega_{\epsilon}}_q= -2\log\epsilon(\zeta_q^M(0))+\sum_{i=1}^{N_q}\log\mu^q_i(\epsilon)+\log (\det\Delta^{\Omega_{0}}_q)+\log (\det\Delta^{M}_q)+o(1),\]  and therefore
\[\log(T(\Omega_\epsilon))=(\sum_{q=0}^n(-1)^{q}q \zeta_q^M(0))\log\epsilon+\sum_{q=0}^{n}(-1)^qq(\sum_{i=1}^{N_q}\log\mu^q_i(\epsilon))\]\[+\log(T(\Omega_0))+\log(T(M))+o(1).\]
\end{theorem}

\begin{remark} An analogous theorem holds for the torsion defined with coefficients in a family of flat vector bundles, assuming that the family is constant in $\epsilon$ on the region where $\Omega_{\epsilon}$ is an exact cone (so that the gluing construction makes sense). To see this, we note (as in \cite[Eq. 2.3]{mv}) that a flat vector bundle $E$ over the cone $C_N$ may be written as a pullback of a flat vector bundle $F$ over the cross-section $N$. Therefore all calculations in local coordinates are exactly the same as in the trivial bundle case, as long as we consider the appropriate cross-section $(N,F)$ and its associated Gauss-Bonnet operator and Laplacian. In particular, although the result of Ann\'e-Takahashi in \cite{at} is not stated for twisted forms, their arguments (which use direct eigenform transplantation methods) work just as well. The conformal scaling property that we need also follows from this pullback observation. \end{remark}

\begin{remark} 
This theorem contains a non-explicit contribution from the small eigenvalues, but we show in Lemma \ref{conditions} that $N_q=0$ for all $q$ 
under the assumptions that $\la_q>1-|q-n/2|^2$, that $\nu_0\geq 1$ for all $q$ (with the notation of \eqref{defnu0intro} and above) and that the cohomology $H^q(N)$ is $0$ for $q\in[1,n-2]$. 
This happens, for instance, if $N=\sph^{n-1}$ is the canonical sphere and $n\geq 5$.
When $N_q=0$ for all $q$ and $\Omega_{\epsilon}$ has trivial cohomology in all degrees between 1 and $n-1$, since we know that $T(\Omega_\eps)$ is constant for $\eps>0$ by the result of Ray-Singer \cite{rs}, we deduce that for all $\eps\in(0,\eps_0]$
\[\log(T(\Omega_\epsilon))=\log(T(\Omega_0))+\log(T(M)).\]
\end{remark}

\textbf{Outline.} In Section \ref{section1}, we recall the relevant material from the b-calculus of Melrose and compute the indicial roots of the Hodge Laplacian on an asymptotically conic manifold. Section \ref{pdcle} contains a discussion of the pseudodifferential calculus of \cite{gh1,gh2}, which we use in Section \ref{resolvent-kernel} to explicitly construct the resolvent for the Hodge Laplacian at low energy. The applications to the Riesz transform are discussed in Section \ref{secRT}, the Sobolev estimates are proved in Section \ref{sec:Sobolev}, and the applications to analytic torsion are considered in Section \ref{section:proof}.\\

\textbf{Ackowledgements.} C.G. is partially supported by grants ANR-09-JCJC-0099-01 and ANR 10-BLAN 0105. D.S. is partially supported by the National Science Foundation grant NSF 1045119, and was supported by a CRM-ISM (Montr\'eal) fellowship in 2012-2013. We thank Andrew Hassell, Rafe Mazzeo, Adam Sikora, Pierre Albin and Alan Mc Intosh for helpful discussions.
 
\section{Laplacian on forms, indicial sets and $b$-pseudodifferential operators}\label{section1}

In this section, we will use the formalism of the paper \cite{gh1}, and we refer to reader in particular to Section 2 of that paper for details about the considered objects.
  
\subsection{Setup and functional spaces} We start by recalling some facts about $b$-structures and polyhomogeneity.\\

\textbf{b-structures.} Throughout, we will use the conformal metric 
 \[g_b = x^2g\] 
 which is an exact b-metric in the sense of Melrose \cite{me}, i.e. an asymptotically cylindrical metric on $M$.
 Associated to this b-structure, we can define an algebra of vector fields 
$\mc{V}_b(M)$ which is the set of smooth vector fields tangent to the boundary $\pl \bbar{M}$. Locally near $N$, if we let $y_1,\dots, y_{n-1}$ be local coordinates on $N$, the vector fieds in $\mc{V}_b(M)$ are linear combinations (over $C^\infty(M)$) of $x\pl_x, \pl_{y_1},\dots, \pl_{y_{n-1}}$. The enveloping algebra of $\mc{V}_b$ is
denoted $\textrm{Diff}_b(M)$; this is the space of smooth differential operators generated by compositions of elements in $\mc{V}_b(M)$ and multiplication by smooth functions. When the operators act linearly from sections of $E$ to sections of $F$, where $E$ and $F$ are smooth vector bundles over $M$, we use the notation $\textrm{Diff}_b(M; E,F)$; when $E=F$ we simply write $\textrm{Diff}_b(M; E)$. 

As in \cite{me}, there is a natural bundle associated to $g_b$, called the b-tangent bundle and denoted ${^b}TM$; the algebra $\mc{V}_b(M)$ may be viewed as the space of smooth 
sections of ${^b}TM$. We also let ${^bT}^*M$ denote the dual of ${^b}TM$, and let $\Lambda_b^q$ be the exterior $q$th power of ${^bT}^*M$. For later purposes, we also introduce some notation regarding half-densities; these are a bit inconvenient notationally but are useful for defining distributional kernels of operators in a more invariant way. The bundle of b-half densities, denoted $\Omega^{\demi}_b$, is the smooth line bundle trivialized by $|dg_b|^{\demi}$, where $dg_b$ is the volume form of $g_b$. See the book \cite{me} for more discussion of these and other b-structures, or the review of Grieser \cite{gri}.\\

\textbf{Scattering structures.} We can define similar ``scattering" objects associated to the original metric $g$. In particular, we define $\mc{V}_{\sca}(M)$ to be the algebra of smooth vector  fields which can be written locally near $N$ as linear combinations (over $C^\infty(M)$) of $x^2\pl_x,x\pl_{y_1},\dots,x\pl_{y_{n-1}}$. These vector fields correspond to vector fields of uniformly bounded length on $(M,g)$, and may again be viewed as sections of a bundle denoted ${^\sca T}M$. The dual bundle is ${^\sca T}^*M$, and the $q$-th exterior power of ${^\sca T}^*M$ is 
denoted $\Lambda_\sca^q$; this bundle is the natural setting for analyzing $q$-forms on $(M,g)$. As we did for $g_b$, we define the bundle of scattering half-densities $\Omega_\sca^{\demi}$ to be the trivial line bundle trivialized (over $C^\infty(M)$) by $|dg|^{\demi}$, where $dg$ is the volume form of $g$. In particular one 
has $x^{\ndemi}\Omega_\sca^{\demi}=\Omega_b^\demi$. Finally, we denote the bundle of smooth $q$-forms on $M$ by $\Lambda^q$.\\

\textbf{b-Sobolev spaces.} We now define the $j$-th Sobolev spaces on $q$-forms on $(M,g)$, but with respect to $b$-densities. First set $H_b^0:=L^2(M;\Lambda^q_{\sca}\otimes\Omega_b^\demi)$. Then write
 \[H_b^j:=\{ \omega \in H^0_b; \,\,\mc{L}_{X_1}\dots \mc{L}_{X_j}\omega\in H^0_b, \forall X_1,
 \dots,X_j\in \mc{V}_b(M)\} 
 \]
where $\mc{L}$ denotes the Lie derivative. For $\nu\in \rr$ we also define the space
\[x^{\nu-}H_b^0:=\bigcap_{\alpha<\nu}x^\alpha H_b^0.\]

\textbf{Polyhomogeneity and index sets.} We shall need the notions of index sets and polyhomogeneous conormal distributions on a manifold with corners, and we refer to \cite{me0} for details. An \emph{index set} $E$ is a discrete subset of $\rr\x \nn_0$ such that for each $m\in\RR$, the number of points $(\beta, j) \in E$ with $\beta \leq m$ is finite. We also adopt the convention that if $(\beta, j)\in E$, then $(\beta+1,j)\in E$ and (if $j>0$) $(\beta,j-1)\in E$.
Recall the operations of addition and extended union of two index sets $E$ and $F$, denoted $E+$ and $E\extunion F$ respectively:
\begin{equation}\begin{gathered}
E +F = \{ (\beta_1 + \beta_2, j_1 + j_2) \mid (\beta_1, j_1) \in E \text{ and } (\beta_2 , j_2) \in F \} \\
E \extunion F = E \cup F \cup \{ (\beta, j) \mid \exists (\beta, j_1) \in E, (\beta, j_2) \in F \text{ with } j = j_1 + j_2 + 1 \}.
\end{gathered}\end{equation} We say that $E\geq z_0$ (resp. $E> z_0$), with $z_0\in\rr$, if all $(z,k)\in E$ is such that $z\geq z_0$ (resp. $z>z_0$).

Now let $Z$ be a manifold with corners, with boundary hypersurfaces $H_1,\dots, H_\ell$.
Let $\rho_{H_i}$ be smooth boundary defining functions for each $H_i$; we write $\rho:=\prod_{i=1}^\ell \rho_{H_i}$, and call $\rho$ a total boundary defining function. An \emph{index family} $\mc{E}=(\mc{E}_{H_1},\dots, \mc{E}_{H_\ell})$ is a collection of index sets, one for each boundary hypersurface. The notation for index sets carries over to index families; for example, we say that $\mc{E}\geq z_0$ if $\mc{E}_{H_i}\geq z_0$ for all $H_i$. A distribution $u$ on  $Z$ is said to be \emph{polyhomogeneous conormal with index family} $\mc{E}$ if at each hypersurface $H_i$ there is an asymptotic expansion
\[u\sim\sum_{(s,j)\in E_{H_j}}a_{(s,j)}\rho_{H_j}^s(\log\rho_{H_j})^{j}\]
with $a_{(s,j)}\in C^{\infty}(Z)$, and with joint asymptotic expansions at each corner. See \cite{me} or \cite{gri} for further details.
 
\subsection{Laplacian on forms}\label{laponforms}
The usual exterior derivative $d$ is defined on $\Lambda^q(M)$. To view it as acting on half-density valued $q$-forms, we write $d(\omega\otimes |dg|^{\demi}):= (d\omega)\otimes |dg|^\demi$ if $\omega\in C^\infty(M;\Lambda^q)$. By writing it out in coordinates, we will see that $d$ is a scattering differential operator:
\[d\in \textrm{Diff}^1_{\sca}(M;\Lambda^q_{\sca}\otimes \Omega_\sca^{1/2},\Lambda^{q+1}_{\sca} \otimes \Omega_\sca^{1/2}).\]
To make explicit calculations, we decompose the bundle $\Lambda_{\sca}^q(M)$ near the boundary as a direct sum
\begin{equation}\label{decompos}
\begin{array}{rcl}
\Lambda^{q}(N)\oplus \Lambda^{q-1}(N)&\to &\Lambda^{q}_{\sca}(M) \\
(\omega_t,\omega_n)& \to & x^{-q}\omega_t+x^{-q-1}dx\wedge \omega_n.
\end{array}\end{equation}
Rather than write out the explicit form of $d$ itself, we write $d=xA=A'x$ for $b$-operators $A,A'\in\textrm{Diff}^1_b(M,\Lambda^q_{\sca}(M)\otimes\Omega_\sca^{1/2},\Lambda^{q+1}_{\sca}(M)\otimes \Omega_\sca^{1/2})$. Computing directly, the forms of $A$ and $A'$ in this decomposition (mapping $\Lambda^{q}(N)\oplus \Lambda^{q-1}(N)$ to 
$\Lambda^{q+1}(N)\oplus \Lambda^{q}(N)$) are
\[A=\left(\begin{array}{cc}
d_N & 0\\
x\pl_x -q & -d_N
\end{array}
\right), \quad A'=\left(\begin{array}{cc}
d_N & 0\\
x\pl_x -q-1 & -d_N
\end{array}
\right)
\] 
where $d_N$ is the exterior derivative on the boundary $N$.
Similarly, the adjoint $\delta$ with respect to the metric $g$ can be written $\delta=Bx=xB'$, and an easy computation shows that in the decomposition (\ref{decompos}), if we let $\delta_N$ be the adjoint of $d_N$ with respect to $h_0$, we have
\[\begin{gathered}B= \left(\begin{array}{cc}
\delta_N & -(x\pl_x-n+q-1)\\
0 & -\delta_N
\end{array}
\right)+ x^{n_0}\textrm{Diff}_b^1(M; \Lambda^q_\sca\otimes\Omega_{\sca}^{1/2},\Lambda^{q-1}_{\sca}\otimes\Omega_\sca^{1/2}),\\ B'= \left(\begin{array}{cc}
\delta_N & -(x\pl_x-n+q)\\
0 & -\delta_N
\end{array}
\right)+ x^{n_0}\textrm{Diff}_b^1(M; \Lambda^q_{\sca}\otimes\Omega_{\sca}^{1/2}, \Lambda^{q-1}_\sca\otimes\Omega_\sca^{1/2}).\end{gathered}\]

It will be convenient to consider $\Delta_q$ as acting on $\Lambda^q_{\sca}\otimes\Omega_b^{1/2}$, which amounts to conjugating all operators by $x^{n/2}$.
We now define the operator $P_b\in \textrm{Diff}_b^2(M;\Lambda_{\sca}^q \otimes\Omega_b^{1/2})$ by 
\begin{equation*}\label{pvspb}
P_b=AB+B'A', \quad \textrm{so that }\Delta_q=xP_bx.
\end{equation*} Let  $\Delta_N=\delta_Nd_N+d_N\delta_N$ be the Hodge Laplacian on the form bundle $\Lambda(N)=\oplus_{p=0}^{n-1}\Lambda^p(N)$. We compute (taking into account the conjugation by $x^{n/2}$) that in the decomposition \eqref{decompos},
\begin{equation} P_b=\left(\begin{array}{cc}
-(x\pl_x)^{2}+(\ndemi-q-1)^2+\Delta_{N} & 2d_N\\
2\delta_N & -(x\pl_x)^2+(\ndemi-q+1)^2+\Delta_N
\end{array}
\right)+ x^{n_0}W,
\end{equation}
where $W\in\textrm{Diff}_b^2(M;\Lambda_{\sca}^q\otimes\Omega_b^{1/2})$. Here $P_b$ acts on an element of $\Lambda_{\sca}^q \otimes\Omega_b^{1/2}$ by 
$P_b(\omega\otimes |dg_b|^\demi)=P_b(\omega)\otimes |dg_b|^\demi$. 
Since $\Delta_q$ is formally self-adjoint with respect to the scalar product induced by $g$, one easily checks that $P_b$ is also formally self-adjoint on $H_b^0$.

\subsection{The operator $P_b$ and its index set}

We will show, using the theory of Melrose \cite{me}, that $P_b$ is Fredholm and that it has a pseudodifferential inverse defined on its image. First we need to compute the \emph{indicial set} of $P_b$. The \emph{indicial family} in the sense of \cite{me} is the one-parameter family of operators  (with $\la\in\cc$)
\begin{equation}\label{indicialop}
I_\la(P_b)=\left(\begin{array}{cc}
-\la^{2}+(\ndemi-q-1)^2+\Delta_{N} & 2d_N\\
2\delta_N & -\la^2+(\ndemi-q+1)^2+\Delta_N
\end{array}
\right)\end{equation}
acting on $L^2(N; \Lambda^q(N)\oplus \Lambda^{q-1}(N))$, and the indicial set $\mc{I}(P_b)$ is the set of those $\la$
for which $I_\la(P_b)$ is not invertible.

To compute the indicial set, we use the Hodge decomposition of $\Delta_N$ on $q$-forms.  Specifically, we can write the $L^2$ spectrum of $\Delta_N$ as follows:
\[\begin{split}
\textrm{Sp}_{\Lambda^q}(\Delta_{N})=& \Big(\textrm{Sp}_{\Lambda^q}(\Delta_N)\cap \{0\}\Big) \cup \Big(\textrm{Sp}_{\Lambda^q}(\Delta_N|_{\textrm{Im} d_N})\Big)\cup \Big(\textrm{Sp}_{\Lambda^q}(\Delta_N|_{\textrm{Im}\delta_N})\Big)\\
=:\ &S^q_{0}\cup S^q_{d_N}\cup S^q_{\delta_N}
\end{split},
\] 
and note that $S^{q}_{d_N}=S^{q-1}_{\delta_N}$. Let $\mc{H}^q(N):=\ker_{\Lambda^q} \Delta_N$ be the space of harmonic $q$-forms on $N$, and let $F^q_{L,\alpha}:=\ker_{\Lambda^q}(\Delta_{N}-\alpha^2)\cap \Ima L$ on $q$-forms, where $L\in\{d_N,\delta_N\}$. 
Using this decomposition, we see that 
$I_{\la}(P_b)$ preserves the following subspaces of $L^2(N;\Lambda^q(N)\oplus \Lambda^{q-1}(N))$:
\begin{equation}\label{preserve}
\begin{gathered}
\Big(\mc{H}^q(N)\oplus \Ima\, \delta_N\Big)\oplus \{0\},\\
\{0\}\oplus \Big(\mc{H}^{q-1}(N)\oplus \Ima\, d_N\Big),\\ 
\Ima\, d_N \oplus  \Ima\, \delta_N.
\end{gathered}
\end{equation}

On the first two spaces of \eqref{preserve}, $I_\la(P_b)$ is diagonal and thus invertible if and only if $\la\notin (I_1\cup I_2)$ where
\begin{equation}\label{I1I2}
\begin{gathered}
I^1:=\left\{\pm\sqrt{(\ndemi-q-1)^2+\alpha^2}; \alpha^2\in S^q_0\cup S_{\delta_N}^q\right\},\\
I^2:=\left\{\pm\sqrt{(\ndemi-q+1)^2+\alpha^2}; \alpha^2\in S_0^{q-1}\cup S_{d_N}^{q-1}\right\}.
\end{gathered}
\end{equation}
For the third space, for each $\alpha$ with $\alpha^2\in S_{d_N}^q$, we use an orthonormal basis $(\phi_{\alpha,k})_{k=1,\dots,\dim F^{q-1}_{\delta_{N},\alpha}}$ for 
$F^{q-1}_{\delta_N,\alpha}$; then 
$(d_N\phi_{\alpha,k})_{k=1,\dots,\dim F^q_{\delta_N,\alpha}}$ is an orthogonal basis 
for $F^{q}_{d_N,\alpha}$. Using these bases, $d_N: F^{q-1}_{\delta_N,\alpha}\to F^{q}_{d_N, \alpha}$ is given by the identity matrix, while $\delta_N$ is the matrix $\alpha^2\textrm{Id}$. Therefore, $I_\la(P_b)$, acting on $\cc d_N\phi_{\alpha,k}\oplus \cc\phi_{\alpha,k}$, is of the form 
\[I_\la(P_b)=\left(\begin{array}{cc}
-\la^{2}+(\ndemi-q-1)^2+\alpha^2 & 2\\
2\alpha^2 & -\la^2+(\ndemi-q+1)^2+\alpha^2
\end{array}
\right).\]
An easy computation shows that the matrix can be diagonalized as
\[\left(\begin{array}{cc}
(\sqrt{(\ndemi-q)^2+\alpha^2}-1)^2-\la^2 & 0\\
0 & (\sqrt{(\ndemi-q)^2+\alpha^2}+1)^2-\la^2
\end{array}
\right)\]
in the basis 
\begin{equation}\begin{gathered}\label{mu+-}
\phi^-_{\alpha,k}:= \Big(d_N\phi_{\alpha,k},\Big((\ndemi-q)-\sqrt{(\ndemi-q)^2+\alpha^2}\Big)\phi_{\alpha,k}\Big)\\
\phi^+_{\alpha,k}:= \Big(d_N\phi_{\alpha,k},\Big((\ndemi-q)+\sqrt{(\ndemi-q)^2+\alpha^2}\Big)\phi_{\alpha,k}\Big)
\end{gathered}.
\end{equation}
The matrix of $I_\la(P_b)$ on the third space of \eqref{preserve} is thus invertible if and only if $\la\notin I^3\cup I^4$, where 
\begin{equation}\label{i3}
\begin{gathered}
I^3:=\left\{\pm(\sqrt{(\ndemi-q)^2+\alpha^2}-1); \alpha^2\in S^q_{d_N}\right\}\\
I^4:=\left\{\pm (\sqrt{(\ndemi-q)^2+\alpha^2}+1);\alpha^2\in S^q_{d_N}\right\}.
\end{gathered}
\end{equation}

The indicial set of $P_b$ (which we also call indicial set of $\Delta_q$) is thus  
\begin{equation}\label{indicial}
\mc{I}(P_b)=\mc{I}(\Delta_q)=I^1\cup I^2\cup I^3\cup I^4.
\end{equation}
We define the following subspaces of $L^2(N, \Lambda^q(N)\oplus\Lambda^{q-1}(N))$ associated to the index sets $I^1,\dots, I^4$:
\begin{equation}\label{E^j}
\begin{gathered}
E^1:=\Big(\ker \Delta_{N}\oplus \Ima \delta_N\Big)\oplus \{0\},\quad 
E^2:=\{0\}\oplus \Big(\ker \Delta_{N}\oplus \Ima d_N\Big),\\ 
E^3:=\bigoplus_{\alpha^2\in S^q_{d_N}}\bigoplus_{k=1}^{\dim F_{\delta_N,\alpha}}\cc\phi^-_{\alpha,k},\quad 
E^4:=\bigoplus_{\alpha^2\in S^q_{d_N}}\bigoplus_{k=1}^{\dim F_{\delta_N,\alpha}}\cc\phi^+_{\alpha,k}.
\end{gathered}
\end{equation}
We also define, for $\nu\in \mc{I}(P_b)$, the following vector subspaces of $L^2(N; \Lambda^q_N\oplus\Lambda^{q-1}_N)$ 
\[E_{\nu}:=\ker I_{\nu}(P_b), \quad E^j_{\nu}:=E_{\nu}\cap E^j\]
which yields the orthogonal decomposition
\begin{equation}\label{L2decomp}
L^2(N;\Lambda^{q}(N)\oplus\Lambda^{q-1}(N))=\bigoplus_{j=1}^4\bigoplus_{\nu\in I^j}E_{\nu}^j.
\end{equation}
We finally define 
\begin{equation}\label{defnu0}
\nu_0:=\min (\mc{I}(P_b)\cap [0,\infty)).
\end{equation}
Notice that the condition $\nu_0>0$ is equivalent to the condition \eqref{0notindroot} of the Introduction.
If $\la_p:=\min({\rm Sp}_{\Lambda^p}(\Delta_N|_{{\rm Im}d_N}))$,
$\mu_p:=\min({\rm Sp}_{\Lambda^p}(\Delta_N|_{({\rm Im}\delta_N)^\perp}))$, 
$\gamma_p:=\min({\rm Sp}_{\Lambda^p}(\Delta_N|_{({\rm Im}d_N)^\perp}))$
for $p=q-1,q,q+1$, then as long as $\la_q>1-|n/2-q|^2$, we get the formula \eqref{defnu0intro} from the introduction:
\begin{equation}\label{altnu0}
\nu_0=\min\Big(\sqrt{(\ndemi-q)^2+\la_q}-1,\, \sqrt{(\ndemi-q-1)^2+\gamma_q}, \,
\sqrt{(\ndemi-q+1)^2+\mu_{q-1}}
\, \Big).
\end{equation}

In the case where the boundary is the canonical sphere $(\pl\Mbar=\sph^{n-1},h_0=d\theta^2)$ (i.e. when  $M$  is  asymptotically Euclidean) we have $S^q_{d_N}=\{(q-1+j)(n-q-1+j);j\in\mathbb N\}$ for $1\leq q\leq(n-1)$; see for instance \cite{GM}. As a consequence, it is easy to compute that
\begin{equation}\label{caseSn}
\begin{gathered}
I^1=\left\{\begin{array}{ll}
\{\pm (j+\ndemi); j\in\nn_0\} & \textrm{ if }q\in [1,n-2]\\
\{\pm (j+\ndemi-1); j\in\nn_0\} & \textrm{ if }q=0\\
\{ \pm \ndemi\} & \textrm{ if }q=n-1\\
\emptyset & \textrm{ if }q=n
\end{array}\right.,\\
I^2=\left\{\begin{array}{ll} 
\{\pm (j+\ndemi); j\in\nn_0\} & \textrm{ if }q\in [2,n-1]\\ 
\{\pm(j+\ndemi-1); j\in\nn_0\} & \textrm{ if }q=n\\
\{\pm \ndemi\} & \textrm{ if }q=1\\
\emptyset & \textrm{ if }q=0
\end{array}\right.,\\
I^3=\left\{\begin{array}{ll} 
\{\pm (j+\ndemi-1); j\in\nn_0\} & \textrm{ if }q\in [1,n-1]\\ 
\emptyset & \textrm{ if }q=0,n
\end{array}\right.,\\
I^4= \left\{\begin{array}{ll} 
\{\pm (j+\ndemi+1); j\in\nn_0\} & \textrm{ if }q\in [1,n-1]\\ 
\emptyset & \textrm{ if }q=0,n
\end{array}\right.,
\end{gathered}\end{equation}
and $\nu_0=n/2-1$.

Now that we have computed the index set, various consequences follow immediately from the theory of Melrose \cite[Section 6.2]{me}. First we have the relative index theorem for 2nd order elliptic b-operators:
\begin{theorem}[Melrose's Relative Index Theorem, \cite{me}]\label{rit} The operator $P_b$ is Fredholm as a map from $x^\alpha H^j_b$  to $x^\alpha H^{j-2}_b$ for all $j \geq 2$ and all $\alpha \in \rr\setminus \mc{I}(P_b)$. The index of $P_b$ is equal to $0$ for $|\alpha| < \nu_0$ and the index increases by $\dim E_{\nu}$ as 
$\alpha$ crosses the value $\nu\in \mc{I}(P_b)$, with $\alpha$ decreasing.  
\end{theorem}
We also have a regularity result \cite{me}:
\begin{theorem}[Regularity of solutions to $P_b u = f$]\label{reg} Suppose that for $\beta\in\rr$, 
$f\in x^{\beta}H_b^0(M)$ is polyhomogeneous on $M$ with respect to the index set $E$, that $u \in x^\alpha H^0_b$, and that $P_b u = f$.
For $b \in \RR$ let $\mu(b,j)=\sharp \{b+k;k=0,\dots j\}\cap \mc{I}(P_b)$. 
Then $u\in x^{\alpha}H^2_b$ is polyhomogeneous with respect to the index set $E \extunion F$, where $F$ is the index set 
$$
\big\{ \big(\nu+j, k\big) \mid  j,k\in\nn_0, \ \nu\in \mc{I}(P_b),\nu\geq \alpha, \ 0 \leq k \leq \mu(\nu,j) - 1 \big\}.
$$
When $(N,h_0)=(S^{n-1},d\theta^2)$, this reduces to 
\[\big\{ \big(n/2+l, k\big) \mid  l\in\zz, \ n/2+l\geq \alpha, \ 0 \leq k \leq N_l - 1 \big\}.\]
where $N_l$ is the number of elements of the form $\pm(n/2-1+j), j\in\nn_0$ in the interval $(\alpha,n/2+l]$ .
\end{theorem}
A consequence of this theorem is that each $u \in x^{\alpha-}H_0^b$ satisfying the assumption of Theorem \ref{reg}
 has a full asymptotic expansion which starts with
\[u=\Big(\sum_{\substack{ \nu\in \mc{I}(P_b)\\ 
\ \alpha\leq \nu<\alpha+n_0}}\sum_{k=0}^{k_{\nu}}x^{\nu}(\log x)^{k}u_{\nu}(y)+\mathcal O(x^{\alpha+n_0})\Big)|dg_b|^\demi\]
for some $k_\nu\in\nn_0$ and some $u_\nu\in C^\infty(N;\Lambda^q(N)\oplus \Lambda^{q-1}(N))$. In fact, by assumption \eqref{assumhx}, we get
\begin{proposition}\label{technical} 
Assume $\nu_0>0$ and $\alpha\in \mc{I}(P_b)$, then if $P_{b}u=\mathcal O(x^{\alpha+n_0})$ with 
$u\in x^{\alpha-}H_b^2$ having asymptotic 
\[u=\Big(\sum_{\substack{ \nu\in \mc{I}(P_b)\\ 
\ \alpha\leq \nu<\alpha+n_0}}\sum_{k=0}^{k_{\nu}}x^{\nu}(\log x)^{k}u_{\nu}(y)+\mathcal O(x^{\alpha+n_0})\Big)|dg_b|^\demi,\]
then in fact $k_{\nu}=0$ (so there are no log terms) and $u_{\nu}\in E_{\nu}$ for all $\nu\in [\alpha,\alpha+n_0)$.
\end{proposition}
The proof of Proposition \ref{technical} is straightforward and proceeds by plugging the asymptotic expansion into the equation $P_bu=\mathcal O(x^{\alpha+n_0})$, then using that for $u_\nu \in C^\infty(N;\Lambda^q(N)\oplus \Lambda^{q-1}(N))$ and $\mu\in\rr$
\begin{equation}\label{pbxnulogk}
\begin{split}
P_b (x^{\mu}\log(x)^ku_\nu(y)|dg_b|^\demi)= & \Big(x^{\mu}\log(x)^k I_\mu(P_b)u_\nu(y)
- k\mu x^\mu\log(x)^{k-1}u_\nu(y)\\
& +k(k-1)\log(x)^{k-2} u_\nu(y)+\mc{O}(x^{\mu+n_0}\log(x)^k)\Big)|dg_b|^\demi.
\end{split}\end{equation}

It follows immediately from Theorems \ref{rit} and \ref{reg} that for $\nu\in \mc{I}(P_b)$, the vector space
$$
\{ \omega \in x^{\nu - \epsilon} H^j_b \mid P_b \omega = 0 \} \big/ \{  \omega \in x^{\nu + \epsilon} H^j_b \mid P_b \omega = 0 \} 
$$
for sufficiently small $\epsilon > 0$ is finite dimensional, independent of $\epsilon$, and  independent of $j$ by elliptic regularity. From Theorem~\ref{reg}, elements of this space have the form $x^{\nu} \phi_\nu(y) + \mc{O}(x^{\nu + \epsilon})$ where $\phi_\nu\in E_{\nu}$.  The span of such $\phi_{\nu}$ is a vector subspace of $E_{|\nu|}$ which we denote $G_{\nu}$.
Finally, from \cite[Chap. 6]{me}:
\begin{proposition}\label{prop:complementary} The subspaces $G_{\nu}$ and $G_{-\nu}$ of $E_{\nu}$ are orthogonal complements with respect to the inner product on $L^2(N;\Lambda^{q-1}(N)\oplus 
\Lambda^q(N))$ given by $h_0$. 
\end{proposition}

\subsection{$b$-pseudo-differential operators}

The $b$-double space $M^2_b$ is defined by blowing-up $\pl \bbar{M}\x\pl \bbar{M}$ inside $M\x M$, which we denote 
$[M\x M; \pl \bbar{M}\x\pl \bbar{M}]$; for details, see for instance \cite[Section 4.2]{me}. Let $\beta_b:M^2_b\to M^2$ be the blow-down map. The double space $M^2_b$ is a manifold with codimension-$2$ corners and three boundary hypersurfaces: 
the left boundary $\lb$, whose interior projects down to $\pl \bbar{M}\x M^\circ$ through $\beta_b$; 
the right boundary $\rb$, whose interior projects down to $M^\circ\x \pl \bbar{M}$; and the $b$-face $\bfa$, which projects down to $\pl \bbar{M}\x \pl \bbar{M}$ through $\beta_b$ and is diffeomorphic to $\pl \bbar{M}\x\pl \bbar{M}\x[-1,1]_{\tau}$ where $\tau:=(x-x')/(x+x')$. We denote by ${\rm diag}_b$ the closure of the 
lift through $\beta_b$ of the diagonal of $M^\circ\x M^\circ$.

Let $\mc{E}_{\lb},\mc{E}_{\rb}$ be index sets. The pseudo-differential operator class $\Psi_b^{m,\mc{E}_{\lb},
\mc{E}_{\rb}}(M)$ is the set of continuous linear operators $A$ mapping the space 
$\{f\in C^\infty(\bbar{M};\Omega_b^{\demi}); f=\mc{O}(x^\infty)\}$ 
to its dual, which have Schwartz kernels $\kappa_A=\kappa_A^1+\kappa_A^2 \in C^{-\infty}(M^2; \Omega_b^{\demi}(M^2))$ so that 

i) $\beta^*\kappa_A^1$ is smooth on 
$M^2_b\setminus {\rm diag}_b$, vanishes to infinite order at $\lb$ and $\rb$, and has a classical conormal singularity of order $m$ at ${\rm diag}_b$; 

ii) $\beta^*\kappa_A^2$ is polyhomogeneous on $M^2_b$ with index sets $\mc{E}_\lb$ at $\lb$, $\mc{E}_\rb$ at $\rb$, and $\nn_0$ at $\bfa$.

Operators acting on a smooth bundles are defined similarly as their Schwartz kernels can be considered as matrix valued distributions.

\subsection{Inverse for $P_b$}
By Proposition 5.64 in \cite{me}, we have: 
\begin{proposition}\label{Qb}
Assume $\nu_0>0$ and \eqref{assumhx} with $n_0\geq 3$. There exists $Q_b\in \Psi_b^{-2, \mc{E}_{{\rm lb}},\mc{E}_{{\rm rb}}}(M)$ such that $P_bQ_b={\rm Id}-\Pi_b$
where $\Pi_b$ is the orthogonal projection on the kernel of $P_b$ in $H_b^0$.
The index sets satisfy  
\[\mc{E}_{{\rm lb}}\subset \{ (\nu, k)\in \mc{I}(P_b)\x \nn_0; \nu\geq \nu_0\},  \quad 
\mc{E}_{{\rm rb}}\subset \{ (\nu, k)\in \mc{I}(P_b)\x \nn_0; \nu\geq \nu_0\}.\]
and if $m:=\min(\til{\nu}-\nu_0,2)$  with $\til{\nu}$ defined in \eqref{tilnu}, we have
\[\begin{gathered}
\mc{E}_{{\rm lb}}\cap [\nu_0,\nu_0+m)\x \nn_0\subset \mc{E}_{\rm lb}\cap [\nu_0,\nu_0+m)\x \{0\}, \\
\mc{E}_{{\rm rb}}\cap [\nu_0,\nu_0+m)\x \nn_0\subset \mc{E}_{\rm rb}\cap [\nu_0,\nu_0+m)\x \{0\},\\
\mc{E}_{{\rm lb}}\cap [\nu_0+m,\nu_0+m+2]\x \nn_0\subset \mc{E}_{\rm lb}\cap [\nu_0+m,\nu_0+m+2]\x \{0,1\}, \\
\mc{E}_{{\rm rb}}\cap [\nu_0+m,\nu_0+m+2]\x \nn_0\subset \mc{E}_{\rm rb}\cap [\nu_0+m,\nu_0+m+2]\x \{0,1\}.
\end{gathered}\]
\end{proposition}
In Proposition 5.64 in \cite{me}, the index set is larger than we claim here, but by writing that 
$P_bQ_b=Q_bP_b=-\Pi_b$ near $\rb$ and $\lb$ (since the kernel of ${\rm Id}$ is supported at the diagonal) and expanding the equation at $x=0$ and $x'=0$, the identity \eqref{pbxnulogk} gives the desired statement for $\mc{E}_\rb,\mc{E}_\lb$.

\section{Pseudo-differential calculus for low energy}\label{pdcle}

In this section we briefly recall the definition of the calculus $\Psi_k^{m,\mc{E}}(M)$ of pseudo-differential operators with parameter $k\in[0,k_0]$, as well as a few facts which are detailed in Section 2.2 of \cite{gh1}.

\subsection{The space $\MMksc$ and half-densities}\label{mmksc}
The Schwartz kernel of the resolvent $R_q(k)=(\Delta_q + k^2)^{-1}$ is a distribution on the space 
$M^\circ \times M^\circ \times (0, 1]$. The space $\MMksc$ is a manifold with codimension-$2$ corners, obtained 
by performing several blow-ups of $M^2\x[0,1]$, as explained in Section 2.2.1 of \cite{gh1} (we refer the reader to this section for a more detailed explanation). The reason for the blowups comes from the off-diagonal behaviour of the resolvent kernel. We denote the boundary hypersurfaces of $M^2\x[0,1]$  by $\zf = M^2 \times \{ 0 \}$, $\rb = M \times \partial M \times [0, k_0]$, and $\lb = \partial M \times M \times [0, 1]$. Then we blow-up\footnote{The reader not familiar with real blow-ups may consult Chapter 5 of \cite{me2}} the submanifold $(\partial M)^2 \times \{ 0 \}$, followed by the lift to this space of $(\partial M)^2 \times [0, k_0]$, $M \times \partial M \times \{ 0 \}$, $\partial M \times M \times \{ 0 \}$, to produce a space we call $M^2_{k, b}$. The new boundary hypersurfaces so created are denoted 
$\bfo$, $\bfa$, $\rbo$ and $\lbo$, respectively, while the old ones are still denoted $\zf,\lb$ and $\rb$. 
The blow-down map 
is denoted $\beta_{k,b}: M^2_b\to M^2\x [0,1]$ and the new boundary hyperfsurfaces satisfy
\[\begin{gathered}
\beta_{k,b}(\bfo)=\pl \bbar{M}\x\pl \bbar{M}\x\{0\},\quad  \beta_{k,b}(\bfa)=\pl \bbar{M}\x\pl \bbar{M}\x[0,1], \\
\beta_{k,b}(\rbo)= M\x\pl \bbar{M}\x\{0\}, \quad \beta_{k,b}(\lbo)= M\x\pl \bbar{M}\x\{0\}.
\end{gathered}\]
Finally, to produce the space $\MMksc$, we blow up the submanifold $\bfa\cap {\rm diag}_{k,b}$, where ${\rm diag}_{k,b}:=
\bbar{\beta^{-1}_{k,b}({\rm diag}(M^\circ\x M^\circ)\x (0,1)}$, to create a new boundary hypersurface $\sca$. See Figure \ref{blownupspace}. There is a natural blow-down map 
$\beta_{\sca,k}:\MMksc\to M^2\x [0,1]$ associated to these iterated blow-ups. 

\begin{figure}\label{blownupspace}
\centering
\includegraphics[scale=0.5]{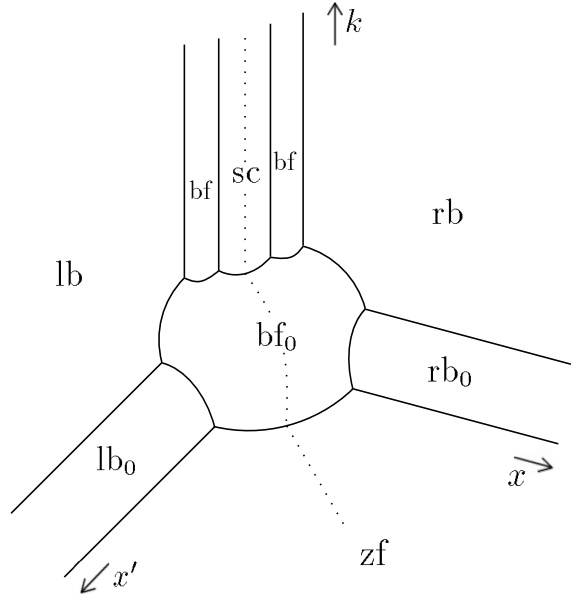}
\caption{The blown-up space $M^2_{k,sc}$.}
\end{figure}

The space $\MMksc$ has eight boundary hypersurfaces, each a geometric realization of a different asymptotic regime. The face $\zf$ may be identified naturally with $M^2_b$. The interior of the face $\bfo$ identifies with the product of exact cones $M_0^2$, where $M_0=(0,\infty)\x N$. Indeed, letting $(y,y')$ be coordinates on $N^2$, we can use the coordinates $k,\kappa=k/x,\kappa'=k/x',y,y'$ near the interior of $\bfo$. In these coordinates, $\bfo$ is defined by $k=0$, so writing $(\kappa,y,\kappa',y')\in ((0,\infty)\x N)^2$ provides the diffeomorphism between the interior of $\bfo$ and $M_0^2$. The interior of the face $\sca$ is a bundle over $\bfa\cap {\rm diag}_{k,b} \simeq N$, with fibers $\rr^n$.\\

The bundle of $b$-half-densities $\Omega_b^{\demi}(M^2_{k,b})$ on $M^2_{k,b}$ is defined to be the smooth line bundle trivialized by $\rho^{-\demi}\mu$,
where $\mu$ is any non-vanishing smooth section of the bundle of half-densities and $\rho$ is a global boundary defining 
function on $M^2_{k,b}$ (i.e. a product of boundary functions for the boundary hypersurfaces). In particular, this bundle is trivialized 
by $|dg_b(z)dg_b(z')\frac{dk}{k}|^\demi$. Let $\tilde{\Omega}_b^\demi(\MMksc)$ be the lift of $\Omega_b^{\demi}(M^2_{k,b})$ to $\MMksc$. Except at $\sca$, this bundle restricts canonically to each hypersurface, and gives the bundle of $b$-half-densities on the hypersurface (which is also a manifold with corners).

Denote by $\pi_k: M^2\x [0,1]\to M^2$ the projection off the $k$ variable. 
The bundle $\Lambda^q_\sca\otimes (\Lambda^q_\sca)^*$ on $M^2$  
pulls-back through the map $\beta_{\sca,k}\circ \pi_k= \MMksc\to M^2$ to a smooth bundle denoted 
\[ F_q:=(\beta_{k,\sca}\circ \pi_k)^*(\Lambda^q_\sca\otimes (\Lambda^q_\sca)^*).\]
The operators we shall consider have Schwartz kernels which pull back to $\MMksc$ as 
distributional sections of $F_q\otimes \tilde{\Omega}_b^\demi(\MMksc)$.

\subsection{Operator calculus}\label{opecal}
We will show that the Schwartz kernel of $R_q(k)$ is a polyhomogeneous conormal distribution on the space $\MMksc$, with 
a classical conormal singularity at the spatial diagonal. The space of operators 
 $\Psi_k^{m,(a_{\bfo},a_{\zf},a_{\sca}), \mc{A}}(M)$ acting on scattering $q$-forms (with $b$-half-density values) is a space of pseudo-differential operators depending parametrically on $k$
and with Schwartz kernels given by conormal polyhomogeneous distributions on $\MMksc$; 
it is introduced in \cite[Def 2.8]{gh1} for functions, but the definition for operators acting on bundles is identical.
The index $m\in\rr$ corresponds to the order of the conormal singularity at the diagonal (i.e. the usual order for pseudo-differential operators). The set $(a_{\bfo},a_{\zf},a_{\sca})\in \rr^3$ gives the behaviour of the part of the kernel which is singular at the diagonal at the faces $\bfo,\zf,\sca$. And $\mc{A}$ is an index set
\[\mc{A}=(\mc{A}_{\bfo},\mc{A}_{\zf},\mc{A}_{\sca},\mc{A}_{\rbo},\mc{A}_{\lbo})\] 
which corresponds to the polyhomogeneous expansion of the part of the kernel which is smooth in the interior at the respective faces. In addition, the kernels of operators in this class vanish to infinite order at $\lb,\rb$, and $\bfa$.
More precisely, $A\in \Psi_k^{m,(a_{\bfo},a_{\zf},a_{\sca}), \mc{A}}(M)$ if its Schwartz kernel pulls-back to $\MMksc$ to a sum $\kappa_1(A)+\kappa_2(A)$, where $\rho_{\bfo}^{-a_{\bfo}}\rho_{\zf}^{-a_\zf}\rho_{\sca}^{-a_{\sca}}\kappa_1(A)$ is a distributional section of $F_q\otimes \tilde{\Omega}_b^\demi(\MMksc)$ supported in a neighbourhood of ${\rm diag}_{k,\sca}:=
\bbar{\beta^{-1}_{k,\sca}({\rm diag}(M^\circ\x M^\circ)\x (0,k_0)}$ and conormal to this submanifold uniformly (and smoothly) up to the boundary, and $\kappa_2(A)$ is smooth in the interior of $\MMksc$ and polyhomogeneous conormal 
with index set $\mc{A}$ (and vanishing to infinite order at $\lb,\rb$, and $\bfa$). We also  ask that $a_{\zf}\in \mc{A}_\zf$, $a_{\bfo}\subset \mc{A}_{\bfo}$ and $a_{\sca}\subset \mc{A}_{\sca}$.\\ 

\textbf{Composition.} This calculus forms an algebra, as there is a composition law: given index families $\mc{A}$ and $\mc{B}$ as above and operators
\[A\in \Psi_k^{m,(a_{\bfo},a_{\zf},a_{\sca}),\mc{A}}(M),\,\, 
B\in \Psi_k^{m',(b_{\bfo},b_{\zf},b_{\sca}),\mc{B}}(M )\]
then $C:=A\circ B\in \Psi_k^{m+m',(a_{\bfo}+b_{\bfo},a_{\zf}+b_{\zf},a_{\sca}+b_{\sca}),\mc{C}}(M)$ with $\mc{C}$ given as in \cite[Prop 2.10]{gh1} by
\begin{equation}\begin{aligned}
\mc{C}_{\sca} &= \mc{A}_{\sca} + \mc{B}_{\sca},\\
\mc{C}_{\zf} &= \big( \mc{A}_{\zf} + \mc{B}_{\zf} \big) \extunion \big(  \mc{A}_{\rbo} + \mc{B}_{\lbo} \big), \\
\mc{C}_{\bfo} &= \big( \mc{A}_{\lbo} + \mc{B}_{\rbo} \big) \extunion \big(  \mc{A}_{\bfo} + \mc{B}_{\bfo} \big), \\
\mc{C}_{\lbo} &= \big( \mc{A}_{\lbo} + \mc{B}_{\zf} \big) \extunion \big(  \mc{A}_{\bfo} + \mc{B}_{\lbo} \big), \\
\mc{C}_{\rbo} &= \big( \mc{A}_{\zf} + \mc{B}_{\rbo} \big) \extunion \big(  \mc{A}_{\rbo} + \mc{B}_{\bfo} \big), \\
\mc{C}_{\bfc} &= \mc{C}_{\lb} = \mc{C}_{\rb} = \emptyset.
\end{aligned}\label{comp-if}\end{equation}
The proof is done in \cite{gh1} for operators acting on $0$-forms, but applies as well for $q$-forms. Here the composition means that the $k$-dependent operator $A\otimes|\frac{dk}{k}|^{-\demi}$ is composed on the right with $B\otimes |\frac{dk}{k}|^{-\demi}$ and then tensored with $|\frac{dk}{k}|^\demi$ again so that its lifted kernel is a section of $\tilde{\Omega}_b^\demi$ on $\MMksc$.
The operator ${\rm Id}\otimes|\frac{dk}{k}|^{\demi}$ is in the calculus (if ${\rm Id}$ is the identity 
operator on $b$-half-densities on $M$), and we will simply denote this operator ${\rm Id}$  to avoid writing the extra $|dk/k|^\demi$ factor everywhere in the paper. \\

\textbf{Inverse of ${\rm Id}-E(k)$.} If $E(k) \in \Psi_k^{m,(e_{\bfo},e_{\zf},e_{\sca}), \mc{E}}(M)$, with $m < 0$, $e_{\rm f}>0$ and $\mc{E}_{\textrm{f}} > 0$ for $\textrm{f}\in
 \{ \zf, \bfo, \sca\}$, and if $\mc{E}_{\rbo}+\mc{E}_{\lbo}>0$, then by the composition law, we have that for large enough $N$, $E(k)^N$ is Hilbert-Schmidt with $\| E^N(k) \|_{HS} \to 0$ as $k \to 0$. In particular, the operator $\Id - E(k)^N$ is invertible for $k$ small enough, and the Neumann series $\sum_{j=0}^\infty E(k)^{Nj}$ for the inverse converges in operator norm (see \cite[Cor. 2.11]{gh1}) providing the following right inverse for $\Id - E(k)$:
\begin{equation}\label{invde1-E}
\sum_{j=0}^\infty E(k)^{N-1}({\rm Id}-E(k)^{Nj})^{-1}.
\end{equation}
This inverse lies in the calculus $\Psi_k^{*}(M)$ by the composition law \eqref{comp-if}.\\
 
\textbf{Normal operators.} If the index family $\mc{A}$ is nonnegative, then the restriction of the Schwartz kernel of 
$A \in \Psi_k^{m, (a_{\bfo},a_{\zf},a_{\sca}),\mc{A}}(M; \Lambda^q_\sca)$ to any of the faces $\zf, \bfo, \sca$ is a well-defined distribution, called the normal operator at $\textrm{f}$ and denoted $I_{\textrm{f}}(A)$, for $\textrm{f} = \zf, \bfo, \sca$. 
The distribution $I_{\zf}(A)$ corresponds to the Schwartz kernel of a b-pseudodifferential operator of order $m$ acting on $\Lambda_\sca^q$.
Using the decomposition \eqref{decompos} as $\Lambda^q(N)\oplus \Lambda^{q-1}(N)$ 
of the bundle $\Lambda^q_\sca$ near $N=\pl \bbar{M}$,  the distribution $I_{\bfo}(A)$ corresponds to the kernel of a pseudodifferential operator acting on $(\Lambda^q(N)\oplus \Lambda^{q-1}(N))\otimes \Omega_b^\demi (M_0)$, 
where $\Omega^\demi_b(M_0)$ is the bundle of $b$ half-densities on the compactification 
$\bbar{M}_0=[0,\infty]\x N$ of $M_0$.
Finally the kernel $I_{\sca}(A)$ is a family, parametrized by $N \times (0, k_0]$,  of convolution pseudodifferential operators acting on half-densities on scattering forms on $\RR^n$. 

The normal operators respect composition:
 $$
 I_{\textrm{f}}(A) \circ I_{\textrm{f}}(B) = I_{\textrm{f}}(A \circ B)
 $$
  provided that 
  \begin{equation*}
\mc{A}_{\rbo} + \mc{B}_{\lbo} > 0 \text{ and } \mc{A}_{\lbo} + \mc{B}_{\rbo} > 0.
\end{equation*}
When $A$ is differential, the Schwartz kernel is supported on the diagonal $\Delta_{k,\sca}$, and we can identify
$I_{\zf}(A)$, $I_{\bfo}(A)$ and $I_{\sca}(A)$ with differential operators on bundles over $M$, $M_0$, and $\rr^n$ respectively.
A b-differential operator $P$ of order $m$ acting in the left variable on $\MMksc$ is a sum of compositions
of at most $m$ vector fields of $\mc{V}_b$ which are tangent to $\rbo$ (resp. to $\lbo$),
and thus restricts smoothly to $\rbo$ (resp. to $\lbo$) as a differential operator denoted $I_{\rbo}(P)$ (resp. $I_{\lbo}(P)$).

\section{Resolvent kernel}\label{resolvent-kernel}

Our strategy to construct the resolvent kernel near $k=0$ follows the method of \cite{gh1} and more particularly \cite{gh2}
when there are $L^2$ forms in $\ker \Delta_q$. Once the indicial operator and index set of $P_b$  are obtained in \eqref{indicialop} and \eqref{indicial}, there are only minor changes to make in the construction of \cite{gh2} (which was for functions) to extend it to forms.  Therefore, we will not repeat all the details used in \cite{gh2} but simply explain the main steps and changes. Since we care about applications to the Riesz transform for forms, we shall need
to construct a precise parametrix with several terms in the asymptotic expansion of the resolvent kernel 
at the face $\rbo$. This would be unnecessary for simply showing the polyhomogeneity of $R(k)$.\\

We will construct a parametrix $G(k) \in \Psi_k^{-2; (-2,0,0), \mc{G}}(M)$ 
that solves 
$$
(\Delta_q + k^2) G(k) = \Id - E(k),
$$ where $E(k)$ is an error term to be specified later. Throughout, we use $k$ as a boundary defining function for the interior of the boundary faces $\zf, \bfo, \lbo, \rbo$ of $\MMksc$. On the left factor of $M^2$, $z \in M$ is written $z = (x,y)$ close to $N=\partial \bbar{M}$, where $y$ are local coordinates on $N$, and primes indicate the same coordinates on the right factor. 
We also denote 
\[\kappa:=k/x, \quad \kappa':=k/x', \quad s=x/x'.\]
We use the coordinates $(z,z')$ on $\zf$, $(\kappa, \kappa', y, y')$ for $\bfo$, $(z, y', \kappa')$ for $\rbo$ and $(z', y, \kappa)$ for $\lbo$. 
Using these coordinates we will write the polyhomogeneous expansion of $G(k)$ at (the interior of) the face $\textrm{f}$, for $\textrm{f} = \zf, \bfo, \rbo, \lbo$, in the form
$$
G(k)\sim \sum_{(j,p) \in J,\,\, j\leq j_0} k^{j}(\log k)^pG^j_{\textrm{f}}+o(k^{j_0})
$$
where $J$ is some index set. We call $G^j_{\textrm{f}}$ the model at order $j$ at the face $\textrm{f}$. 
At the other boundary hypersurfaces of $\MMksc$, elements of the calculus will be rapidly decreasing (except at $\sca$, where there will be a smooth expansion).

We construct $G(k)$ by setting a finite number of models at each boundary hypersurface, together with the singularity at the diagonal, with the property that the models match at adjacent faces so that there exists a polyhomogeneous 
distribution which has polyhomogeneous expansion at each face corresponding to our models. 
The parametrix $G(k)$ can then be taken to be this distribution. If the models are constructed to solve the appropriate 
model problems at each face, then $(\Delta_q+k^2)G(k)$ will be the Schwartz kernel of the Identity plus a polyhomogeneous distribution $E(k)$ which vanishes to high order at the faces corresponding to $k=0$ (i. e. $\bfo$, $\zf$, $\lbo$, $\rbo$).\\   

Throughout, we will make the assumption that $\Delta_q$ has no zero-resonances; recall that this means $\ker_{x^{-1}L^2(M)}(\Delta_q)=\ker_{L^2(M)}(\Delta_q)$. This assumption is equivalent to the statement that (as an operator on $b$-half density-valued forms) \begin{equation}\label{noresPb}
\ker_{H_b^0}(P_b)=\ker_{xH_b^0}(P_b). 
\end{equation}
and by using Theorem \ref{reg} giving us the conormal regularity of solutions of $P_bu=0$, we deduce easily that the no zero resonance condition is equivalent to \eqref{kerLr}.
We define the kernel exponent
\begin{equation}\label{tilnu} 
\til{\nu}:=\max\{\nu \in \mc{I}(P_b); \ker_{H_b^0}(\Delta_q)\subset x^{\nu-1-}H_b^0\} ,
\end{equation}
which tells us which weighted space $\ker_{H_b^0}(P_b)$ lies in; notice that our assumption \eqref{noresPb} implies that $\til{\nu}>1$. By Theorem \ref{reg}, $\til{\nu}$ can also be defined as 
\[\til{\nu}=\max\{\nu \in \mc{I}(P_b); \ker_{L^2}(\Delta_q)\subset x^{n/2+\nu-1}L^\infty(M,\Lambda_{\sca}^q)\}.\]
where we consider $\ker_{L^2}(\Delta_q)$ as pure forms and not density valued forms, and the $L^p$ spaces are with respect to the measure $dg$.
From Proposition \ref{prop:complementary}, we have
\begin{equation}\label{Gnu=0}
G_\nu=0 \textrm{ and }G_{-\nu}=E_\nu \textrm{ for all } \nu\in \mc{I}(P_b)\cap (0,\til{\nu}).
\end{equation} Finally, we define 
\begin{equation}\label{defm}
m:= \min(\til{\nu}-\nu_0,2).
\end{equation}
In the case where $\ker_{H_b^0}(\Delta_q)=0$, the construction below works as well and we set $m=2$.

\subsection{Singularity at the diagonal} 
As in \cite{gh1}, this is standard and corresponds to the usual parametrix construction for elliptic operators on compact manifolds. In particular, $\Delta_q+k^2$ 
is elliptic in the sense that its symbol times $\rho_{\bfo}^{-2}$ is elliptic uniformly on $N^*{\rm diag}_{k,\sca}$. Therefore, there exists an operator $Q\in \Psi_k^{-2,(-2,0,0)}(M)$ such that $(\Delta_q+k^2)Q-{\rm Id}\in\Psi_k^{-\infty,(0,0,0)}(M)$. The full symbol of $Q$ at the diagonal
${\rm diag}_{k,\sca}$ is uniquely determined modulo symbols of order $-\infty$ by ellipticity.

\subsection{Leading term at $\sca$}
As explained in \cite[Sec. 4.2]{gh1}, $\sca$ is a bundle with Euclidean fibers, and the normal operator $I_{\sca}(\Delta_q+k^2)$ corresponds to a Euclidean Laplacian  $\Delta^{\rr^n}_q+k^2$ 
on $q$-forms. This Laplacian has an inverse for any $k>0$, and we set $G_{\sca}^{0}$ to be that inverse. Moreover, by scaling in $k$, it is immediate that $G_{\sca}^0$ is polyhomegenous on $\sca$, with index set $-2$ at $\bfo\cap\sca$ and infinite-order decay at $\sca\cap\bfa$.

\subsection{Leading term at $\bfo$}

We follow the description in \cite[Sec. 3.4]{gh1}. 
The interior of the face $\bfo$ may be identified with the product of two exact cones $M_0^2$, where the first cone is $(0,\infty)_{\kappa}\x N_y$ and the second cone is the same with primed coordinates. Therefore, we may use $(\kappa,\kappa',y,y')$ as coordinates. The operator $(\Delta_q+k^2)$ vanishes at order $2$ at $\bfo$, and
$k^{-2}\Delta_q+1$ has indical operator  (acting on b-half-densities)
$I_{\bfo}(k^{-2}\Delta_q+1)=\kappa^{-1}P_{\bfo}\kappa^{-1}$ at $\bfo$. Here $P_{\bfo}$ 
is a differential operator on $M_0$ on sections of $(\Lambda^q(N)\oplus \Lambda^{q-1}(N))\otimes \Omega_b^{\demi}(M_0)$ 
(using the decomposition \eqref{decompos} of $\Lambda_{\sca}^q$), which can be written in matrix form as follows:
\begin{equation}
\label{Pbfo}
P_{\bfo}=
\left(\begin{array}{cc}
-(\kappa\pl_\kappa)^{2}+(\ndemi-q-1)^2+\Delta_{N}+\kappa^2 & 2d_N\\
2\delta_N & -(\kappa\pl_\kappa)^2+(\ndemi-q+1)^2+\Delta_N+\kappa^2
\end{array}
\right),\end{equation} 
where $\Omega^\demi_b(M_0)$ are $b$-half densities on $\bbar{M}_0=[0,\infty]\x N$.
Using the spectral decomposition \eqref{E^j}, we can diagonalize in the same way we did for $P_b$. We find that $P_{\bfo}$ has an inverse $Q_{\bfo}$ in terms of this decomposition:
\begin{equation}\label{Qbfo}
\begin{split}
 Q_{\bfo}:=\sum_{\nu\in \mc{I}(P_b),\nu>0}\Pi_{E_\nu}\Big(I_{\nu}(\kappa)K_{\nu}(\kappa')H(\kappa'-\kappa)+I_{\nu}(\kappa')K_{\nu}(\kappa)H(\kappa-\kappa')\Big) \left|\frac{d\kappa dyd\kappa'dy'}{\kappa\kappa'}\right|^\demi.
\end{split}\end{equation}
Here and below, $\Pi_{E_\nu}=\Pi_{E_\nu}(y,y')$ means the Schwartz kernel of the 
orthogonal projection on $E_\nu\subset L^2(N;\Lambda^q_N\oplus\Lambda^{q-1}_N)$, 
and $I_{\nu}(z),K_{\nu}(z)$ are the modified Bessel functions.
As in \cite[Sec. 3.4]{gh1}, we set
\begin{equation}
G_{\bfo}^{-2} = (\kappa \kappa') Q_{\bfo} .
\label{Gbfo}\end{equation}
The consistency of this model with $G_{\sca}^0$ follows exactly as in \cite[Sec. 3.5]{gh1}. The model $G_{\bfo}^{-2}$ also vanishes
to infinite order at $\rb,\lb$ and $\bfa$; this follows from the exponential 
decay of Bessel functions $K_{\nu}(\kappa)$ as $\kappa\to \infty$.

\subsection{Terms at zf}\label{subsect:zf}
We follow closely the approach of Section 4 in \cite{gh2}. By Proposition \ref{Qb} and the fact that $0\notin\mc I(P_b)$, the operator $P_b$ has a generalized inverse $Q_b\in \Psi_b^{-2, \mc{E}_\lb,\mc{E}_\rb}(M)$ such that 
$$P_b Q_b = Q_b P_b = \Id - \Pi_b, \quad \textrm{ with }\, \Pi_b=\sum_{j=1}^J\varphi_j\otimes\varphi_j,$$
where $\{\varphi_j\}$ is a basis of the $L^2$ real orthonormalized (half-density) eigenforms of $P_b$ with eigenvalue $0$.
When $\zf$ is viewed as $M^2_b$, the intersection $\zf\cap \bfo$ is identified with the b-face $\bf$ of $M_b^2$, which in turn may be identified with $(0,\infty)_s\x N\x N$, where $s:=x/x'$. Using these coordinates, one has, as in \cite[Sec 4.6]{gh1}, that the kernel of $Q_b$ restricted to $\zf\cap\bfo$ is given by
\begin{equation}\label{Qbatbf} 
Q_b|_{\zf\cap\bfo}=\sum_{\nu\in \mc{I}(P_b),\nu>0}\Pi_{E_\nu}(y,y')\frac{e^{-\nu|\log s|}}{2\nu}|dydy'ds/s|^\demi.
\end{equation}
Since $\Delta_q=xP_bx$, we obtain
\[ \Delta_q(x^{-1}Q_{b}x^{-1})={\rm Id}-\sum_{j=1}^{J}(x\varphi_{j})\otimes(x^{-1}\varphi_{j}).\]

In order to solve the model problem at $\zf$, we need to understand the asymptotics of the $\varphi_j$. By the absence of zero-resonances (see \eqref{noresPb}), all elements $\varphi_j$ are such that $x^{-1}\varphi_j\in \ker_{H_b^0}(\Delta_q)$. Using Proposition \ref{technical} and the definition of $\tilde\nu$, we have 
\begin{equation}\label{varphij}
\varphi_j=\Big(\sum_{\substack{\nu \in \mc{I}(P_b)\\ 
\til{\nu}\leq\nu \leq \til{\nu}+2}}x^{\nu}a^j_{\nu}+\mathcal O(x^{\til{\nu}+2+\epsilon})\Big)|dg_b|^\demi
\end{equation}
with $a^j_\nu\in E_\nu$. Let $(\psi_i)_{i=1,\dots,J}$ be an orthonormal basis of $\ker_{H_b^0}\Delta_q$; then
\[\psi_i=\sum_{j=1}^J\alpha_{ij}x^{-1}\varphi_j\]
for some $\alpha_{ij}\in\rr$.
Now, if $\Pi_{\ker \Delta_q}$ is the orthogonal projection on $\ker_{H_b^0}\Delta_q$, we write 
\[x\varphi_{j}=\Pi_{\ker \Delta_q}(x\varphi_{j})+\psi_{j}^{\perp}.\]
and as in \cite[eq 3.10, 3.11]{gh2}, we have that
\[\begin{gathered}
\Pi_{\ker \Delta_q}(x\varphi_{j})=\sum_{k=1}^{J}\alpha_{kj}\psi_{k} \quad \textrm{and }
\sum_{j=1}^{J}(\Pi_{\ker \Delta_q}(x\varphi_{j}))\otimes x^{-1}\varphi_{j}=\sum_{j=1}^{J}\psi_{j}\otimes\psi_{j}.
\end{gathered}\]
This tells us that $\psi_{j}^{\perp}\in x^{\til{\nu}-1-}H_b^0$, so there exists $\eps_0\in(0,1)$ such that
$x^{-1}\psi_{j}^{\perp}\in x^{\til{\nu}-2-}H^0_{b}\subset x^{-1+\eps}H_b^0$ for all $\eps\in(0,\eps_0)$. By Theorem \ref{rit} and the fact that the indicial set is discrete, we may fix an $\eps\in(0,\eps_0)$ so that the operator $P_{b}$ is Fredholm on $x^{-1+\eps}H_b^0$. Now the null space of $P_{b}$ on $x^{1-\eps}H^0_b$ is the same as the null space of $P_{b}$ on $H^0_b$ by our assumption \eqref{noresPb},
thus it is spanned by the $\varphi_{j}$. But $x^{-1}\varphi_{j}$ is a linear combination of the $\psi_{k}$,
so $x^{-1}\varphi_{j}$ is orthogonal to $\psi_{k}^{\perp}$, so $\varphi_{j}$ is orthogonal to $x^{-1}\psi_{k}^{\perp}$ for
any $j$, $k$. Hence $x^{-1}\psi_{j}^{\perp}\in x^{-1+\eps}H^0_b$ is
orthogonal to the null space of $P_b$ on $x^{1-\eps}H^0_b$, and thus by Proposition \ref{prop:complementary} 
it is orthogonal to the kernel of the adjoint of $P_b$ on $x^{-1+\eps}H^0_b$.
This means that it is in the range of $P_b$ in $x^{-1+\eps}H_b^0$. 
So there exists $\chi_{j}\in x^{-1+\eps}H^0_{b}(M)$
such that
\[P_b\chi_{j}=x^{-1}\psi_{j}^{\perp}.\]
The forms $\chi_{j}$ are not necessarily unique, but since $\nu_0>0$ and $G_{-\nu}=E_\nu$ for $\nu\in(0,\til{\nu})$ by \eqref{Gnu=0}, we can always add elements in $\ker_{x^{-1+\eps}H_b^0}P_b$ to $\chi_j$ to ensure that
$\chi_j\in x^{\til{\nu}-2-}H_b^0$. 

Using the same analysis as in \cite[Sec. 4]{gh2}, first note that $b_\nu^j\in E_\nu$, defined by 
\[b_{\nu}^{j}(y)=\sum_{k=1}^{J}\sum_{l=1}^{J}\alpha_{kj}\alpha_{kl}a^{l}_{\nu}(y),\]
provide the leading asymptotics of $\psi_j^\perp$:
\[x^{-1}\psi_j^\perp=\Big(\sum_{\substack{\nu \in \mc{I}(P_b)\\ 
\til{\nu}\leq\nu \leq\til{\nu}+2}}-b_\nu^j(y) x^{\nu-2}+a_{\til{\nu}}^j(y)x^{\til{\nu}}+\mc{O}(x^{\til{\nu}+\eps})\Big)|dg_b|^\demi.\]
Then, as in \cite{gh2}, we have
\begin{equation}\label{aschij}
x^{-1}\chi_j=\Big(\sum_{\substack{\nu \in \mc{I}(P_b)\\ 
\til{\nu}\leq\nu \leq \til{\nu}+2}}x^{\nu-3}\big(\frac{b_{\nu}^{j}(y)}{4(1-\nu)}+\beta^{j}_{\nu-2}(y)\big)
-\frac{a_{\til{\nu}}^j(y)x^{\til{\nu}-1}\log x}{2\til{\nu}}+\mathcal O(x^{\til{\nu}+2+\epsilon})\Big)|dg_b|^\demi.
\end{equation} 
for some $\beta^{j}_{z}\in E_{z}$ such that
$\beta^{j}_{z}=0$ for $z<\nu_{0}$.
By construction, one has 
\[ \Delta_q\Big(x^{-1}(Q_{b}+\sum_{j=1}^{J}(\chi_{j}\otimes\varphi_{j}+
\varphi_{j}\otimes\chi_{j}))x^{-1}\Big)={\rm Id}-\sum_{j=1}^J\psi_j\otimes\psi_j.\]
We now define our model operators
at zf as follows:
\begin{equation}\label{Gzf}
\begin{split}
G_{\zf}^{-2} = &\sum_{j=1}^{J}\psi_{j}\otimes\psi_{j},\\  
G_{\zf}^{\alpha}= & 0, \quad  \forall\ \alpha\in(0,2),\\
G_{\zf}^{0}= & (xx')^{-1}(Q_{b}+\sum_{j=1}^{J}(\chi_{j}\otimes\varphi_{j}+
\varphi_{j}\otimes\chi_{j})).
\end{split}\end{equation}
Again as in \cite{gh2}, we have near the interior of $\zf$
\[(\Delta_q+k^{2})(k^{-2}G_{\zf}^{-2}+G_{\zf}^{0})|\tfrac{dk}{k}|^\demi=
(\Delta_qG_{\zf}^{0}+G_{\zf}^{-2}+\mathcal O(k^{2}))|\tfrac{dk}{k}|^\demi={\rm Id}+\mathcal O(k^{2}).\]
This means that
for any polyhomogeneous conormal parametrix $G(k)$ which agrees with these models at $\zf$ to positive order,
the lifted Schwartz kernel of $(\Delta_q+k^{2})G(k)-{\rm Id}$ vanishes to a positive order at $\zf$. \\

Now we need to check compatibility of these models with the model at $\bfo$; the analysis again proceeds exactly as in \cite{gh2}. Let
$\rho_{\zf}$ and $\rho_{\bfo}$ be smooth defining functions of $\zf,\bfo$ such that $\rho_{\zf}\rho_{\bfo}=k$; then we must show that the coefficient of the $\rho_{\zf}^{l+2}$ term in the asymptotic expansion of $G_{\bfo}^{-2}$ at $\bfo\cap \zf$ agrees with 
the coefficient of the $\rho_{\bfo}^{-2-l}$ term in the expansion of $G_{\zf}^l$ at $\bfo\cap\zf$ for $-2\leq l\leq 0$.
Even though they are not smooth on $\MMksc$, it suffices (and is convenient) to choose near the interior of $\zf\cap\bfo$ the functions 
$\rho_{\zf}=(\kappa\kappa')^{1/2}=k/(xx')^{\demi}$ and $\rho_{\bfo}=(xx')^{1/2}$. 
The term $G_{\bfo}^{-2}$ vanishes to second order at $\zf$, and $G_{\zf}^{-2}$ vanishes at $\bfo$ and therefore matches with $G^{-2}_{\bfo}$. Then the same exact argument as in \cite[Sec 3.5]{gh1} shows that $(\rho_{\zf}^{-2}G_{\bfo}^{-2})|_{\sf\cap\bfo}=(\rho_{\bfo}^{2}G_{\zf}^0)|_{\zf\cap\bfo}$. Note in particular (as in \cite{gh2}) that since
$x^{-1}\chi_j \otimes x^{-1}\varphi_j$ has a kernel which has order $2(\til{\nu}-2)>-2$ at $\bfo$ by \eqref{varphij} and \eqref{aschij}, it does not contribute.

\subsection{Terms at $\rbo$ and $\lbo$. }
Next we construct terms $G_{\rbo}^{*}$ at $\rbo$.  As in \cite[Sec. 3.7 and 4.4]{gh1} and \cite[Sec 3.3 and 4]{gh2},
these are determined by the first few terms in the Taylor series of $G_{\zf}^{l}$ at $\zf \cap \rbo$ for $l\leq 0$, as well as the expansion of $G_{\bfo}^{-2}$ at $\bfo\cap\rbo$. To analyze these Taylor series, we begin with the kernel $Q_b$. 
Localizing near $\rb$, the kernel of the identity vanishes identically and we have
$$
P_b Q_b = Q_b P_b = -\Pi_b.
$$
Using Theorem~\ref{reg}, \eqref{noresPb} and the formal expansion at $\rbo$ of this equation,
 we can write the following asymptotic for $Q_b$ at $\rb$:
\begin{equation}\label{qbright}
Q_b= \Big(\sum_{k=0}^1\sum_{\substack{\nu\in \mc{I}(P_b),\\
\nu_0\leq\nu\leq \nu_0+2}} \omega_{\nu,k}(z,y'){x'}^{\nu}\log(x')^k +\mc{O}({x'}^{\nu_0+2+\eps})\Big)|dg_bdg'_b|^\demi
\end{equation}
for some $\omega_{\nu,k}\in x^{-\nu-\eps}L^2(M\x\pl \bbar{M};\Lambda_\sca^q\otimes E_\nu)$.
By considering the operator operating on the right variable and using Proposition \ref{technical}, 
we see that $\omega_{\nu,1}=0$ if $\nu<\til{\nu}$ and 
\[\omega_{\til{\nu},1}(z,y')|dg_bdg'_b|^\demi= \sum_{j=1}^J\varphi_j(z)\otimes \frac{a_{\til{\nu}}^j(y')}{2\til{\nu}}|dg'_b|^\demi.
\]
We write $\omega_\nu=\omega_{\nu,0}$ to simplify notation. By considering  the operator $P_b$ acting on the left variable, and using \eqref{pbxnulogk}, we see that 
\[P_b(\omega_{\nu}(z ,y')|dg_bdg'_b|^\demi)=\left\{\begin{array}{ll}
0, & \textrm{ if }\nu<\til{\nu} \\
-\sum_{j=1}^J\varphi_j(z)\otimes a_\nu^j(y')|dg'_b|^\demi, & \textrm{ if }\nu\geq \til{\nu} 
\end{array}\right..\]
By matching the series 
\eqref{qbright} with the expansion of $Q_b|_{\zf\cap \bfo}$ (given by \eqref{Qbatbf}) at $\rbo$ (i.e. at $s=\infty$), and using Proposition \ref{technical}, we have that  if $\nu\in[\nu_0,\nu_0+2)$
\begin{equation}
\omega_{\nu}(x, y, y') =  \Big(\frac{x^{-\nu} \Pi_{E_\nu}(y,y')}{2\nu} +
\sum_{\substack{\mu\in \mc{I}(P_b),\\
\nu_0\leq \mu <\nu}} x^{-\mu}\gamma_{\nu,\mu}(y,y')+
\mc{O}(x^{-\nu+n_0})\Big).
\label{omeganu}\end{equation}
for some $\gamma_{\nu,\mu}(y,y')\in E_{\mu}\otimes E_{\mu}$. Notice that the $\mc{O}(x^{-\nu+n_0})$ remainder is in $H_b^0$ when  $n_0>\nu_0+2$. 

Recalling that $m=\min(\til{\nu}-\nu_0,2)$, we may now write down the asymptotic expansions at $\rbo$ of $G_{\zf}^{*}$:
\[\begin{split} 
G_{\zf}^{-2}= & \sum_{\substack{\nu\in \mc{I}(P_b),\\
\til{\nu}\leq \nu\leq\til{\nu}+2}}\sum_{j=1}^Jx^{-1}\varphi_j(z)\otimes {x'}^{\nu-1}b_\nu^j(y')|dg_b'|^\demi+
\mc{O}({x'}^{\til{\nu}+2+\eps}),\\
G_{\zf}^0=& \Big(\sum_{j=1}^J\sum_{\substack{\nu\in \mc{I}(P_b),\\
\til{\nu}\leq \nu \leq \til{\nu}+2}}x^{-1}\varphi_j(z)\otimes (\frac{b_\nu^j(y')}{4(1-\nu)}+\beta^j_{\nu-2}(y')){x'}^{\nu-3}+\sum_{j=1}^Nx^{-1}\chi_j(z)\otimes {x'}^{\til{\nu}-1}a_{\til{\nu}}^j(y')\Big)|dg'_b|^\demi
\\
& + \sum_{\substack{\nu\in \mc{I}(P_b),\\
\nu_0\leq\nu\leq \nu_0+2}} x^{-1}{x'}^{\nu-1}\omega_{\nu}(z,y')|dg_bdg'_b|^\demi+\mc{O}({x'}^{\nu_0+m-1+\eps}).
\end{split}\]
On the other hand, the matching with $G_{\bfo}^{-2}$ is explained in details in the asymptotically Euclidean case in \cite[Sec. 3.7]{gh1} (see also \cite[Sec. 4]{gh2} for the general asymptotically conic case) and comes from the expansion at $\kappa=0$ of the $I_{\nu}(\kappa)$ terms in \eqref{Qbfo}.

We can now write down several terms at $\rbo$. For $\nu\in\mc{I}(P_b) \cap [\nu_0+m,\nu_0+2+m)$, we set
\begin{equation}\label{Grbonu-3}
\begin{split}
G_{\rbo}^{\nu-3}:= & {\kappa'}\til{K}_{\nu-2}(\kappa')\Big(\sum_{j=1}^Jx^{-1}\varphi_j(z)\otimes \beta_{\nu-2}^j(y')+x^{-1}\omega_{\nu-2}(z,y')|dg_b|^\demi\Big)\Big|\frac{d\kappa'}{\kappa'}dy'\Big|^\demi\\
&+ {\kappa'}\til{K}_{\nu}(\kappa')\sum_{j=1}^Jx^{-1}\varphi_j(z)\otimes b_{\nu}^j(y')\Big|\frac{d\kappa'}{\kappa'}dy'\Big|^\demi
\end{split}\end{equation}
where $\til{K}_{\nu}(\kappa'):=\frac{1}{\Gamma(\nu)2^{\nu-1}}K_{\nu}(\kappa')$. Here we use the convention that $a_\nu^j=0$ when $\nu<\til{\nu}$ and $\beta^j_{\nu}=0$ when $\nu<\nu_0$; this is in order to make the notation consistent when $\til{\nu}>\nu_0+2$. 
Notice that $\Delta_q$ acting on the left annihilates these models, since it kills $x^{-1}\varphi_j(z)$ and 
$x^{-1}\omega_{\nu-2}(z,y')$ for $\nu-2<\nu_0+m\leq \til{\nu}$. Moreover, these models match the models at $\zf$ and $\bfo$ by construction (they also vanish at infinite order at $\rb$, since $K_{\nu}(\kappa')=\mc{O}({\kappa'}^{-\infty})$ as $\kappa'\to \infty$).

For Riesz transform purposes, we compute $(d+\delta)G_{\rbo}^{\nu-3}$ with $\nu\in\mc{I}(P_b) \cap [\nu_0+m,\nu_0+2+m)$, where $d+\delta$ acts in the left $z$-variable: 
\begin{equation}\label{d+delta2exp}
(d+\delta) G_{\rbo}^{\nu-3}=\Big((d+\delta)(x^{-1}\omega_{\nu-2}(z,y'))|dg_b|^\demi\Big) {\kappa'}\til{K}_{\nu-2}(\kappa')\Big|\frac{d\kappa'}{\kappa'}dy'\Big|^\demi.
\end{equation}
Later, we shall discuss when these terms are $0$.

Next we compute a higher order term at $\rbo$; the analysis splits into two cases. \\

\textbf{Case 1.} Assume first that $\til{\nu}<\nu_0+2$ (i.e. $m<2$). We set 
\[\begin{split}
G_{\rbo}^{\til{\nu}-1}:= & {\kappa'}\til{K}_{\til{\nu}}(\kappa')\Big(\sum_{j=1}^Jx^{-1}\varphi_j(z)\otimes \beta_{\til{\nu}}^j(y')+
x^{-1}\chi_j(z)\otimes a^j_{\til{\nu}}(y')+x^{-1}\omega_{\til{\nu}}(z,y')|dg_b|^\demi\Big)\Big|\frac{d\kappa'}{\kappa'}dy'\Big|^\demi\\
&+ {\kappa'}\til{K}_{\til{\nu}+2}(\kappa')\sum_{j=1}^Jx^{-1}\varphi_j(z)\otimes b_{\til{\nu}+2}^j(y')\Big|\frac{d\kappa'}{\kappa'}dy'\Big|^\demi.
\end{split}
\]
The matching with the models at $\bfo$ and $\zf$ may be checked in a similar way to the matching conditions for the lower order terms. Using that 
$P_b(\sum_j\chi_j(z)\otimes a_{\til{\nu}}^j(y'))+P_b(\omega_{\til{\nu}}(z,y'))=
-\sum_{j}\varphi_j(z)\otimes b^j_{\til{\nu}}(y')$, we have 
\begin{equation}\label{DeltaGrbo}
\Delta_qG^{\til{\nu}-1}_{\rbo}=-G_{\rbo}^{\til{\nu}-3},
\end{equation}
where $\Delta_q$ acts tangentially on $\rbo$ in the left $z$ variable. This implies that
any polyhomogeneous kernel $G(k)$ with expansion at $\rbo$
\[G(k)\sim \sum_{\substack{\nu\in\mc{I}(P_b)\\ \til{\nu}\leq \nu\leq \til{\nu}+2
}}k^{\nu-3}G_{\rbo}^{\nu-3}|\tfrac{dk}{k}|^\demi\] 
will be such that near any point of the interior of $\rbo$
\[(\Delta_q+k^2)G(k)=\mc{O}(\rho_{\rbo}^{\til{\nu}-1+\eps})\]
for some $\eps>0$.
An important observation for Riesz transform purposes is that 
\begin{equation}\label{d+delta1}
(d+\delta)G_{\rbo}^{\til{\nu}-1}\not=0
\end{equation}
since otherwise one would have $\Delta_qG_{\rbo}^{\til{\nu}-1}=0$ by applying $d+\delta$ on the left, contradicting \eqref{DeltaGrbo}.\\

\textbf{Case 2.} On the other hand, if $\til{\nu}\geq \nu_0+2$ (i.e. $m=2$), we set 
\[ \begin{split}
G_{\rbo}^{\nu_0+1}:= & {\kappa'}\til{K}_{\nu_0+2}(\kappa')x^{-1}\Big(\sum_{j=1}^J\varphi_j\otimes 
\beta_{\nu_0+2}^j+
\chi_j\otimes a^j_{\nu_0+2}+\omega_{\nu_0+2}|dg_b|^\demi\Big)\Big|\frac{d\kappa'}{\kappa'}dy'\Big|^\demi\\
&+ {\kappa'}\til{K}_{\nu_0+4}(\kappa')\sum_{j=1}^Jx^{-1}\varphi_j\otimes b_{\nu_0+4}^j\Big|\frac{d\kappa'}{\kappa'}dy'\Big|^\demi-{\kappa'}\til{K}_{\nu_0}(\kappa')x^{-1}v(z,y')\Big|\frac{d\kappa'}{\kappa'}dy'\Big|^\demi
\end{split}
\]
where $v(\cdot,y')\in x^{-\nu_0-2-}H_b^0$, depending smoothly on $y'$, satisfies for all $y'\in N$
\begin{equation}\label{v(z,y')}
\begin{gathered}
P_bv(z,y')=\omega_{\nu_0}(z,y')+\sum_{j=1}^J\varphi_j(z)\otimes\beta_{\nu_0}^j(y'), \\
 v(z,y')=\Big(\frac{x^{-\nu_0-2}\Pi_{E_{\nu_0}}(y,y')}{8\nu_0(\nu_0+1)}+\mc{O}(x^{-\nu_0-1})\Big)|dg_b|^\demi.
\end{gathered}\end{equation}
It is not obvious that a form $v(z,y')$ satisfying \eqref{v(z,y')} exists, in particular when $\til{\nu}=\nu_0+2$,
but this follows from Lemma 3.1 of \cite{gh2}. We refer the reader to that paper for details; note that the proof is written in the asymptotically Euclidean setting, but as in \cite[Sec. 4]{gh2}, it applies to the asymptotically conic setting as well.
It is now easy to check that
$G_{\rbo}^{\nu_0+1}$ matches with $G_{\bfo}^{-2}$, due to the leading asymptotic of $v(z,y')$ in \eqref{v(z,y')} as $x\to 0$; matching with $G_{zf}^*$ is also easy to check. Moreover, by construction, one has 
\begin{equation}\label{DeltaGrbo2}
\Delta_qG^{\nu_0+1}_{\rbo}=-G_{\rbo}^{\nu_0-1}
\end{equation}
where $\Delta_q$ acts tangentially on $\rbo$ in the left $z$ variable. This implies that 
any polyhomogeneous kernel $G(k)$ with expansion at $\rbo$
\[G(k)\sim \sum_{\substack{\nu\in\mc{I}(P_b)\\ \nu_0\leq \nu\leq \nu_0+2
}}k^{\nu-1}G_{\rbo}^{\nu-1}\otimes |\tfrac{dk}{k}|^\demi\] 
will be such that near any point of the interior of $\rbo$
\[(\Delta_q+k^2)G(k)=\mc{O}(\rho_{\rbo}^{\nu_0+1+\eps})\]
for some $\eps>0$. As before, a consequence of \eqref{DeltaGrbo2} is that 
\begin{equation}\label{d+delta2}
(d+\delta)G_{\rbo}^{\nu_0+1}\not=0
\end{equation}
since otherwise one would have $\Delta_qG_{\rbo}^{\nu_0+1}=0$.\\

The $\lbo$ terms $G_{\lbo}^{\nu-3}$ for $\nu\in [\nu_0+m,\nu_0+m+2]$ are defined from $G_{\rbo}^{\nu-3}$ by switching the primed and unprimed coordinates. The matching follows from the symmetry of the models at $\zf$ and $\bfo$ under the involution $(z,z')\rightarrow (z',z)$.
For later Riesz transform purposes, we observe that the dependence of $G_{\lbo}^{\mu}(\kappa, y,z')$ on $\kappa$ and $y$ 
is always a sum of terms of the form $\kappa \til{K}_{\mu}(\kappa) c_\mu(y)$, where $c_\mu \in E_\mu$. On the other hand, as in \eqref{Pbfo}, we have
\[\begin{gathered}
(k^{-2}\Delta_q+1)|_{\lbo}= \kappa^{-1}P_{\lbo}\kappa^{-1}\quad \textrm{ with }\\
P_{\lbo}:=\left(\begin{array}{cc}
-(\kappa\pl_\kappa)^{2}+(\ndemi-q-1)^2+\Delta_{N}+\kappa^2 & 2d_N\\
2\delta_N & -(\kappa\pl_\kappa)^2+(\ndemi-q+1)^2+\Delta_N+\kappa^2
\end{array}
\right).\end{gathered}\] 
From this equation together with the spectral decomposition, we see that $(k^{-2}\Delta_q+1)|_{\lbo}$ acting in the left variable kills the models $G_{\lbo}^{\nu-3}$ (the key is that $K_{\mu}(\kappa)c_\mu(y)$ is in the kernel of the corresponding Bessel operator). This implies that any polyhomogeneous $G(k)$ with asymptotic expansion at $\lbo$ given by the $G_{\lbo}^{\nu-3}$ will
satisfy $(\Delta_q+k^2)G(k)=\mc{O}(\rho_{\lbo}^{\nu_0+m+1+\eps})$ near any point of the interior of $\rbo$.

\subsection{Error term and resolvent} 
Let $G(k)\in\Psi_k^{-2,(-2,0,0),\mc{G}}(M)$ be a pseudo-differential operator in the calculus
which has all the prescribed terms defined above at the faces $\bfo,\rbo,\lbo,\sca,\zf$, and where $\mc{G}$ is an index set satisfying   
\[\begin{gathered} \mc{G}_\sca=0, \quad \mc{G}_{\zf}\subset (-2,0)\cup (0,0)\cup \mc{N},\quad 
\mc{G}_{\bfo}\subset (-2,0)\cup \mc{N},\,\, \\
\mc{G}_{\rbo}=\mc{G}_{\lbo}\subset \bigcup_{\nu=\nu_0+m}^{\nu_0+m+2}(\nu-3,0) \cup \mc{N}
\end{gathered}\]
for some index set $\mc{N}>0$.
The error term $E(k)$ defined by $(P + k^2) G(k) = \Id - E(k)$ is also polyhomogeneous conormal 
on $\MMksc$. By construction, its index set $\mc{E}$ satisfies 
\begin{equation}\label{indexsetE}
\begin{gathered} \mc{E}_\sca\geq 1, \quad \mc{E}_{\zf}>0,\quad 
\mc{E}_{\bfo}>0,\,\, \\
\mc{E}_{\rbo}>\nu_0+m-1, \quad \mc{E}_{\lbo}>\nu_0+m+1.
\end{gathered}\end{equation}
Therefore, as a consequence of the discussion of the inverse of $\Id-E(k)$ in section \ref{pdcle}, 
 the series \eqref{invde1-E} converges for small $k>0$, so the Neumann series construction yields an inverse which we write $\Id+S(k)$. By the composition law, $\Id + S(k) = (\Id - E(k))^{-1}$ lies in the calculus, and $S(k)$ has an index family $\mc{S}$ satisfying the same lower bounds \eqref{indexsetE} as $\mc{E}$.

The resolvent itself is given by $R_q(k) = G(k) + G(k) S(k)$, and a final application of the composition law to analyze $G(k)S(k)$ gives the following refinement of Theorem \ref{introresthm}:
\begin{theorem}\label{thm:nullspace} 
Let $(M,g)$ be  asymptotically conic to order $3$, ie. it satisfies \eqref{assumhx} with $n_0=3$. Assume that $0\notin \mc I(P_b)$ and there are no zero-resonances.
Then there exists $k_0>0$ such that the resolvent $R_q(k) = (\Delta_q + k^2)^{-1}$ on half-densities satisfies for $k\in(0,k_0]$
\begin{equation}
R_q(k) \in \Psi_k^{-2, (-2, 0, 0), \mc{R}}(M)
\end{equation}
for some index family $\mc{R}$ satisfying 
\[
\mc{R}_{\rm zf} \subset -2 + \mc{N}, \quad  
\mc{R}_{\bfo} \subset -2+ \mc{N}, \quad 
\mc{R}_{\sca} = 0, \quad \mc{R}_{\lbo} = \mc{R}_{\rbo} \subset \nu_0+m-3+ \mc{N}
\] 
for some 
with some index set $\mc{N}>0$ and $m\in[0,2]$ defined by \eqref{defm}. Moreover, the leading term of the resolvent kernel at $\zf, \bfo, \sca, \lbo, \rbo$ is equal to the leading term of the parametrix $G(k)$ at the corresponding face as defined above.
\end{theorem}

\begin{remark} The condition $0\notin\mc I(P_b)$ is equivalent to the statement that $\nu_0>0$, where $\nu_0$ is defined by \eqref{defnu0}, which is also equivalent to the condition \eqref{0notindroot} of the Introduction. 
We will show in Lemma \ref{wittconseq} that when $n$ is odd, $0\notin \mc{I}(P_b)$
and there is no zero-resonance for $\Delta_q$ if either $q\not=(n\pm 1)/2$ or if $q=(n\pm 1)/2$ and the modified Witt condition ${\rm Sp}_{\Lambda^{(n-1)/2}}(\Delta_N)\cap[0,3/4]=\emptyset$ holds. When $n$ is even, the same conclusion holds if $|q-n/2|>1$, if $q=n/2-1$ and $H^{n/2-1}(N)=0$, if $q=n/2+1$ and $H^{n/2}(N)=0$, or if $q=n/2$ and ${\rm Sp}_{\Lambda^{n/2}}(\Delta_N)\cap[0,1]=\emptyset$ (see Remark \ref{noresonanceneven}).  
For asymptotically Euclidean manifolds, these conditions hold; thus Theorem \ref{thm:nullspace} applies to asymptotically Euclidean manifolds whenever $n\geq 3$ (as long as $n_0\geq 3$). 
\end{remark}
\begin{remark}\label{rem5} Taking $\theta\in (-\pi/2,\pi/2)$ fixed, the same proof as above shows that $R_q(e^{i\theta}k)=(\Delta_q-k^2e^{2i\theta})^{-1}$ has exactly the same property as $R_q(k)$ in terms of conormal polyhomegeneity and the construction above is smooth in $\theta$. The construction remains essentially the same 
by replacing $k$ and $\kappa$ by $ke^{i\theta}$
and $ke^{i\theta}/x$, and using the fact that $I_{\nu}(\kappa e^{i\theta})$ and $K_{\nu}(e^{i\theta}k)$
satisfy all the needed properties for any $\theta\in(-\pi/2,\pi/2)$.
\end{remark}


\section{Riesz transform}\label{secRT}

\subsection{Definition of Riesz transorm on forms} 
First, let us define the operator 
\[D:=(d+\delta)\]
whose square $D^2$ is the Laplacian $\Delta_q$ when acting on $q$-forms.
To discuss the Riesz transform for $q$-forms, we first define, for $\eps>0$ small, the bounded operator on $L^2$
\[T_q^\eps:= D(\Delta_q+\eps)^{-1/2}.\]
Notice that $(T_q^\eps)^*T_q^\eps=\Delta_q(\Delta_q+\eps)^{-1}$ is a non-negative self-adjoint operator on $L^2$ with norm $1$ and with kernel $\ker_{L^2}(\Delta_q)$. For each $f\in L^2$, there exists a sequence $\epsilon_i$ going to zero for which $T_q^{\epsilon_i}f$ converges weakly to $h\in L^2(M;\Lambda_{\sca}^{q+1}\oplus \Lambda_{\sca}^{q-1})$, satisfying $||h||_{L^2}\leq ||f||_{L^2}$ and 
\[Dh= \Delta_q^{\demi}f \in H^{-1}(M;\Lambda_{\sca}^{q+1}\oplus \Lambda_{\sca}^{q-1}).\]
The limit $h$ is independent of the choice of $\epsilon_i$; to see this, note that since $\cjg T_q^\eps f,\psi\cjd=0$ for all $\eps>0$ and all $\psi$ in the kernel of $D$ on  
$L^2(M,\Lambda_{\sca}^{q+1}\oplus \Lambda_{\sca}^{q-1})$, we must have $\cjg h,\psi\cjd=0$. Taking another converging subsequence with limit $h'$, we observe that $D(h'-h)=0$ and $\cjg h'-h,\psi\cjd=0$ for all $\psi$ in the kernel of $D$, hence $h'=h$. We may therefore define the Riesz transform on $q$-forms by 
\[ T_qf:= h.\]
This is a linear map bounded on $L^2$ with norm $||T_q||_{L^2\to L^2}\leq 1$, and we denote it $T_q=D\Delta_q^{-1/2}$.

Notice that $T_q^\eps$ can be written in terms of the resolvent as the integral
\[T_q^\eps=\frac{2}{\pi}\int_{0}^{\infty}D(\Delta_q+\eps+k^2)^{-1}dk;\]
as $\epsilon$ goes to zero, for any $\alpha>0$, these converge as operators mapping $L^2(M)$ to the Sobolev space $H^{-\alpha}(M)$. We want to consider the weak limit as $\eps\to 0$, so we shall consider its Schwartz kernel.
We split the $k$-integral into an integral on $(0,k_0]$ and an integral on $[k_0,\infty)$ for some small $k_0$. By the arguments of \cite[Sec. 5.2]{gh1}, we have that 
$\frac{2}{\pi}\int_{k_0}^{\infty}DR_q(\sqrt{k^2+\eps})dk$
is a continous family for $\eps\in[0,1)$ of scattering pseudo-differential operators of order $-1$ in the sense of \cite{me3}. These are Calderon-Zygmund operators and hence bounded on $L^p$ for any $p\in(1,\infty)$. The value at $\eps=0$ is the following operator, which is also bounded on $L^p$:
\[T_q^{\rm hf}:=\frac{2}{\pi}\int_{k_0}^{\infty}DR_q(k)dk.\]
The more delicate part of the analysis concerns the low frequency region, and this is where we need our parametrix. We want to show that $\frac{2}{\pi}\int_{0}^{k_0}DR_q(\sqrt{k^2+\eps})dk$ has a limit as $\eps\to 0$ given by 
\begin{equation}\label{Tlf}
T_q^{\rm lf}:=\frac{2}{\pi}\int_{0}^{k_0}DR_q(k)dk,
\end{equation} 
and that the limit is well defined as a bounded operator on $L^p$ for some range of $p$ containing $2$. In fact we will simply show that \eqref{Tlf} is well defined, and from the proof it will be clear that the operators
$\frac{2}{\pi}\int_{0}^{k_0}DR_q(\sqrt{k^2+\eps})dk$
are bounded on $L^p$ uniformly down to $\eps=0$ for some interval of $p$ containing $2$.

\subsection{Indicial roots of $d$ and $\delta$}
To state the result about boundedness of Riesz transform, we need to define the index sets of $d$ and $\delta=d^*$ (or equivalently of $D:=d+\delta$). 
These operators act on sections of 
$\Lambda_{\sca}^q\otimes \Omega_\sca^{1/2}$ as described in Section \ref{laponforms}, but we can also see them as acting on $\Lambda_{\sca}^q\otimes \Omega_b^{1/2}$ by conjugating by $x^{n/2}$.\\
For $\nu\in I^j$, and $\omega_\nu\in E_{\nu}^j$ non zero, we see by using the expressions of $d,\delta$ in Section \ref{laponforms} (as before we use the isomorphism \eqref{decompos}) and performing a bit of algebra that
\begin{equation}\label{indicddelta}
\begin{gathered}
d(x^{\nu+\ndemi-1}\omega_\nu)=0\textrm{ near }\pl \bbar{M} \iff \nu\in I_d^j, \\
\delta(x^{\nu+\ndemi-1}\omega_\nu)=\mc{O}(x^{\nu+\ndemi+n_0})\textrm{ near }\pl \bbar{M} 
\iff \nu\in I^j_\delta
\end{gathered}
\end{equation}
where 
\begin{equation}\label{indrootofD}
\left\{\begin{array}{l}
I^1_d=\{-(\ndemi-q-1)\}\cap I^1\\ 
I^2_d=I^2\\
I^3_d =\{-(\sqrt{(\ndemi-q)^2+\alpha^2}-1); \alpha^2\in S^q_{d_N}\}\\  
I^4_d =\{\sqrt{(\ndemi-q)^2+\alpha^2}+1; \alpha^2\in S^q_{d_N}\}
\end{array}\right.,
\left\{\begin{array}{l}
I^1_\delta=I^1\\ 
I^2_\delta=\{\ndemi-q+1\}\cap I^2 \\
I^3_\delta= \{-(\sqrt{(\ndemi-q)^2+\alpha^2}-1); \alpha^2\in S^q_{d_N}\}\\  
I^4_\delta= \{\sqrt{(\ndemi-q)^2+\alpha^2}+1; \alpha^2\in S^q_{d_N}\}.
\end{array}\right.
\end{equation}
Let us define these indicial sets of $d$ and $\delta$ by  
\[ \mc{I}(d):=\cup_{j=1}^4I_d^j, \quad \mc{I}(\delta):=\cup_{j=1}^4I_\delta^j.\]
We ultimately want to know when a harmonic $q$-form $\omega\in x^{-\alpha}H_b^{0}$ (as a section of $\Omega_b^{1/2}$), for some $\alpha\in [\nu_0,\nu_0+2]$,  is such that 
$d\omega \in H_b^0$ and $\delta \omega\in H_b^0$. To that aim we define 
\begin{equation}\label{defindices}
\begin{gathered}
\nu_d:=-\max (\rr^-\cap \mc{I}(P_b)\setminus \mc{I}(d)), \quad
\nu_\delta:=-\max(\rr^-\cap \mc{I}(P_b)\setminus \mc{I}(\delta))\\
\nu_D:=\min(\nu_d,\nu_\delta).
\end{gathered}
\end{equation}
From these definitions, we see that for a harmonic $q$-form $\omega\in x^{-\mu-1-}H_b^{0}$ with $\mu\in[\nu_0,\nu_0+2)$ which satisfies
\[ 
\omega=\Big(\sum_{\substack{\nu \in \mc{I}(P_b)\\ -\mu\leq \nu\leq -\nu_0}}x^{\nu-1}\omega_\nu(y)+\mc{O}(x^{\eps})\Big)|dg_b|^\demi, \quad \omega_{-\mu}\not=0, \,\, \omega_\nu\in E_\nu
\]
for some $\eps>0$, we certainly have (recall $D=d+\delta$)
\begin{equation}\label{d+deltaomzero}
d\omega \in H_b^0\Longrightarrow \mu <\nu_d ,\quad  
 \delta \omega \in H_b^0\Longrightarrow  \mu< \nu_\delta, \quad  D\omega \in H_b^0\Longrightarrow \mu <\nu_D.
\end{equation}
Moreover, if $n_0>\nu_0+2$, then the error term in $\delta$ does not interfere and we in fact have
\begin{equation}\label{d+deltaom}
d\omega \in H_b^0\iff \mu <\nu_d ,\quad  
 \delta \omega \in H_b^0\iff  \mu< \nu_\delta, \quad D\omega \in H_b^0\iff \mu <\nu_D. 
\end{equation}

In many cases these indices can be expressed in terms of smallest eigenvalues:  if  $\la_q>1-|q-n/2|^2$ we can 
characterize $\nu_D$  by the formula \eqref{defnu0intro} in the Introduction.
For the asymptotically Euclidean case where $N$ is a disjoint union of $\sph^{n-1}$, this gives 
\begin{equation}\label{casrn}
\begin{gathered}
\nu_\delta=\left\{\begin{array}{cc}
n/2 & \textrm{ if }q>0\\
\emptyset & \textrm{ if }q=0
\end{array}\right., \quad \nu_d=\left\{\begin{array}{cc}
n/2 & \textrm{ if }q<n\\
\emptyset & \textrm{ if }q=n
\end{array}\right.\end{gathered}.
\end{equation}

\subsection{Main theorem}
We will show the following:
\begin{theorem}\label{mainthRT}
Let $(M^n,g)$ be asymptotically conic to order $n_0\geq 3$ with cross-section $N$. 
Assume that $0\notin\mc I(P_b)$ and that $\Delta_q$ has no zero-resonances. Finally, let
and $\nu_0$ defined by \eqref{defnu0} and
\begin{equation}\label{defofmth} 
\nu_{\ker}:= \min \Big(\nu_0+2, \max \{ \nu; \ker_{L^2}(\Delta_q)\subset x^{\nu+\ndemi-1}L^\infty(M,\Lambda^q_\sca)\})\Big).
\end{equation}
Then the Riesz transform $T=D\Delta_q^{-1/2}$ on $q$-forms is bounded on $L^p$ if 
\begin{equation}\label{suff}
 \frac{n}{n-(n/2+1-\nu_{\ker})_+}<p< \frac{n}{(n/2-\nu_0)_+}.
 \end{equation}
To get the precise interval of boundedness, assume that $n_0>\nu_0+2$ if $\nu_0<n/2$. Then\\ 
\textbf{Case $1$.} If $q<n/2-1$ and the natural map $H^{q+1}(\bbar{M},\pl\bbar{M})\to H^{q+1}(\bbar{M})$ in cohomology is not injective, or if 
$q>n/2+1$ and the natural map $H^{n-q+1}(\bbar{M},\pl\bbar{M})\to H^{n-q+1}(\bbar{M})$ in cohomology is not injective, then $T_q$ is bounded on $L^p$ if and only if \eqref{suff} holds.\\
\textbf{Case $2$.} In all other cases, let $\nu_D$ be the index defined by \eqref{defindices}; then $T_q$ is bounded on $L^p$ if and only if 
\[  \frac{n}{n-(n/2+1-\nu_{\ker})_+}<p< \frac{n}{(n/2-\min(\nu_D,\nu_{\ker}))_+}.\]
\end{theorem}

\subsection{Proof of Theorem \ref{mainthRT}}
The main step in our analysis is to describe the asymptotic behaviour of the Schwartz kernel of $T_q^{\rm lf}$ to 
deduce its sharp $L^p$ boundedness. 
We start with the following:
\begin{proposition}\label{basicinfo} 
The operator $DR_q(k)$ is a pseudo-differential operator in our calculus of order $-1$ with index set bounded below by $0$ at $\zf$, $-1$ at $\bfo$, $1$ at $\sca$, $\nu_{\ker}-3$ at $\rbo$, and $\nu_{\ker}-2$ at 
$\lbo$. Moreover, at $\lbo$ and $\rbo$ the leading nontrivial coefficient of the Schwartz kernel of $DR_q(k)$ is the same as that of $DG(k)$. 
\end{proposition}
\begin{proof} 
First notice that in the notation of Section \ref{resolvent-kernel}, $\nu_{\ker}=\nu_0+m$. Observe that $DR_q(k)=DG(k)+(DG(k))S(k)$, where $S(k)$ is defined in the previous section, so it is certainly an element of the calculus. We first analyze $DG(k)$. The operator $d+\delta=x.x^{-1}D$ is a first order scattering operator, with $x^{-1}D$ being an operator  
in ${\rm Diff}_b^1$. This algebra preserves conormal polyhomogeneity (with orders) on $M^2_{k,b}$ since $\mc{V}_b$ lifted from the left to $M^2_{k,b}$ consists of vector fields tangent to all boundary hypersurfaces. Thus $D$ acting in the left variable increases the index sets at $\lbo$ and $\bfo$ by $1$ and the pseudo-differential order (singularity at the diagonal) by $1$, it preserves the index set at $\sca$. Moreover,  
$DG_{\zf}^{-2}=0$ since $G_{\zf}^{-2}$ is the orthogonal projector on the $L^2$ kernel of $\Delta_q$ (thus of $D$), we see that $DG(k)$ has index set bounded below by $0$ at $\zf$. Finally, let $\mu_{\rbo}$ and $\mu_{\lbo}$ be the leading nontrivial orders of $DG(k)$ at $\rbo$ and $\lbo$ respectively. From \eqref{d+delta1} and \eqref{d+delta2}, we know that $\mu_{\rbo}\in[\nu_{\ker}-3,\nu_{\ker}-1]$, and we will prove that $\mu_{\lbo}=\nu_{\ker}-2$ (in fact, we will compute both $\mu_{\rbo}$ and $\mu_{\lbo}$ explicitly; see Lemma \ref{idleadingorder}). Since the index set $\mc{S}$ of $S(k)$ has $\mc{S}\geq\mc{E}$, we read off from the composition law \eqref{comp-if} that the leading order of $(DG(k))S(k)$ at $\lbo$ is at least $\min(\mu_{\lbo}+\eps,\nu_{\ker}+\eps)$, which is greater than $\mu_{\lbo}$. Similarly, the leading order of $(DG(k))S(k)$ at $\rbo$ is at least $\min(\nu_{\ker}-1+\eps,\mu_{\rbo}+\eps)$, which is greater than $\mu_{\rbo}$. This completes the proof. \end{proof}

As in the proof of the Proposition, we define $\mu_{\rbo}$, $\mu_{\lbo}$ by 
\[ \begin{gathered}
\mu_{\rbo} =\textrm{ the smallest elements of the index set of } DR_q(k) \textrm{ at }\rbo, \\ 
 \mu_{\lbo} =\textrm{ the smallest elements of the index set of } DR_q(k) \textrm{ at }\lbo .
\end{gathered} \]
We must compute $\mu_{\rbo}$ and $\mu_{\lbo}$.
The following two lemmas are the necessary ingredients.
\begin{lemma}\label{idleadingorder} For each degree $q$, let $e_q:H^{q}(\bbar{M},\pl\bbar{M})\rightarrow H^{q}(\bbar{M})$ be the map coming from the long exact sequence for relative cohomology. Assume $n_0\geq 3$.

a) In all cases, $\mu_{\lbo}=\nu_{\ker}-2$ and $\mu_{\rbo}\geq\nu_0-1$.

b) Suppose $n_0>\nu_0+2$. If $q<n/2-1$ and $e_{q+1}$ is not injective, or if $q>n/2+1$ and $e_{n-q+1}$ is not injective, then $\mu_{\rbo}=\nu_0-1$.

c) If $n_0>\nu_0+2$ and b) does not hold, $\mu_{\rbo}=\min(\nu_D,\nu_{\ker})-1$, with $\nu_D$ defined by \eqref{defindices}.
\end{lemma}

\begin{proof} We need to determine as far as possible the leading-order terms of the Schwartz kernel of $DG(k)$ at the faces $\lbo$ and $\rbo$.\\

\emph{Asymptotic at $\lbo$.}
It is direct to see that $I_{\lbo}(k^{-1}d)$ and $I_{\lbo}(k^{-1}\delta)$ do not both kill the leading term 
$G_{\lbo}^{\nu_{\ker}-3}$ of $R_q(k)$ at $\lbo$.
Indeed, in the decomposition \eqref{decompos},  
$I_{\lbo}(k^{-1}d)= A'\kappa^{-1}$ and $I_{\rbo}(k^{-1}\delta)= B\kappa^{-1}$ where 
\begin{equation}\label{AB} 
A'=\left(\begin{array}{cc}
d_N & 0\\
-\kappa\pl_\kappa+\ndemi -q -1& -d_N
\end{array}\right), \quad  B=\left(\begin{array}{cc}
\delta_N & \kappa\pl_\kappa+\ndemi-q+1\\
0 & -\delta_N
\end{array}
\right).\end{equation} Now consider an element of the form $F(\kappa)a(y)\otimes \phi(z')$ with $a\in \cup_{i=1}^4 E^i$, 
$F\in C^\infty((0,\infty))$ and $\phi\in x^{-m}H_b^0$ for some $m>0$. It is killed by $I_{\lbo}(k^{-1}d)$ if and only if either $a\in E^2$ or $a\in E^{3}\cup E^4$ with $F(\kappa)$ a power of $\kappa$. Similarly, those killed by $I_{\lbo}(k^{-1}\delta)$ must have either $a\in E^1$ or 
$a\in E^{3}\cup E^4$ with $F(\kappa)$ a power of $\kappa$. In particular, since $G_{\lbo}^{\nu_{\ker}-3}$ is a sum of elements of the form $\kappa K_{\nu}(\kappa)a_\nu(y)\otimes \phi(z')$ with $a_\nu\in E_\nu$ for some $\nu\in \mc{I}(P_b)$, and since $E^1\cap E^2=\{0\}$,
we have that $DG(k)$ (and hence $DR(k)$ by Proposition \ref{basicinfo}) has a non-vanishing term at $\lbo$ at order $\nu_{\ker}-2$ given by 
$(A'+B)(G_{\lbo}^{\nu_{\ker}-3})$. Hence $\mu_{\lbo}=\nu_{\ker}-2$, which proves the first claim of a) in Lemma \ref{idleadingorder}.\\

\emph{Asymptotic at $\rbo$.} Applying $D$ to $G(k)$ and considering the asymptotic expansion at $\rbo$, we have using \eqref{d+delta2exp}
\[\begin{split}
DG(k)= & \sum_{\substack{\nu \in \mc{I}(P_b)\\ \nu_{\ker}\leq \nu<\nu_{\ker}+2}}k^{\nu-3}
D(x^{\ndemi-1}\omega_{\nu-2}(z,y'))|dg|^\demi {\kappa'}\til{K}_{\nu-2}(\kappa')\Big|\frac{d\kappa'}{{\kappa'}}dy'\Big|^\demi+\\
 & + k^{\nu_{\ker}-1}D(G_{\rbo}^{\nu_{\ker}-1})+\mc{O}(\rho_{\rbo}^{\nu_{\ker}-1+\eps})
\end{split};\]
moreover, by \eqref{d+delta1} and \eqref{d+delta2}, $DG_{\rbo}^{\nu_{\ker}-1}\not=0$.
The problem of determining $\mu_{\rbo}$ thus reduces to finding when $D(x^{\ndemi-1}\omega_{\nu-2}(\cdot,y'))=0$ for all $y'\in N$. Note that since $\omega_{\nu-2}$ is zero by definition whenever $\nu-2<\nu_0$, there are no terms at order less than $\nu_0-1$; hence $\mu_{\rbo}\geq\nu_0-1$, which ends the proof of a).

Assume now that $n_0>\nu_0+2$, and let $\mu=\nu-2$. From \eqref{d+deltaom}, we see that  $d(x^{\ndemi-1}\omega_{\nu-2}(\cdot,y'))\in 
L^2$ for all $y'\in N$ if and only if $\mu<\nu_d$ and  
$\delta(x^{\ndemi-1}\omega_{\nu-2}(\cdot,y'))\in 
L^2$ if and only if $\mu<\nu_\delta$. Fix $y'\in \pl \bbar{M}$ and define \[v_{\mu,y'}:=x^{\ndemi-1}\omega_{\mu}(\cdot,y').\]

\emph{Action of $d$.} Suppose that $\mu\in[\nu_0,\min(\nu_d,\nu_{\ker}))$.
Recall the result of Hausel-Hunsicker-Mazzeo \cite[Thm. 1.A]{hhm}: there is a canonical linear isomorphism  
\begin{equation}\label{Th.hhm} 
\ker_{L^2}(\Delta_q) \to \left\{\begin{array}{ll}
H^q(\bbar{M},\pl\bbar{M}) & \textrm{ if }q<n/2\\
{\rm Im}(H^{q}(\bbar{M},\pl\bbar{M})\to H^{q}(\bbar{M})) & \textrm{ if }q=n/2\\
H^q(\bbar{M}) & \textrm{ if }q>n/2
\end{array}\right.
\end{equation}
where $H^q(\bbar{M})$ is the absolute cohomology and $H^{q}(\bbar{M},\pl\bbar{M})$ the relative.
Since $dv_{\mu,y'}$ is an $L^2$-harmonic form, we see that it represents a class $[dv_{\mu,y'}]$ in $H^{q+1}(\bbar{M},\pl\bbar{M})$ if $q<n/2-1$ and
a class in $H^{q+1}(\bbar{M})$ if $q\geq n/2-1$. For $q=n/2-1$, it is in the image of the extension map 
$e_{q+1}:H^{q+1}(\bbar{M},\pl\bbar{M})\to H^{q+1}(\bbar{M})$ arising in the long exact sequence for relative cohomology:
\begin{equation}\label{les} 
\dots \to H^{q}(\pl \bbar{M})\to H^{q+1}(\bbar{M},\pl\bbar{M})\to H^{q+1}(\bbar{M})\to H^{q+1}(\pl \bbar{M})\to \dots .
\end{equation}

To compute $[dv_{\mu,y'}]$, we first note that $v_{\mu,y'}$ is only conormal at the boundary rather than smooth. However, we may apply the argument of Melrose \cite[Lemma 6.11]{me} in order to regularize $v_{\mu,y'}$ without changing the class. Namely, we use the map $F_s:M\to M$ defined by 
$\pl_sF_s^*\phi:=F_s^*\pl_x\phi$ for $s>0$ small (with $F_0:={\rm Id}$) to show that 
$dF^*_s(v_{\mu,y'})=F^*_sdv_{\mu,y'}$
is an exact smooth form on $M$ (up to the boundary). By the fact that $F_s^*$ induces the identity in exact cohomology $H^{*}(M)$,we have that $[dv_{\mu,y'}]=0$ in $H^{q+1}(\bbar{M})$ if $q\geq n/2-1$
and $e_{q+1}([dv_{\mu,y'}])=0$ in $H^{q+1}(\bbar{M})$ if $q< n/2-1$.
If $q\geq n/2-1$ this implies that $dv_{\mu,y'}=0$. 

On the other hand, if $q<n/2-1$, we consider two cases. First, if $e_{q+1}:H^{q+1}(\bbar{M},\pl\bbar{M})\to H^{q+1}(\bbar{M})$ is injective, then $dv_{\mu,y'}=0$. 
If on the contrary it is not injective, then we must have $H^{q}(\pl \bbar{M})=\mc{H}^q(N)\not=0$, and then \eqref{I1I2} and \eqref{i3} imply that $\nu_0=n/2-q-1$. Then the form $v_{\nu_0,y'}$ is conormal at the boundary and is in $C^{0}(M,\Lambda^q)$ with leading behaviour $\omega_{\nu_0,y'}|_{\pl \bbar{M}}\in \ker \Delta_{N}=\mc{H}^q(N)\simeq H^q(\pl \bbar{M})$. As above, one can regularize it to make it smooth without changing the class in $H^q(\bbar{M})$, and so the class of $dv_{\nu_0,y'}$ in $H^{q+1}(\bbar{M},\pl\bbar{M})$ corresponds precisely to the image of $[v_{\nu_0,y'}|_{\pl \bbar{M}}]$ under the map $H^{q}(\pl \bbar{M})\to H^{q+1}(\bbar{M},\pl\bbar{M})$ in the long exact sequence. 
However, by \eqref{omeganu}, $2\nu_0v_{\nu_0,y'}(x,y)|_{x=0}=\Pi_{E_{\nu_0}}(y,y')$. Suppose $dv_{\nu_0,y'}=0$ for all $y'$. Then the image of the map $H^{q}(\pl \bbar{M})\to H^{q+1}(\bbar{M},\pl\bbar{M})$ is zero, and hence by exactness, $e_{q+1}$ is injective, which is a contradiction. We conclude that when $e_{q+1}$ is not injective, the supremum of $\mu\in[\nu_0,\nu_{\ker}]$ such that
$d(x^{\ndemi-1}\omega_{\mu'}(.,y'))=0$ for all $y'\in \pl \bbar{M}$ and all $\mu'<\mu$ is given by $\mu=\nu_0$. On the other hand, when $e_{q+1}$ is injective, it is given by $\mu=\min(\nu_d,\nu_{\ker})$. Therefore, the leading term at $\rbo$ of $dG(k)$ is at order $\nu_0-1$ when $e_{q+1}$ is not injective and $q<n/2-1$, while it is at order at least $\min(\nu_d,\nu_{\ker})-1$ in all other cases, and is at order exactly $\nu_d-1$ if $\nu_d<\nu_{\ker}$.\\ 

\emph{Action of $\delta$.} Now suppose that $\mu\in[\nu_0,\min(\nu_\delta,\nu_{\ker}))$ and consider $\delta(v_{\mu,y'})$. Using Poincar\'e duality and the Hodge star operator $\star$, we see that $\star \delta(v_{\mu,y'})\in L^2(M,\Lambda_{\sca}^{n-q+1})$ gives a representative in $H^{n-q+1}(\bbar{M})$ by \eqref{Th.hhm} when $q\leq n/2+1$ and in $H^{n-q+1}(\bbar{M},\pl\bbar{M})$ when $q>n/2+1$.
But this form is exact, so we may use the previous argument with the map $F_s$ to regularize and proceed as above. We conclude as before that the leading term of $\delta G(k)$ at $\rbo$ is at order $\nu_0-1$ when $e_{n-q+1}$ is not injective and $q>n/2-1$, while it is at order at least $\min(\nu_\delta,\nu_{\ker})-1$ in all other cases (and is at order exactly $\nu_\delta$ when $\nu_\delta<\nu_{\ker}$). \\

Finally, using the fact that $DG_{\rbo}^{\nu_{\ker}-1}\not=0$, the exact order of $DG(k)$ at $\rbo$  is
\begin{equation}\label{leadingrbo}
\begin{array}{ll}
\min(\nu_d,\nu_\delta,\nu_{\ker})-1  & \textrm{ if }q\in[n/2-1,n/2+1],\\
\min(\nu_d,\nu_\delta,\nu_{\ker})-1  & \textrm{ if }q<n/2-1 \textrm{ and }e_{q+1} \textrm{ is injective},\\
\min(\nu_d,\nu_\delta,\nu_{\ker})-1   & \textrm{ if }q>n/2+1 \textrm{ and }e_{n-q+1} \textrm{ is injective},\\
\nu_0-1 & \textrm{ in all other cases}.
\end{array}\end{equation}
This completes the proof of Lemma \ref{idleadingorder}.
\end{proof}

Now to know the asymptotic behaviour of the kernel of $T^{\rm lf}_q$, we need to integrate the kernel of $DR(k)$ in $k$:
\begin{lemma}\label{asymptoticd+deltaG} The first (resp. second) integral below does not vanish identically in $z$ (resp. $z'$):
\begin{equation}\label{kappaint}
\int_{0}^{\infty}{\kappa'}^{\mu_{\rbo}}(\delta+d)G_{{\rm rb}_0}^{\mu_{\rbo}}(z,\kappa',y')d\kappa';
\end{equation}
\begin{equation}\label{kappaintlbo}
\int_{0}^{\infty}{\kappa}^{\mu_{\lbo}+1}(I_{{\rm lb}_0}(k^{-1}\delta)+I_{{\rm lb}_0}(k^{-1}d))G_{{\rm lb}_0}^{\mu_{\lbo}}(\kappa,y,z') d\kappa.
\end{equation}
 \end{lemma}
\begin{proof}
First we prove the statement at $\rbo$. The argument splits into two cases. Suppose first that $DG^{\nu-3}_{\rbo}$ does not vanish for some $\nu\in[\nu_{\ker},\nu_{\ker}+2)$; then $D(x^{\ndemi-1}\omega_{\nu-2}(\cdot, y'))$ is nonzero. Note that $\nu>2$, otherwise $\omega_{\nu-2}=0$ by definition, yielding a contradiction. We want to show that
\[D(x^{\frac{n}{2}-1}\omega_{\nu-2}(\cdot, y'))\int_0^{\infty}\kappa'^{\nu-2}\tilde K_{\nu-2}(\kappa')\ d\kappa'\]
is nonzero. But since $\nu-2>0$, as in \cite[Sec. 5.2]{gh1}, the $\kappa'$ integral becomes $-\pi^{1/2}\frac{\Gamma(\nu-\frac{3}{2})}{\Gamma(\nu-2)}$, which is nonzero. This is enough.

The other possibility is that $DG^{\nu_{\ker}-1}_{\rbo}$ is the first nonvanishing term. Suppose for contradiction that (\ref{kappaint}) vanishes. Then applying $D$ to (\ref{kappaint}) and using the fact that in this case $D^{2}G_{\rbo}^{\nu_{\ker}-1}=-G_{\rbo}^{\nu_{\ker}-3}$, we must have that
\[\int_0^{\infty}\kappa'^{\nu_{\ker}-1}G_{\rbo}^{\nu_{\ker}-3}(\kappa',z,y')\ d\kappa'=0.\]
Using the formula (\ref{Grbonu-3}), we have
\begin{equation}
\begin{split}
& \int_0^{\infty}\Big(\kappa'^{\nu_{\ker}}\til{K}_{\nu_{\ker}-2}(\kappa')\Big(\sum_{j=1}^Jx^{-1}\varphi_j(z)\otimes \beta_{\nu_{\ker}-2}^j(y')+x^{-1}\omega_{\nu_{\ker}-2}(z,y')\Big)\\
&+ \kappa'^{\nu_{\ker}}\til{K}_{\nu_{\ker}}(\kappa')\sum_{j=1}^Jx^{-1}\varphi_j(z)\otimes b_{\nu_{\ker}}^j(y')\Big)=0.
\end{split}\end{equation}
This integral is a sum of three integrals, each given by a function of $(z,y')$ times an integral in $\kappa'$. Each of the three functions has different asymptotics in $(x,y,y')$ near the boundary, so each term must vanish individually. The argument above shows that the $\kappa'$ integral in the third term does not vanish, therefore the third function of $(z,y')$ does. To complete the proof, all we need to do is show that 
\[\int_0^{\infty}{\kappa'}^{\nu_{\ker}}\til{K}_{\nu_{\ker}-2}(\kappa')\ d\kappa'\neq 0,\]
as then the same argument will apply to the first two terms, forcing $G_{\rbo}^{\nu_{\ker}-3}=0$, which is a contradiction. We may assume $\nu_{\ker}>2$, as otherwise the first two terms are zero.

To compute the integral, write $\nu_s=\nu_{\ker}-2$, and let $F(\kappa')=\kappa^{\nu_s}\til K_{\nu_s}$. As in \cite{gh1}, $F(\kappa')$ is the Fourier transform of the function $f(t)=C(1+t^2)^{-(\nu_s)-1/2}$, where $C=-\frac{\Gamma(\nu_s+1/2)\pi^{-1/2}}{\Gamma(\nu_s)}$ is a nonzero constant. So the integral of $\kappa'^2F(\kappa')$ is a nonzero multiple of $f''(0)$. Doing the calculations, the integral is a nonzero multiple of $2C(-\nu_s-1)$, which is nonzero. This completes the case of $\rbo$.

We now need to prove a similar statement for $\lbo$. By Lemma \ref{idleadingorder}, $\mu_{\lbo}=\nu_{\ker}-2$. Using symmetry and (\ref{Grbonu-3}), we may write $(\ref{kappaintlbo})$ as
\begin{equation}\label{integralatlbo}\int_0^{\infty}\kappa^{\nu_{\ker}-2}(A'+B)\Big(\tilde K_{\nu_{\ker}-2}(\kappa)g_{\nu_{\ker}-2}(y,z')+\tilde K_{\nu_{\ker}}(\kappa)g_{\nu_{\ker}}(y,z')\Big)\ d\kappa,\end{equation}
where $A',B$ are given by \eqref{AB} and $g_{\nu}(\cdot, z')\in E_{\nu}\otimes \Lambda_{\sca}^q$ with $\nu>0$ for each $z'\in M$. Note that $g_{\nu_{\ker}}(y,z')$ and $g_{\nu_{\ker}-2}(y,z')$ have different asymptotics as $z'\to \pl \bbar{M}$. For the integral to be $0$, we thus need each term to be $0$; since $g_{\nu_{\ker}-2}$ is never identically zero, we consider the first integral in \eqref{integralatlbo}; suppose that it vanishes and set $\nu=\nu_{\ker}-2$.
We write $g_\nu(y,z')=\sum_{j=1}^4g_\nu^j(y,z')$ with $g_\nu^j\in E_\nu^j\otimes \Lambda_{\sca}^q$ (following the notation \eqref{L2decomp}).

First, we observe from the form of $A'$ that the $\Lambda^{q}(N)$ component of 
$A'(\til{K}_\nu(\kappa)g_\nu(y,z'))$ is zero only if $g^1_\nu(\cdot,z')\in
(\mc{H}^q(N)\oplus \{0\})\otimes \Lambda_{\sca}^q$, and thus $\nu=|n/2-q-1|$. In this case
\[\begin{split}
\int_{0}^{\infty}\kappa^\nu A'(\til{K}_\nu(\kappa)g^1_\nu(y,z'))d\kappa= \frac{dx}{x}\wedge g^1_\nu(y,z') \int_{0}^{\infty}\kappa^\nu (-\kappa\pl_\kappa+n/2-q-1)(\til{K}_\nu(\kappa))d\kappa
\end{split}\]
but integration by parts in $\kappa$ (here $\kappa^\nu \til{K}_\nu(\kappa)$ is continuous at $\kappa=0$) shows that 
\[\int_{0}^{\infty}\kappa^\nu (-\kappa\pl_\kappa+n/2-q-1)(\til{K}_\nu(\kappa))d\kappa=C\int_{0}^{\infty}\kappa^\nu \til{K}_\nu(\kappa)d\kappa
\]
with $C\geq 1$, and the integral does not vanish. This implies that $g_\nu^1=0$. 
The same argument works to show that $g_\nu^2=0$ using $B$ instead of $A'$. 
Finally, using \eqref{mu+-} and \eqref{i3} 
\[\begin{split}
\int_{0}^{\infty}\kappa^\nu A'(\til{K}_\nu(\kappa)g_\nu(y,z'))d\kappa= &
\frac{dx}{x}\wedge d_N\phi_{z'}^-(y) \int_{0}^{\infty}\kappa^\nu (-\kappa\pl_\kappa-1+\sqrt{(\ndemi-q)^2+\alpha_-^2})
(\til{K}_\nu(\kappa))d\kappa\\
+&\frac{dx}{x}\wedge d_N\phi_{z'}^+(y) \int_{0}^{\infty}\kappa^\nu (-\kappa\pl_\kappa-1-\sqrt{(\ndemi-q)^2+\alpha_+^2})
(\til{K}_\nu(\kappa))d\kappa
\end{split}
\]
for some $\phi^\mp_{z'}\in (\ker(\Delta_{N}-\alpha_\mp^2)\cap \Ima \delta_N)\otimes \Lambda_{\sca}^q$  with $\alpha_\mp>0$ and
$\nu=|\sqrt{(\ndemi-q)^2+\alpha_\mp^2}\mp1|$. Notice that $d\phi^-_{z'}$ and $d\phi^+_{z'}$ are independent (if non-zero) since $\alpha_+\not=\alpha_-$.
Integrating by parts, the integrals become
\[C_{\mp}\int_{0}^{\infty}\kappa^\nu 
\til{K}_\nu(\kappa)d\kappa\not=0,\]
where $C_{-}$ (for the first integral) is $\nu+\sqrt{(\ndemi-q)^2+\alpha_{-}^2}>0$, and $C_{+}$ (for the second integral) is $\nu-\sqrt{(\ndemi-q)^2+\alpha_{+}^2}=1$. We conclude that \eqref{integralatlbo} cannot be zero for all $z'$, since $g_\nu$ is not identically zero. This completes the proof of Lemma \ref{asymptoticd+deltaG}.\end{proof}

The last step in the proof of Theorem \ref{mainthRT} is to describe the $L^p$ boundedness of $T_q^{\rm lf}$, using very similar arguments as in Proposition 5.1 of \cite{gh1}. For this purpose we switch back to writing the kernels as multiples of the scattering half-density $|dgdg'|^{\demi}$ and then multiply by $|dgdg'|^{-\demi}$ to remove the density factors and see the kernel as acting on pure forms (which is a more convenient thing to deal with $L^p(M,\Lambda_\sca^q)$ 
spaces defined with respect to the volume density $|dg|$). This has the effect of adding $n/2$ to the index sets at $\rbo$ and $\lbo$ and adding $n$ at $\bfo$: the index set $\mc{A}$ of the Schwartz kernel $DR_q(k)$ (with density removed) satisfies 
\begin{equation}\label{indexAbis} 
\mc{A}_\sca\geq 1, \, \, \mc{A}_\zf\geq 0, \,\, \mc{A}_{\bfo}\geq n-1, \,\, \mc{A}_{\rbo}\geq \mu_{\rbo}+n/2, \,\, \mc{A}_{\lbo}\geq \mu_{\lbo}+n/2.
\end{equation}
\begin{proposition}\label{boundednessT2} 
Let $\alpha=n/2-1-\mu_{\rbo}$ and let $\beta=n/2-1-\mu_{\lbo}$. The operator $T_q^{\rm lf}$ defined in \eqref{Tlf}   bounded on $L^p(M)$ for all $p$ between
\[ \frac{n}{n-\beta_+}<p<\frac{n}{\alpha_+},\] 
where $\alpha_+=\min(\alpha,0)$, $\beta_+=\min(\beta,0)$. Moreover, this range is sharp.
\end{proposition}
\noindent Theorem \ref{mainthRT} then follows immediately from this proposition, the preceding discussion, and Lemma \ref{idleadingorder}.
\begin{proof} 
Throughout, we identify all operators with their Schwarz kernels. Let $\chi_1(z,z')$ be a smooth cutoff function which is 1 on the region where $x\leq\epsilon$, $x\leq\epsilon$, and
$1/4\leq x/x'\leq 4$, and which is 0 outside a small neighborhood of this region. Then define $T_1=\chi_1T_q^{\rm lf}$. The remainder $T_q^{\rm lf}-T_1$ may be split into three pieces: $T_2$, supported in the region where $x<x'/4$; $T_3$, supported in the region where $x>4x'$; and $T_4$, supported in the region where $x\geq\epsilon$ and $x'\geq\epsilon$ (see Figure \ref{cutoffs} for an illustration of this decomposition). We then consider each $T_i$ separately. Note first that the operator with kernel $T_4$ is a compactly supported classical pseudo-differential operator of order $0$, and thus is bounded on all $L^p$ for $1<p<\infty$. 

\begin{figure}\label{cutoffs}
\centering
\includegraphics[scale=0.5]{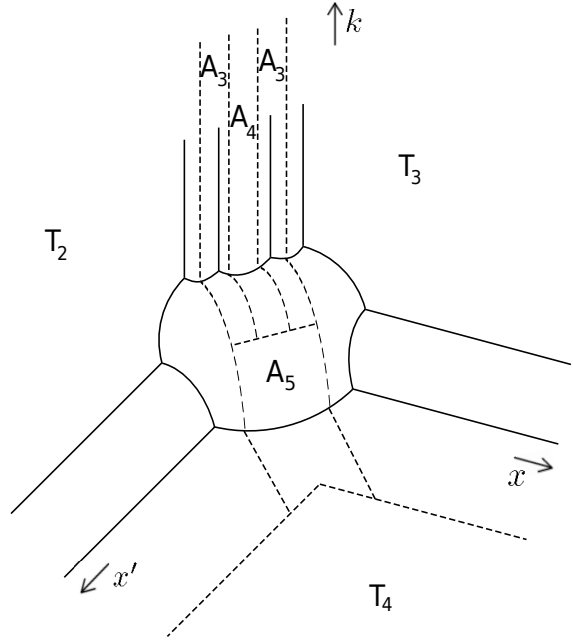}
\caption{Decomposition of integrals: $T_i$ and $A_i$.}
\end{figure}

Next, we obtain pointwise bounds on the kernels of $T_2$ and $T_3$: we claim that there exists $C>0$ such that 
\begin{equation}\label{boundT2T3}
|T_2(z,z')|\leq Cx^{\alpha}x'^{n-\alpha}, \quad |T_3(z,z')|\leq Cx^{n-\beta}x'^{\beta}.
\end{equation}
We prove the statement for $T_3$ only; the proof for $T_2$ is precisely analogous (this is very similar to the proof in Proposition 5.1 of \cite{gh1}, with a slightly improved bound, so we refer there for details). We break the integral \eqref{Tlf} in $k$ into two pieces: the first from $0$ to $x'$ and the second from $x'$ to 1; since we are only considering $T_3$, we may assume $x'<x/4$. 

In the first region of integration, we have $k\leq x'<x$, so we use the coordinates $(k/x', x'/x, x,y,y')$. In these coordinates, using \eqref{indexAbis} shows that the integrand is bounded by
\[Cx^{n-1}(x'/x)^{n-1-\alpha}=Cx^{-\alpha}(x')^{n-1-\alpha}.\]
Integration in $k$ from $0$ to $x'$ gives a bound of $Cx^{-\alpha}(x')^{n-\alpha}$. 

In the second region, we use the same coordinates as in \cite{gh1}: $x'/(x'+k)$ as a boundary defining function for rb, $x+x'+k$ for bf$_0$ and $(x'+k)/(x+x'+k)$ for rb$_0$. Therefore, for any $N$, the second integral is bounded by
\[C_N\int_{x'}^{1}(x+x'+k)^{n-1}\frac{(x'+k)^{n-1-\alpha}}{(x+x'+k)^{n-1-\alpha}}\frac{x'^{N}}{(x'+k)^{N}}dk.\]
Letting $N=2n$ and using $\kappa'=k/x'$, we bound this integral by 
\[Cx^{\alpha}x'^{n-\alpha}\int_{1}^{\infty}(1+\frac{x'}{x}(1+\kappa'))^{\alpha}(1+\kappa')^{-n-1-\alpha}\ d\kappa'.\]
Since $x'<x/4$ on the support of $T_3$, this is bounded (if $\alpha\geq 0$) by
\[Cx^{\alpha}x'^{n-\alpha}\int_{1}^{\infty}(1+(1+\kappa'))^{\alpha}(1+\kappa')^{-n-1-\alpha}\ du=C'x^{\alpha}x'^{n-\alpha},\]
as desired. If $\alpha<0$, then note that $(1+x'/x(1+\kappa'))^{\alpha}\leq 1$, giving an analogous bound. This proves (\ref{boundT2T3}).\\

From (\ref{boundT2T3}), we deduce directly (as in the proof of \cite[Corollary 5.9]{HaLi}) that the operator with kernel $T_2$ is bounded on $L^p$ for $p<n/\alpha$ and the operator with kernel $T_3$ is bounded on $L^p$ for $p>n/(n-\beta)$. By Lemma \ref{idleadingorder}, both $\mu_{\lbo}$ and $\mu_{\rbo}$ are at least $\min(\nu_0-1,\tilde\nu-2)>-1$, so $\alpha$ and $\beta$ are less than $n/2$, hence $T_2$ and $T_3$ are bounded on $L^2$.
As a consequence, $T_1=T-T_2-T_3-T_4$ is also bounded on $L^2$.

To finish the proof of the Proposition, we shall show that $T_1$ is bounded on $L^p$ for all $1<p<\infty$; to do so, we will follow the argument of Hassell-Lin \cite[Sec. 5]{HaLi}.
 It suffices to show that there is a constant $C$ such that
\begin{equation}\label{czbounds}|T_1(z,z')|\leq C|z-z'|^{-n};\ |\nabla_{z}T_1(z,z')|\leq C|z-z'|^{-n-1};\ |\nabla_{z'}T_1(z,z')|\leq C|z-z'|^{-n-1}.\end{equation}
Assuming this claim, we may combine it with the boundedness of $T_1$ on $L^2$ to apply Calderon-Zygmund theory, which shows that $T_1$ is of weak type $(1,1)$, as is its adjoint. Interpolation and duality then show boundedness on $L^p$. 

The proof of \eqref{czbounds} is based on the last three parts of the proof of Proposition 5.1 in \cite{gh1}. Recall from \eqref{indexAbis} that the integrand of $T_1$ (obtained from \eqref{Tlf} by multiplying by $\chi_1$) has orders $n-1$ at bf$_0$, $1$ at sc, $\infty$ at bf, and conormal order $-1$ at the diagonal; since $\nabla_z$ and $\nabla_{z'}$ are first order scattering differential operators, the integrands of $\nabla_zT_1(z,z')$ and $\nabla_{z'}T_1(z,z')$ have orders $n$ at bf$_0$, $2$ at sc, $\infty$ at bf, and conormal order 0 at the diagonal. Since they have the same orders and our proof only depends on those orders, we consider only $\nabla_zT_1$. We break the kernel into three pieces: $A_4$, which is localized near the diagonal and away from zf and bf; $A_5$, which is localized near the diagonal and away from sc and bf (say that it has support where $k\leq x$); and $A_3$, which is localized near sc but away from the diagonal. Again, see Figure \ref{cutoffs}.

For $A_3$, we let boundary defining functions in the support of $A_3$ be $\rho_{\bf}=(1+k^2|z-z'|^2)^{-1/2}$ and $\rho_{\bfo}=k$. The integrals of $|T_1|$ and $|\nabla_zT_1|$ are therefore respectively bounded by
\[\int_0^1(1+k^2|z-z'|^2)^{-N}k^{n-1}\ dk, \quad\int_0^1(1+k^2|z-z'|^2)^{-N}k^{n}\ dk.\]
Changing variables in the integral to $u=k|z-z'|$ and choosing $N$ large then gives bounds of $C|z-z'|^{-n}$ and $C|z-z'|^{-n-1}$ respectively (the region of integration may vary with $|z-z'|$ but is always contained in $[0,\infty)$).

For $A_4$, we let $W=k(z-z')$ be a defining function for the diagonal and $k$ as a defining function for $\bfo$; we only need order 0 at sc (even though we have orders 1 and 2 respectively). Plugging these in, we get bounds of
\[\int_0^1k^{n-1}h_1(|W|)dk,\quad \int_0^1k^{n}h_2(|W|)dk,\]
where $h_1(t)$ and $h_2(t)$ are both exponentially decreasing for large $t$ and are bounded by $Ct^{-n+1}$ and $Ct^{-n}$ respectively for $t$ small. Changing variables to integrate in $W$ and using the same argument as for $A_3$ then gives bounds of $C|z-z'|^{-n}$ and $C|z-z'|^{-n-1}$ respectively.

Last, for $A_5$, we may use $x(z-z')$ as a defining function for the diagonal and $x$ as a defining function for bf$_0$. Our integrals are bounded respectively by
\[\int_0^xx^{n-1}h_1(x|z-z'|)\ d\lambda=x^nh_1(x|z-z'|);\ \int_0^xx^{n}h_2(x|z-z'|)\ d\lambda=x^{n+1}h_2(x|z-z'|).\]
Writing $t=x|z-z'|$ gives bounds of
\[|z-z'|^{-n}t^nh_1(t);\ |z-z'|^{-n-1}t^{n+1}h_2(t).\]
Since $t^nh_1(t)$ and $t^{n+1}h_2(t)$ are globally bounded, \eqref{czbounds}, and hence the boundedness part of Proposition \ref{boundednessT2}, follows. \\

It remains to show that the range is sharp. However, this is a direct consequence of Lemma  \ref{asymptoticd+deltaG}. Indeed, the asymptotic of the Schwartz kernel $T(z,z')$ of the Riesz transform (as a half-density) as $z'\to \infty$ is obtained by using the expression 
\[ T_q(z,z')=\frac{2}{\pi}\int_{0}^\infty DR_q(k;z,z')dk,\]
where the resolvent kernel $R_q(k;z,z')$ is now viewed as a scattering half-density. This amounts to multiplying the kernel 
of $R_q(k)\in\Psi_k^{-2, (-2, 0, 0), \mc{R}}(M, \Lambda_{\sca}^q\otimes \Omega_b^\demi)$ by
$(xx')^{n/2}\otimes |\frac{dk}{k}|^{-\demi}$. For $z\in M^\circ$ fixed, we easily see that as $x'\to 0$
\begin{equation}\label{integratedthing} 
T_q(z,z')\sim \frac{2}{\pi}{x'}^{\mu_{\rbo}+1+\ndemi}x^\ndemi\Big(\int_{0}^\infty {\kappa'}^{\mu_{\rbo}}DG_{\rbo}^{\mu_{\rbo}}(z,\kappa',y')d\kappa'\Big) |dg(z)dg(z')|^\demi,
\end{equation}
which, by Lemma \ref{asymptoticd+deltaG}, does not vanish identically in $z$. Therefore, by the same argument given at the end of \cite[Sec.5]{gh1}, $T_q$ cannot be bounded on $L^p$ for 
$p\geq n/\alpha_+$ (note that $\mu_{\rbo}+1+\ndemi=n-\alpha$); in particular it will not act boundedly on a function of the form $(\log x)^{-1}x^{\alpha}$ near the boundary. The same argument for the asymptotic $z\to \infty$ involves the second integral in Lemma 
\ref{asymptoticd+deltaG}, and an analogous argument shows that $T_q$ is not bounded on $L^p$ for $p\leq n/(n-\beta)_+$. This completes the proof of Proposition \ref{boundednessT2}.
\end{proof}

\subsection{Proof of Corollary \ref{euclidean}}

Now we prove Corollary \ref{euclidean}, which follows more or less immediately from \eqref{casrn}. It suffices to apply the result of Theorem \ref{first}, together with the values $\nu_0=n/2-1$ and 
$\nu_d=\nu_\delta=n/2$ for $q\notin\{0,n\}$. Notice that the map $H^q(\bbar{M},\pl\bbar{M})\to H^q(\bbar{M})$ is always injective 
if $q\in[2,n/2]$ because for those $q$, $H^{q-1}(\sph^n)=0$. We also use the fact that by the maximum principle (and Poincar\'e duality), $\ker_{L^2}(\Delta_q)=0$ when $q\in\{0,n\}$.

\section{Sobolev estimates}\label{sec:Sobolev}

In this section, we use the construction of the resolvent to prove some Sobolev estimates for $q$-forms. For $k\in\nn$ and $p\in(1,\infty)$, we define $W^{k,p}$ to be the Sobolev space consisting of $q$-forms for which applying up to $k$ Lie derivatives in the direction of scattering vector fields keeps the forms in $L^p$; the measure is $dg$. Since scattering vector fields are precisely those in $x\mc{V}_b(M)$,
\[ W^{k,p}:=\{u\in C^{-\infty}(M,\oplus_{q=0}^n\Lambda^q(M)); \mc{L}_{X_1}\dots\mc{L}_{X_j}u\in
L^p, \, \forall j\leq k,\, \forall X_i\in x\mc{V}_b(M)\}.\]  
When $p=2$, we denote $H^k:=W^{k,2}.$
\begin{theorem}\label{SobolevTh}
Let $(M,g)$ be asymptotically conic to order $n_0\geq 3$. Assume that $\nu_0>0$ and that there are no zero-resonances. We also assume that $\nu_{\ker}>2$ if $\nu_{\ker}$ is defined by \eqref{defofmth}. Let $p=2n/(n+2)$ and $p'=2n/(n-2)$;
then there exists $C>0$ such that for all $q$-forms $u\in W^{2,p}$ and $v\in H^1$,
\begin{equation}\label{Sobolevest} 
||({\rm Id}-\Pi_{\ker(\Delta_q)})v||_{L^{p'}}\leq C||Dv||_{L^2}, \quad 
||({\rm Id}-\Pi_{\ker(\Delta_q)})u||_{L^{p'}}\leq C||\Delta_q u||_{L^p}
\end{equation}
where $\Pi_{\ker(\Delta_q)}$ is the orthogonal projector on $\ker_{L^2}(\Delta_q)$ in $L^2$.
\end{theorem}
Notice that the condition $\nu_{\ker}>2$ is the condition that $\ker_{L^2}(\Delta_q)\subset L^r$ for all $r\geq p=2n/(n+2)$; it is required to make sense of $({\rm Id}-\Pi_{\ker(\Delta_q)})u\in L^{p'}$ when $u\in H^1$ or $u\in W^{2,p}$.
\begin{proof} First consider the integral kernel given by $G_{\zf}^0$ in \eqref{Gzf}.
Viewing it as a scattering half-density on zf, it is a conormal distribution with leading orders $n-2$ at $\bfo$, interior conormal order $-2$ at the diagonal, 
and order $n/2+\nu_{\ker}-3$ at both $\rbo$ and $\lbo$. We prove first that there is $C>0$ such that for all 
$u$ compactly supported
\[||G_{\zf}^0u||_{L^{p'}}\leq C||u||_{L^p}.\]
To prove this, it suffices to get pointwise estimates on this kernel and use the Hardy-Littlewood-Sobolev result. We follow the arguments of \cite{HaLi} to get pointwise bounds: the kernel $G_{\zf}^0$ can be split into $G_d+G_L+G_R$, where $G_1$ is supported near the diagonal in the region $x/x'\in[1/4,4]$ (its support intersects only the boundary $\bfo$), $G_L$ is supported in $x/x'\leq 1/2$ (its support intersects only the boundaries $\lbo$ and $\bfo$) and $G_R$ is suported in $x'/x\leq 1/2$ (its support intersects only the boundaries $\rbo$ and $\bfo$). By Lemma 5.3 and 5.4 in \cite{HaLi}, we have a pointwise estimate for the kernel (with the densities removed) outside the diagonal:  
\[ |G_d(z,z')|\leq Cd(z,z')^{-n+2}, \] 
where $d(\cdot,\cdot)$ is the Riemannian distance. In particular, applying the Hardy-Littlewood-Sobolev result of \cite{GG}, we deduce that $G_d$ is the kernel of a bounded operator from $L^p$ to $L^{p'}$ if $p=2n/(n+2)$. 
Now the kernel $G_R$ has pointwise norm bounded by $C\rho_{\bfo}^{n-2}\rho_{\rbo}^{\nu_{\ker}-3}$, where $\rho_{\bfo},\rho_{\rbo}$ are boundary defining functions of $\bfo$ and $\rbo$ respectively on $\zf$. Since $x'/x$ and $x$ are such boundary defining functions in the region where $G_R$ is supported, the kernel of $G_R$ is bounded by $C{x'}^{\nu_{\ker}}x^{n-2}1_{[0,1/2]}(x'/x)$. This kernel is then easily seen (as in \cite{HaLi}) to be bounded as a map $L^p\to L^{p'}$ as long as $\nu_{\ker}-3>-1$. The same argument applies to $G_L$, and this proves the boundedness of $G_{\zf}^0$.

To prove the Sobolev estimate for $\Delta_q$, we take $u\in C_0^\infty(M;\Lambda_{\sca}^q)$. We have 
\[ ({\rm Id}-\Pi_{\ker(\Delta_q)})u=G_{\zf}^0\Delta_q u\] 
by the construction of $G_{\zf}^0$ in Section \ref{subsect:zf}; this equality makes sense when 
$u$ is compactly supported. Therefore, using the $L^p\to L^{p'}$ bound on $G_{\zf}^0$, we get 
\[ ||({\rm Id}-\Pi_{\ker(\Delta_q)})u||_{L^{p'}}\leq C||\Delta_q u||_{L^p}\]
and by density, this shows the second estimate of \eqref{Sobolevest}.  
 
Finally, we have seen in the proof of Proposition \ref{basicinfo} that $DR_q(k)u=DG_{\zf}^0u+o(1)$ as $k\to 0$ for all $u\in C_0^\infty(M;\Lambda_{\sca}^q)$ compactly supported.
Moreover Proposition \ref{basicinfo} shows that the kernel of $DG_{\zf}^0$ (living on $\zf$) has order bounded below by $n-1$ at $\bfo$, interior conormal order $-1$ at the diagonal, $n/2+\nu_{\ker}-2$ at $\lbo$ and $n/2+\nu_{\ker}-3$ at $\rbo$. Splitting the kernel into near-diagonal and off-diagonal terms as before, we see that the near-diagonal term is bounded from $L^2$ to $L^{p'}$ by the Hardy-Littlewood-Sobolev result in \cite{GG}, and the off-diagonal terms are bounded by the same argument as above. This shows that $||DG_{\zf}^0||_{L^2\to L^{p'}}\leq C$ for some $C>0$. Since moreover $R_q(k)D=DR_q(k)$ for $k>0$ by the $L^2$-functional calculus, we deduce that the same holds at the limit $k=0$ when acting on compactly supported functions; that is, $DG_{\zf}^0u=G_{\zf}^0Du$ for $u\in C_0^\infty(M;\Lambda_{\sca}^q)$.
Then for such $u$,  $({\rm Id}-\Pi_{\ker(\Delta_q)})u=G_{\zf}^0D^2u=DG_{\zf}^0Du$,
and by using the $L^2\to L^{p'}$ bound of $DG_{\zf}^0$ we obtain 
\[ ||({\rm Id}-\Pi_{\ker(\Delta_q)})u||_{L^{p'}}\leq C||Du||_{L^2}.\]
By density, this shows the first estimate of \eqref{Sobolevest}.
 \end{proof}

Notice that the condition $\nu_{\ker}>2$ is equivalent to asking that $\ker_{L^2}(\Delta_q)=\ker_{L^r}(\Delta_q)$ 
for all $r\in [2n/(n+2),2]$, by the regularity result of Theorem \ref{reg} on solutions of $\Delta_qu=0$. Combining with 
the fact that the absence of zero-resonances is equivalent to \eqref{kerLr}, the statement 
of Theorem \ref{mainthSob} in the introduction is thus the same as Theorem \ref{SobolevTh}. We also notice that when 
$\pl\bbar{M}=\sph^n$, any $u\in \ker_{L^2}(\Delta_q)$ satisfies $Du=0$; by considering the indicial roots of 
$D$ in \eqref{indrootofD} together with \eqref{caseSn}, we see that $u\in x^{n}L^\infty$ for $q\in[1,n-1]$ (and $u=0$ if $q=0$ or $n$), and therefore $\nu_{\ker}>2$ in the asymptotically Euclidean case. 

\section{Analytic Torsion and the proof of Theorem \ref{conictorsion}}\label{section:proof}
In this section, we prove Theorem \ref{conictorsion} by following the method introduced by the second author \cite{s1,s2}. Once we have the result on the resolvent behavior at low energy, the proof is rather similar, and we just adapt the arguments as needed to our case. We consider the geometric setting of a family of smooth compact Riemannian manifolds $(\Omega_\eps,g_\eps)$ degenerating as $\eps\to 0$ to a compact manifold with a conic singularity $(\Omega_0,g_0)$ with cross section $(N,h_0)$, as explained in the Introduction (here $\eps\in (0,\eps_0]$ for some $\eps_0>0$ small). The non-compact manifold $(M,g)$ with exact conic end from which $\Omega_\eps$ is constructed has also cross section $(N,h_0)$. We denote by $\Delta^{\Omega_\eps}_q$, $\Delta_q^{\Omega_0}$ and $\Delta_q^M$ the Laplacians on $\Omega_\eps$, $\Omega_0$, and $M$ acting on $q$-forms, and as above, $\Delta_N$ is the Laplacian acting on the whole form bundle $\Lambda(N)=\oplus_{p=0}^{n-1}\Lambda^p(N)$. 

The key to describe the analytic torsion of $\Omega_\eps$ as $\eps\to 0$ is to obtain a 
detailed analysis of the heat trace on $q$-forms on $\Omega_\epsilon$, which we denote 
$\tra(H_q^{\Omega_\epsilon})(t)$ (here and below we use the notation $H_q^{Z}(t):=e^{-t\Delta_q^{Z}}$ when $Z\in \{\Omega_\eps,\Omega_0,M\}$). 
The proof proceeds in two steps. First we shall prove Theorem \ref{structure}, which expresses 
$\tra(H_q^{\Omega_\epsilon})(t)$ in terms of the heat traces on $\Omega_0$ and $M$. This will be proved in part by using the work of Section \ref{resolvent-kernel} describing the low-energy resolvent on $M$. 
Then we shall use what we know about the heat trace on $\Omega_0$ and on $M$ 
to directly analyze the spectral zeta function
\[\zeta_q^{\Omega_{\epsilon}}(s)=\frac{1}{\Gamma(s)}\int_0^{\infty}(\tra(H_q^{\Omega_\epsilon})(t)-\dim \ker(\Delta^{\Omega_{\epsilon}}_q))t^{s-1}\ dt.\]
This second step depends on a scaling argument on exact cones and also on a spectral convergence result of Ann\'e and Takahashi \cite{at}. From there, we proceed to the determinants of the Laplacians and the analytic torsion. 
We will assume that $n$ is odd since the analytic torsion 
vanishes in even dimension, but it is easily seen from the proof below that the results about 
asymptotic expansions of the determinants of the Laplacians on $q$-forms in even dimension work similarly. There is one difference: in that case, the expansion of $\det (\Delta_q^{\Omega_\eps})$ will contain a $\log(\eps)^2$ coefficient coming from the  $\log(t)$ coefficient of $\tra(H^{\Omega_0}_q(t))$ as $t\to 0$ (see for instance \cite{s2} for the case $q=0$).

\subsection{Consequences of the modified Witt condition}
We first show that the Witt condition rules out zero-resonances on $M$ and thus allows us to apply the results of Section \ref{resolvent-kernel} to study the resolvent $R^M_q(k)$.

\begin{lemma}\label{wittconseq} 
Assume $M$ satisfies the modified Witt condition \eqref{Wittcond}. Then $0\notin\mc{I}(\Delta^M_q)$, where $\mc{I}(\Delta^M_q)$ is the indicial set defined by \eqref{indicial}. 
Moreover, $\Delta_q^M$ has no zero-resonance for any $q$.
\end{lemma}
\begin{proof}
To prove the lemma, first suppose for contradiction that $0\in \mc{I}(\Delta_q^M)$. Since $n$ is odd, we see by examining the indicial set in \eqref{I1I2} and \eqref{i3} that $0\in I^{3}$. Moreover, $q$ must be either $\frac{n+1}{2}$ or 
$\frac{n-1}{2}$; in either case, $\alpha^{2}=3/4$, so $3/4\in S^{q}_{d_N}$.
But $S^{q}_{d_N}=S^{q-1}_{\delta_N}$, so in either case, $3/4\in {\rm Sp}_{\Lambda^{(n-1)/2}}(\Delta_N)$. This contradicts the modified Witt assumption, and therefore $0\notin \mc{I}(\Delta^M_q)$.

Now we show that $\Delta_q^M$ has no zero-resonances. 
Recall that a zero-resonance is a $q$-form $u$ (valued in b-half-densities) satisfying $\Delta_q^Mu=0$ which is in $x^{-1}H_b^0\setminus H_b^0$. Such a form must be smooth in $M$ by elliptic regularity, and by Theorem \ref{reg}, it is polyhomogeneous conormal on $\bbar{M}$; in fact, one has $u\in x^{\nu_0-1-}H^2_b$ with $\nu_0>0$ and $u\sim x^{\nu}u_\nu(y)$ for some $\nu\geq \nu_0$ and $u_\nu\in E_{\nu}$. 
The first step is to show that such a form must satisfy $du=0$ and $\delta u=0$, which follows immediately as long as the integration by parts
\[\ip{\Delta^M_q u}{u}=\ip{du}{du}+\ip{\delta u}{\delta u}\]
is justified. Since $x^{-1}d$ and $x^{-1}\delta$ are b-differential operators of order $1$, one has $du\in x^{\nu_0-}H^1_b$ 
and $\delta u\in x^{\nu_0-}H^1_b$. We switch back to scattering half-densities and factor out 
the $|dg|^{\demi}$ factor to view $u,du,\delta u$ as forms:
we get 
\[u\in x^{\ndemi +\nu_0-1}L^\infty(M,\Lambda^{q}_\sca), \,\, 
du\in x^{\ndemi+\nu_0}L^\infty(M,\Lambda^{q+1}_\sca),\,\,  
\delta u\in x^{\ndemi+\nu_0}L^\infty(M,\Lambda^{q-1}_\sca).\] 
Then 
we integrate by parts on the region $M_\eps=\{x\geq \eps\}$ for $\eps>0$ small and let $\eps\to 0$.
The boundary term in the integration by parts is
\begin{equation}\label{ibp}
\int_{\pl M_\eps}\iota_{\pl M_\eps}(\delta u\wedge *u-u\wedge *du).
\end{equation}
and this integral goes to zero as $\eps\to 0$: indeed each term in the integrand is $\eps^{n+2\nu_0-1}$ times 
a bounded element of $\Lambda^{n-1}_\sca$ pulled-back by the inclusion $\iota_{\pl M_\eps}$, thus the integral is $\mc{O}(\eps^{2\nu_0})$. 
We conclude that any zero-resonance is also in the kernel of the operator $D=d+\delta$. This means that $\nu$ must be an indicial root of $D=d+\delta$, i.e. $\nu$ belongs to $\mc{I}(d)\cap \mc{I}(\delta)$ as defined in 
\eqref{indicddelta}. These roots are given by \eqref{indrootofD}, and under the modified Witt condition  we see that $\min(\mc{I}(d)\cap \mc{I}(\delta)\cap \rr^+)\geq 3/2$, which implies that $\nu\geq 3/2$. Consequently, $u\in L^2(M,\Lambda^q_{\sca})$ (or $u\in H_b^0$ when viewed as a b-half-density), which is a contradiction.
\end{proof}
\begin{remark}\label{noresonanceneven}
We notice that the proof above also works similarly in even dimension: for $q=n/2-1$, the Laplacian 
$\Delta_q^M$ has no resonance if $H^{\ndemi-1}(N)=0$; for $q=n/2+1$, it has no resonance if $H^{\ndemi}(N)=0$; for $q=n/2$, it has no resonances if  ${\rm Sp}_{\Lambda^{n/2}}(\Delta_N)\cap [0,1]=\emptyset$; and when $|q-n/2|>1$, the Laplacian never has resonances.
\end{remark}
As a consequence, the manifold $M$ satisfies the hypotheses of Theorem \ref{thm:nullspace} and therefore the 
 low-energy resolvent $R^M_q(k)$ of the Laplacian on $q$-forms has the structure described there.
 
\subsection{Heat trace structure theorem} 

The key ingredient in the proof of Theorem \ref{conictorsion} is a structure theorem for the trace of the heat kernel on $q$-forms on $\Omega_{\epsilon}$, as a function of $(t,\eps)$. In fact, it is more convenient to work with the variables 
$(\sqrt t,\epsilon)$. We consider $Q=[0,1]_{\sqrt{t}}\x [0,\eps_0]_{\eps}$ where $\eps_0>0$ is small, and we let $Q_0$ be the manifold with corners obtained by radially blowing up the corner $\sqrt{t}=\eps=0$ in $Q$; see Figure \ref{qnaught}. Functions which are smooth (or polyhomogeneous conormal) on $Q_0$ correspond to functions on $Q$ which are smooth (or polyhomogeneous conormal) when written in polar coordinates around $(0,0)$.
Denote the face created by this blowup by $F$ (projecting to $(0,0)$ by the blow-down map);
the face projecting to $\{\sqrt t=0\}$ is denoted by $L$, and the one projecting to $\{\epsilon=0\}$ is called $R$. 

We shall use the arguments of \cite{s2} to prove that the heat trace $\tra(H_q^{\Omega_\epsilon})(t)$ is polyhomogeneous conormal on $Q_0$.
Let $\chi_2(z)$ be a smooth radial cutoff function on $\Omega_\epsilon$, equal to zero for $r\leq 15/16$ and to one for $r\geq 17/16$, and non-decreasing in $r$; moreover, let $\chi_1(z)=1-\chi_2(z)$. 
We explain how Theorem 3 in \cite{s2} can be extended from functions to the following 
case of general $q$-forms:

\begin{theorem}\label{structure} 
The heat trace ${\rm Tr}(H_q^{\Omega_\eps}(t))$ is polyhomogeneous conormal on $Q_{0}$ 
with leading orders $-n$ at $L$ and $F$ and leading order $0$ at $R$. Moreover, one has
\begin{equation}\label{structureeq}
{\rm Tr}(H_q^{\Omega_{\epsilon}}(t))=
\int_{M}\chi_1(\epsilon z){\rm Tr}(H_q^{M}(\frac{t}{\epsilon^2},z,z))dg(z)+\int_{\Omega_0}\chi_2(z){\rm Tr}(H_q^{\Omega_0}(t,z,z))dg_0(z)+R(\epsilon,t),
\end{equation}
where $R(\epsilon,t)$ is polyhomogeneous conormal on $Q$ with infinite-order decay at $t=0$. In particular, for all $k\in\nn$ there is $C_k\geq 0$ such that for all $\eps>0$ small and $t\leq 1$, one has $|R(\eps,t)|\leq C_kt^{k}$.
\end{theorem}

\begin{figure}
\centering
\includegraphics[scale=0.70]{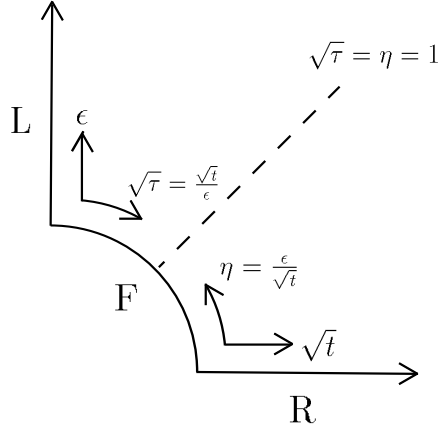}
\caption{The blown-up space $Q_0$.}\label{qnaught}
\end{figure}

\begin{proof}
We only describe the slight changes to the proof in \cite[Sec. 4]{s2} which are needed to extend it from functions 
to $q$-forms.
The main point is to obtain a uniform parametrix construction for the heat kernel on $\Omega_\epsilon$, and in fact Section 4 of \cite{s2} concerning the parametrix construction was written with this situation in mind. 
The argument requires precise descriptions of the asymptotic structure of the heat kernels on $\Omega_0$ and on $M$. 
The polyhomogenous structure of the integral kernel of 
$\chi_2 H_q^{\Omega_0}(t)\chi_2$ for any $\chi_2$ with support not intersecting the cone tip 
is precisely the same in the general-$q$ case as it is in the $q=0$ case; the kernel is polyhomogeneous conormal
on the space $\Omega^2_{0,{\rm heat}}$ obtained by blowing up $\Omega_0^2\x\{\sqrt{t}=0\}$ inside $\Omega_0^2\x[0,1]_{\sqrt t}$.
This follows from the general heat calculus for second order elliptic self-adjoint differential operators on closed manifolds (note that Mooers \cite{mo} also describes the behaviour near the cone tip).

The heat kernel on $M$ for short time is also polyhomogeneous conormal on a certain space: we set
\[M^2_\sca:=\MMksc\cap \{k=1\},\quad {\rm diag}_{\sca}:={\rm diag}_{k,\sca}\cap\{k=1\}\]
with $\MMksc$ defined as in Section \ref{mmksc} and ${\rm diag}_{k,\sca}$ as in Section \ref{opecal}. Then as a direct consequence of the work of Albin \cite{alb}, (see \cite[Appendix]{s1} for details\footnote{We notice that Albin's heat calculus does not require $q=0$ (in fact, it allows general vector bundle coefficients), and neither does the argument in the appendix of \cite{s1} which completes the construction of the heat kernel. The leading orders of the heat kernel on $M$ for time bounded above are also the same as in the $q=0$ case.}) we have the following
\begin{proposition}\label{zshorttime} 
Let $T>1$ fixed, then for $t<T$, the heat kernel $H^{M}_q(t)$ for $q$-forms on $M$ is polyhomogeneous conormal on the space
\begin{equation}\label{M2heat}
M^2_{\rm heat}:=[M_{{\rm sc}}^{2}\times[0,T]_{\sqrt t}; {\rm diag}_{\rm sc}\x \{\sqrt t=0\}]
\end{equation}
obtained by blowing-up $M_{{\rm sc}}^{2}\times[0,T]_{\sqrt t}$ at the submanifold 
${\rm diag}_{\rm sc}\x \{\sqrt t=0\}$.
Moreover there is infinite-order decay at all faces except the scattering face ${\rm sc}$ (where the leading order is zero) and the temporal face ${\rm tf}$ obtained from the blowup (where the leading order is $-n$ in terms of $\sqrt t$).
\end{proposition}

In order to understand the long-time heat kernel on $M$, we apply the analysis of the low-energy resolvent and the methods in \cite[Th. 2]{s1}, based on the expression of the heat operator as a contour integral of the resolvent. 
We shall show that the space on which the heat kernel is polyhomogeneous is the same for general $q$-forms, 
but that the leading orders of the heat kernel at the boundary hypersurfaces may change:
\begin{proposition}\label{maincor} 
For any fixed time $0<T<1$ and any $q$ with $0\leq q\leq n$, $H_q^M(t)$ is polyhomogeneous conormal on 
$M^2_{\omega,sc}$ for $t>T$, where $\omega:=t^{-1/2}$. The leading orders at the boundary hypersurfaces are bounded below by $0$ at ${\rm sc}$, by $n$ at ${\rm bf}_0$, by $0$ at ${\rm zf}$, and by $n/2-1+\nu_{\ker}$ at ${\rm lb}_0$ 
and ${\rm rb}_0$. Moreover, the coefficient of the zeroth order term at ${\rm zf}$ (if any) 
decays to order strictly greater than $n$ at ${\rm bf}_0$.
\end{proposition}

\begin{proof} By Lemma \ref{wittconseq},  we can apply and Theorem \ref{thm:nullspace} and Remark \ref{rem5}, which show that 
the Schwartz kernel of $(\Delta_q^M+k^2e^{i\theta})^{-1}$ is phg conormal on $M^2_{k,\sca}$ 
for each fixed $\theta \in (-\pi/2,\pi/2)$, and with smooth dependence in $\theta$. 
The leading orders of the resolvent, with respect to the half-density bundle $\til{\Omega}_b^\demi(\MMksc)$, 
are $0$ at $\sca$, $-2$ at $\bfo$ and $\zf$, and $\nu_{\ker}-3$ at $\rbo$ and $\lbo$. 
Switching to the natural scattering half-densities adds $n/2$ at $\lbo$ and $\rbo$ and $n$ at $\bfo$, giving $n-2$ at 
$\bfo$ and $n/2+\nu_{\ker}-3$ at $\lbo$ and $\rbo$. Additionally, from the construction in Section \ref{subsect:zf}, the leading-order term $G^{-2}_{\zf}$ at $\zf$ decays to higher order at $\bfo$; with respect to scattering half-densities, it decays to order strictly greater than $n-2$, which we write $n-2+\delta$ for some $\delta>0$.
Theorem 8 in \cite{s1} also requires that the spectrum of $\Delta_q^M$ be contained in $[0,\infty)$, which it is, and that the high-energy resolvent be polynomially bounded in $|\lambda|$ for $\theta$ with $|\theta|>\pi/4$. As in \cite{s1}, the polynomial bounds on the high-energy resolvent are a direct consequence of the usual semiclassical scattering calculus and just comes from ellipticity (see \cite{vz,wz} and section 10 of \cite{hw}). The conditions of Theorem 8 in \cite{s1} are therefore met. The polyhomogeneity and leading-order statements in Theorem \ref{maincor} then follow immediately from the application of \cite[Th. 8]{s1}. Note that in particular, the transition from resolvent to heat kernel adds $2$ to the orders at zf, bf$_0$, lb$_0$, and rb$_0$.

Finally, we must show that the coefficient of the order-0 term at zf decays to higher order at bf$_0$. To do this, we note that the corresponding property is true for the low-energy resolvent, so we may break off the leading order term. Write the resolvent as a sum of two kernels, $R_q^M(k)=R_1+R_2$, where $R_1$ consists precisely of the leading-order term at zf. $R_1$ then has order $-2$ at zf and order $n-2+\delta$ at bf$_0$, while $R_2$ has order greater than $-2$ at zf and order 0 at bf$_0$. Applying Lemma 10 of \cite{s1} separately to the two terms, we see that $H_q^M(t)=H_1(t)+H_2(t)$, where $H_1(t)$ has order 0 at $\zf$ and order $n+\delta$ at $\bfo$, while $H_2(t)$ 
has positive order at $\zf$ and order $n$ at $\bfo$. Therefore, the order-0 term of $H_q^M(t)$ at $\zf$ is the same as the order-0 term of $H_1(t)$ at $\zf$, and hence decays to order $n+\delta$ at $\bfo$. This completes the proof of Proposition \ref{maincor}.
\end{proof}

The proof of Theorem 3 in \cite[Sec. 4]{s2} is done in a general setting by simply assuming that the heat kernel  
is polyhomogeneous conormal on $M^2_{\omega,\sca}$ for large time and on $M^2_{\rm heat}$ for small time, 
with lower bounds on the index sets at the boundaries of these spaces. The small time behaviour for forms is the same as for functions, so the argument for small time is unchanged. The necessary large time assumption is written down in Section 4.2 of \cite{s2} and called the ``Alternative Hypothesis.'' It requires the following lower bounds on the orders of $M^2_{\omega,\sca}$ (with $\omega=t^{-1/2}$) for $H^M_q(t)$: $0$ at $\zf$ and $\ndemi+\eps$ for some $\eps>0$ at $\rbo,\lbo$. Moreover, the coefficient of the term of order $0$ at $\zf$ (if any) must decay to order strictly greater than $n$ at $\bfo$. 
Thus Proposition \ref{maincor} implies that $H_q^M(t)$ satisfies the Alternative Hypothesis and we can apply the result of Section 4.2 of \cite{s2}. Indeed, nothing in the argument in \cite{s2} requires $q=0$; the parametrix construction for the heat kernel is the same for differential forms and for functions, once the definition of t-convolution is generalized in the usual way to operators on forms. The only other key property used is the conformal scaling of the heat kernel:
\[H_q^{\epsilon Z}(t,z,z')=\epsilon^{-n}H_q^Z(\frac{t}{\epsilon},\frac{z}{\epsilon},\frac{z'}{\epsilon}).\] As discussed for example in \cite{ch2}, this holds for $q>0$ as well as $q=0$. This completes the proof of Theorem \ref{structure}.
\end{proof}

\subsection{Renormalized heat trace on $M$}

We now need to analyze each of the terms in Theorem \ref{structure}. We begin by studying
\begin{equation}\label{firstthing}
\int_{M}\chi_1(\epsilon z)\tra(H^M_q(\tau,z,z))dv(z), 
\end{equation}
where $\tau=t/\epsilon^2$. This was studied in Section 3.1 of \cite{s1} for the case of functions. Since we have the polyhomogeneous conormality of the heat kernel on 
$M^2_{\omega,\sca}$ for large time and on $M^2_{\rm heat}$ for small time, the same proof 
extends verbatim to the case of $q$-forms: this gives 
\begin{proposition} \label{polyhomof69}
As a function of $(\tau,\eps)$, the integral (\ref{firstthing}) has a polyhomogeneous conormal expansion in $[0,1]_\tau\x [0,1]_\eps$, and  in the space $[ [0,1]_{\tau^{-1/2}}\x [0,1]_{\eps}; \tau^{-1/2}=\eps=0]$ obtained by blowing-up
$\tau^{-1/2}=\eps=0$ inside the square $[0,1]_{\tau^{-1/2}}\x [0,1]_{\eps}$. As a consequence, the finite part as $\eps\to 0$, denoted 
\begin{equation}\label{finsmooth}{\rm FP}_{\eps\to 0}\int_{M}\chi_1(\epsilon z){\rm Tr}(H_q^M(\tau,z,z))dg(z),\end{equation}
exists and has polyhomogeneous asymptotic expansions in $\tau$ at both $\tau=0$ and $\tau=\infty$.
\end{proposition}
As in \cite{s1}, a similar argument applies to the integral
\[\int_{\{x\geq\eps\}}{\rm Tr}(H_q^M(\tau,z,z))dg(z),\]
and hence the renormalized heat trace  $^{R}\tra(H^M_q(\tau))$ defined by \eqref{renormtrace} exists and has polyhomogeneous expansions in $\tau$ at $\tau=0$ and $\tau=\infty$.

It will be necessary to be more specific about the asymptotic expansion of (\ref{firstthing}) as $\eps\rightarrow 0$. In fact, we can explicitly compute most of the divergent terms. By \cite{ch2}, the local trace of the heat kernel on the exact metric cone $C_N=(0;\infty)_r\x N$ with metric $dr^2+r^2h_0$ has a local asymptotic expansion away from the conic points and away from infinity: 
for $z\in C_N$,
\[\tra(H_q^{C_N}(\tau,z,z))\sim\sum_{k=0}^{\infty}u_k(z)\tau^{(k-n)/2} \textrm{ as }\tau\to 0\]
where the $u_k(z)$ are functions on $C_N$ which are zero for all odd $k$. Since $n$ is odd,  there is no $\tau^{0}$ coefficient in the 
heat expansion. For each $k\in[0,n-1]$, let
\[f_k(\tau)=\frac{\tau^{(k-n)/2}}{k-n}\int_Nu_k(1,y)\ dh_0(y), \quad l_{k}=-\int_{1/2}^{2}\chi_{1}'(r)r^{k-n}\ dr.\]
Then we have the following expansion, which is precisely analogous\footnote{There is however a difference with \cite[Sec. 3.2]{s1}: since we consider $n$ odd, there is no $\log \eps$ term in the expansion and $^{R}\tra(H^M_q(\tau))$ is the only $\eps^0$ coefficient.} 
to \cite[Lemma 19]{s1}:
\begin{proposition}\label{expansion} 
For each fixed $\tau>0$,
\[\begin{split}
\int_{M}\chi_1(\epsilon z){\rm Tr}(H^{M}_q(\tau,z,z)) dg(z)= & 
\sum_{k=0}^{n-1}l_{k}f_{k}(\tau)\epsilon^{k-n}
+ {^R}{\rm Tr}(H^M_q(\tau))+\tilde R(\epsilon,\tau)
,\end{split}\]
where $\tilde R(\epsilon,\tau)$ goes to zero as $\epsilon$ goes to zero.
\end{proposition}
\begin{remark} Note that this proposition shows \eqref{finsmooth} is equal to $^{R}\tra(H^M_q(\tau))$, and therefore, as in \cite[Prop. 20]{s1}, the renormalized trace is independent of the cutoff function used to define it. \end{remark}

\begin{proof}
As in \cite[Lemma 19]{s1}, the proof proceeds by comparing the heat kernel on $M$ 
with the heat kernel on $C_N$. The first step is the following comparison lemma: we identify 
a region $\{z\in M; x(z)\leq \delta\}$ with $\{z\in C_N; r(z)\geq 1/\delta\}$ for some $\delta>0$ small enough and we have
\begin{lemma} \label{comparison}
For any fixed $t>0$ and any $q$ between 0 and $n$, as $|z|\rightarrow\infty$, the pointwise norm of the difference $H^M_q(t,z,z)-H^{C_N}_q(t,z,z)$ decays faster than any polynomial in $|z|$.
\end{lemma}
\begin{proof} This is shown for $q=0$ in \cite[Lemma 16]{s1} by using the maximum principle, so we must find an alternate approach for general $q$. The approach we use is based on the asymptotically conic heat calculus\footnote{The heat calculus for metrics with iterated edge structures, which includes asymptotically conic metrics, is discussed in the end of Section 4 in \cite{alb}, and the composition law is given in Theorem 4.3 of that paper.} of Albin \cite{alb}. 
We use the scattering heat space $M^2_{{\rm heat}}$ of \eqref{M2heat}; 
recall that $\sca$ denotes the scattering face (the lift of the scattering face in $M_{\sca}^2$ to $M^2_{\rm heat}$),
and let ${\rm tf}$ be the temporal face of the final blowup (of the $t=0$ diagonal). 
We now define various classes of time-dependent pseudodifferential operators whose kernels are polyhomogeneous conormal on $M^2_{\rm heat}$, again following \cite{alb}. For any real numbers $\alpha_1$ and $\alpha_2$, we let $\Psi^{\alpha_1,\alpha_2}(M^2_{\rm heat})$ be the class of operators whose Schwartz kernels are polyhomogeneous conormal on $M_h^2$, with index set bounded below by $\alpha_1$ at $\sca$ and $\alpha_2$ at ${\rm tf}$, 
and with infinite-order decay at all other boundary hypersurfaces. Let $\Psi^*(M^2_{\rm heat})$ be the union of all such classes. Further, let $\Psi^{\infty,\infty}(M^2_{\rm heat})$ be the class of operators whose kernel is phg conormal on $M^2_{\rm heat}$ 
with infinite-order decay at all boundary hypersurfaces.
On these classes of operators, Albin shows that $t$-convolution 
\[ A*B(t)=\int_{0}^t A(t-s)B(s)ds\]
is a well-defined binary operation whenever the integrals in the convolution converge. In particular from the composition rule in \cite[Th. 4.3]{alb}, we obtain that  all convolutions of an operator in $\Psi^*$ with one in $\Psi^{\infty,\infty}$ are well-defined, and the class $\Psi^{\infty,\infty}$ is a two-sided ideal in $\Psi^*$. 

Now let $\chi_3,\chi_4\in C^\infty(M)$ be smooth functions so that $\chi_3+\chi_4=1$, $\chi_3$ has compact support, and $M$ is exactly conic on the support of $\chi_4$. Let $\tilde\chi_3$ be a smooth function with compact support which is equal to $1$ on a neighbourhood of the support of $\chi_3$ and let $\tilde\chi_4$ be a smooth function with $\til{\chi}_4=1$ in a neighbourhood of the support of $\chi_4$.
Finally, let $W$ be a compact manifold which agrees with $M$ whenever $\tilde\chi_3\not=0$. 
Then define an approximate heat kernel $A(\tau)$ whose integral kernel on $M^2_{\rm heat}$ is:
\[A(t,z,z')=\tilde\chi_3(z)H_q^W(t,z,z')\chi_3(z')+\tilde\chi_4(z)H_q^{C_N}(t,z,z')\chi_4(z').\]
Note that the limit of $A(t)$ as $t\rightarrow 0$ is the delta function. We claim that $||A(t,z,z)-H^q_Z(t,z,z)||$ has infinite-order decay in $|z|$ as $|z|\rightarrow\infty$.
To prove the claim, we compute
\[\begin{split}
E(t,z,z'):= & (\partial_t+\Delta_q^M)A(t,z,z')=
[\Delta_Z^q,\tilde\chi_3(z)]H_q^W(t,z,z')\chi_3(z')\\
&+ [\Delta_Z^q,\tilde\chi_4(z)]H_q^{C_N}(t,z,z')\chi_4(z').\end{split}\]
By the properties of the cutoff functions, we see that the support of $E(t,z,z')$ is bounded away from the lift of the spatial diagonal in $M^2_{\rm heat}$, and the outgoing support (in $z$) is compact. Therefore, $E(t,z,z')$ is zero in a neighborhood of sc and tf. Moreover, it is easy to see from the explicit form of $H_q^{C_N}(t)$ and the known properties of $H_q^W(t)$ that for $t<T$, each decays to infinite order, together with all derivatives, at each of the other boundary faces of $M^2_{\rm heat}$. Therefore, $E(t)\in\Psi^{\infty,\infty}(M^2_{\rm heat})$. On the other hand, by Duhamel's formula, $H_q^{M}(t)=A(t)-(H^q_M *E)(t)$. Since $\Psi^{\infty,\infty}$ is an ideal and we know that $H_q^M(t)\in 
\Psi^{0,-n}(M^2_{\rm heat})$ (Proposition \ref{zshorttime}), we have $A(t)-H^M_q(t)
\in\Psi^{\infty,\infty}(M^2_{\rm heat})$. 
This immediately implies the infinite-order decay of $||A(t,z,z)-H^M_q(t,z,z)||$, and hence the lemma.
\end{proof}

From this point, the proof of Proposition \ref{expansion} proceeds precisely as in \cite{s1}. Namely, we first use the conformal homogeneity of the heat kernel, as in the well-known work of Cheeger \cite{ch2}, to compute the asymptotic expansion in $\epsilon$ of
\[\int_{r\leq 1/\epsilon}\tra(H^{C_N}_q(\tau,r,y,r,y)) r^{n-1}dr dh_0(y)\]
in terms of the short-time heat invariants. Then we may use elementary calculus, as in Section 3.2 of \cite{s1}, to compute the expansion 
\[\int_{C_N}\chi_1(\epsilon r,y)\tra(H^{C_N}_q(\tau,r,y,r,y)) r^{n-1}drdh_0(y).\]
Finally, we use Lemma \ref{comparison} with $\tau$ replacing $t$ to finish the proof. We refer to \cite[Sec. 3.2]{s1} for more  details. 
\end{proof}

\begin{remark}\label{remimp}
We notice that the expansion of Proposition \ref{expansion} is really an expansion at the face $F$ of $Q_0$, and 
$\til{R}(\eps,\tau)$ is thus polyhomogeneous conormal on $Q_0$ with positive order at $F$ and zeroth order at $R$ by Theorem \ref{structure}.
\end{remark}

\subsection{Leading order at $\eps=0$}

The final ingredient needed to analyze the asymptotics of the determinant, as in \cite[Th. 2]{s2}, 
is the leading-order behavior of $\tra(H^{\Omega_\eps}_q(t))$ at the face $R$ of the $(\sqrt{t},\eps)$ blown-up space $Q_0$ (recall that $R$ is the face that projects to $\eps=0$ on $Q$ by the blow down map $Q_0\to Q$). 
The key is a spectral convergence result of Ann\'e and Takahashi (theorems A and B in \cite{at}), generalizing earlier work of Degeratu-Mazzeo and Rowlett in the $q=0$ case \cite{ma2, row}. Our assumption that $N$ satisfies the modified Witt condition simplifies the statement of their result by ruling out zero-resonances on $M$ (zero-resonances are called \emph{extended solutions} in their work). In the absence of zero-resonances, their result can be stated as:

\begin{lemma}\label{specconv}\cite{at} Let $n\geq 3$. As sets with multiplicity, the spectrum ${\rm Sp}(\Delta^{\Omega_\epsilon}_q)$ converges as $\epsilon\rightarrow 0$ to the union of ${\rm Sp}(\Delta_{\Omega_0}^q)$ 
with a set consisting of $M_q$ zeroes, where $M_q=\dim\ker_{L^2}\Delta^M_q$.
\end{lemma}
\begin{remark}
Recall that $\Delta^q_{\Omega_0}$ is the Friedrichs extension of the Laplacian; this particular extension is singled out because of the absence of zero-resonances. In the notation of \cite{at}, the space $W$ of elements which `generate extended solutions' is $\{0\}$, and hence the Gauss-Bonnet operator $D_{1,W}$ is the minimal extension $D_{1,\textrm{min}}$. Since $\Delta_{1,W}=(D_{1,W})^*D_{1,W}$, we see that $\Delta_{1,W}$ is the Friedrichs Laplacian. 
\end{remark}

In particular, Lemma \ref{specconv} implies that the number of `small eigenvalues' of $\Delta^q_{\Omega_{\epsilon}}$ which are positive for $\epsilon>0$ but converge to zero as $\epsilon\rightarrow 0$ is precisely 
\[N_q=M_q+\dim\ker\Delta_q^{\Omega_0}-\dim\ker\Delta_q^{\Omega_{\eps_0}}.\] 
Note that the Hodge-de Rham theorem implies that $\dim\ker\Delta_q^{\Omega_\epsilon}$ is independent of $\epsilon$. We call the small eigenvalues $\mu_1(\epsilon),\ldots,\mu_{N_q}(\epsilon)$, and break off the contribution from the kernel and the small eigenvalues to write
\begin{equation}\label{decomptraHq}
\tra(H_q^{\Omega_\epsilon}(t))=
\dim \ker\Delta_q^{\Omega_\epsilon}+\sum_{i=1}^{N_q}e^{-t\mu_i(\epsilon)}+\tra'(H_q^{\Omega_\epsilon}(t))
\end{equation}
where the last term is defined by the expression and correspond to the sum over eigenvalues bounded below 
by a positive constant as $\eps\to 0$. 
With this notation, we now prove:
\begin{proposition}\label{htconve}
For any fixed positive $t$, as $\eps \to 0$
\[{\rm Tr}'(H^{\Omega_\epsilon}_q(t))\rightarrow {\rm Tr}(H^{\Omega_0}_q(t))-\dim\ker\Delta^{\Omega_0}_q.\]
Moreover, there are constants $C<\infty$, $\delta>0$, and $\mu>0$ such that for any $\epsilon<\delta$ and any $t\geq 1$,
\begin{equation}\label{unifexp}
|{\rm Tr}'(H^{\Omega_\epsilon}_q(t))|\leq Ce^{-\mu t}.
\end{equation}
\end{proposition}
\begin{proof}
The proof is very similar to the proof of \cite[Th.4]{s2}. The idea is to obtain a suitable uniform lower bound on the non-small eigenvalues of $\Omega_\epsilon$ and then use Lemma \ref{specconv}. The first step is the following upper bound on the heat trace:
\begin{lemma}\label{prevlemma} 
Fix $T>0$. There is a constant $C$ such that for all $\epsilon\leq 1/2$ and all $t\leq T$,
\[|{\rm Tr}(H_q^{\Omega_\epsilon}(t))|\leq Ct^{-n/2}.\] 
\end{lemma}
\begin{proof} The function $t^{n/2}$ is polyhomogeneous conormal on $Q$ and hence lifts to be so 
on $Q_0$, with leading orders $n$ at L and F and 0 at R. By Theorem \ref{structure}, 
$t^{n/2}\tra(H^{\Omega_\epsilon}_q(t))$ is polyhomogeneous conormal on $Q_0$ with leading order 0 at each boundary face, and is therefore continuous up to the boundary of $Q_0$ - and hence bounded on the compact subset $\{\epsilon\leq 1/2,t\leq T\}$.
\end{proof}

By Lemma \ref{specconv}, there exist $\lambda_0>0$ and $\epsilon_0>0$ such that for all $\epsilon\leq\epsilon_0$, all the small eigenvalues $\mu_i(\epsilon)$ are less than $\lambda_0/2$ and all other eigenvalues of $\Omega_\epsilon$ are greater than $\lambda_0$; assume without loss of generality that $\lambda_0<1$. The needed lower bound on the non-small eigenvalues is:
\begin{lemma}\label{evalbd} Let $\tilde\lambda_{\epsilon,k}$ be the $k$-th eigenvalue of $\Delta^{\Omega_\epsilon}_q$ which is greater than or equal to $\lambda_0$. Then there are constants $C'>0$ and $N_0\in\mathbb N$, both independent of $\epsilon$, such that for all $\epsilon\in [0,\epsilon_0]$, $k\geq N_0$ 
\[\tilde\lambda_{\epsilon,k}\geq C'k^{2/n}.\] \end{lemma}
\begin{proof} Let $\tilde N_{\epsilon}(\lambda)$ be the number of eigenvalues of $\Delta^q_{\Omega_\epsilon}$ which are greater than or equal to $\lambda_0$ but less than or equal to $\lambda$. Applying Lemma \ref{prevlemma} with $T=n/2\lambda_0$, we have that for all $t\leq n/2\lambda_0$,
\begin{equation}\label{counting}
\tilde{N}_{\epsilon}(\lambda)e^{-\lambda t}\leq \tra'(H^{\Omega_\epsilon}_q(t))
\leq \tra(H^{\Omega_\epsilon}_q(t))\leq Ct^{-n/2}.\end{equation}
In particular, applying this with $t=n/2\lambda$ gives that $\tilde{N}_{\epsilon}(\lambda)\leq \til{C}\lambda^{\ndemi}$ for some $\til{C}$ independent of $\eps$, and thus taking $\lambda=\tilde\lambda_{\eps,k}$ gives the result.
\end{proof}

Finally, we use Lemmas \ref{specconv} and \ref{evalbd} to prove Proposition \ref{htconve}. 
For any $N\in\mathbb N$, we may write
\[\tra'(H^q_{\Omega_\epsilon}(t))=\sum_{k=0}^N e^{-t\tilde\lambda_{\epsilon,k}}+\sum_{k=N+1}^{\infty}e^{-t\tilde\lambda_{\epsilon,k}}.\]
The first term (finite sum) decays exponentially in $t$ uniformly in $\eps$ by the uniform lower bounds on $\til{\la}_{\eps,k}$, and the second term decays exponentially as well, by Lemma \ref{evalbd}; this proves the second claim of Proposition \ref{htconve}. In addition, the 
first term converges to the analogous sum for $\Omega_0$ by Lemma \ref{specconv}, 
and the remainder may be chosen as small as we like by taking $N$ large enough and using Lemma \ref{evalbd}; this easily implies the first claim in Proposition \ref{htconve}. Note also that by taking the pointwise limit as $\epsilon$ goes to zero, the same exponential bound in Proposition \ref{htconve} holds for $\tra(H^{\Omega_0}_q(t))-\dim\ker\Delta^{\Omega_0}_q$.\end{proof}

\subsection{Zeta function and determinant}

The zeta function and determinant may now be analyzed directly, again following section 3 of \cite{s2}. Denoting 
$\mc{H}^q(\Omega_\eps)=\ker \Delta_q^{\Omega_\eps}$,  we define
\begin{equation}\label{zetafn}
\zeta_q^{\Omega_\epsilon}(s)=\frac{1}{\Gamma(s)}\int_0^{\infty}
(\tra(H^{\Omega_\epsilon}_q(t))-\dim \mc{H}^q(\Omega_\epsilon))t^{s-1}\ dt.
\end{equation}
Notice that $\dim(\mc{H}^q(\Omega_\eps))$ is independent of $\eps$ since it is the $q$-th Betti number of $\Omega_\eps$.
To analyze (\ref{zetafn}), we break it into several pieces, as in \cite{s2}. The only differences between our case and the $q=0$ case are in the projection term ($\dim\mathcal H^q(\Omega_\epsilon)$ is not equal to $1$ in general) 
and the leading-order term of $\tra(H^{\Omega_\epsilon}_q(t))$ at the face $R$ corresponding to $\eps=0$ in the blown-up space $Q_0$. Therefore, the analysis will only be different for those terms which involve the leading-order term at R and/or the projection off the kernel. We now consider each of them in turn.

\subsubsection{Long-time contribution} First  we break up \eqref{zetafn} at $t=1$ 
and analyze the integral from 1 to infinity. Rewriting it in terms of $\tra'(H_q^{\Omega_\eps}(t))$:
\begin{equation}\label{int1infty}
\frac{1}{\Gamma(s)}\int_1^{\infty}
(\tra'(H^{\Omega_\epsilon}_q(t))+\sum_{i=1}^{N_q}e^{-\mu_i(\eps) t})t^{s-1}\ dt.
\end{equation}
We apply the dominated convergence theorem, together with Theorem \ref{htconve}, to analyze the part of the integral coming from $\tra'$. That part of the integral is holomorphic near $s=0$. As in Section 3.1 of \cite{s2}, its contribution to the derivative of the zeta function at zero is precisely
\[\int_1^{\infty}
\tra'(H^{\Omega_\epsilon}_q(t))t^{-1}\ dt.\]
The uniform exponential bound \eqref{unifexp}Ê on $\tra'(H^{\Omega_\epsilon}_q(t))$ allows us to apply the dominated convergence theorem to get a limit as $\epsilon\rightarrow 0$: we see that this term contributes precisely
\[\int_1^{\infty}(\tra(H^{\Omega_0}_q(t))-\dim\ker\Delta^{\Omega_0}_q)t^{-1}\ dt+o(1)\]
to the derivative of the zeta function, and by a similar argument, precisely
\[\frac{1}{\Gamma(s)}\int_1^{\infty}(\tra(H^{\Omega_0}_q(t))-\dim\ker\Delta^{\Omega_0}_q)t^{s-1}\ dt+o(1)\] to the zeta function itself in a neighborhood of $s=0$. 
The second part of the integral \eqref{int1infty} is the function
\[\sum_{i=1}^{N_q}\frac{1}{\Gamma(s)}\int_1^{\infty}e^{-\mu_i(\eps) t}t^{s-1}\ dt=\sum_{i=1}^{N_q}\Big(\mu_i(\eps)^{-s}-\frac{1}{\Gamma(s)}\int_0^1e^{-\mu_i(\eps)t}t^{s-1}\ dt\Big).\]
For each $i$, we write
\[\frac{1}{\Gamma(s)}\int_0^1e^{-\mu_i(\eps) t}t^{s-1}\ dt=\frac{1}{\Gamma(s)}\int_0^1(e^{-\mu_i(\eps) t}-1)t^{s-1}\ dt+\frac{1}{s\Gamma(s)}.\]
The integrand in the first term is holomorphic in a neighborhood of $s=0$, and hence its contribution to the determinant is $\int_0^1(e^{-\mu_i(\eps) t}-1)t^{-1}\ dt$, which goes to zero as $\epsilon$ (and hence $\mu_i(\eps)$) goes to zero. 
Putting everything together, the long-time contribution as $\eps\to 0$ to the zeta function in a neighborhood of $s=0$ is:
\begin{equation}\label{longtimecont}
\begin{gathered}
\frac{1}{\Gamma(s)}\int_1^{\infty}
(\tra(H^{\Omega_\epsilon}_q(t))-\dim \mc{H}^q(\Omega_\eps))t^{s-1}\ dt=\\
\frac{1}{\Gamma(s)}\int_1^{\infty}(\tra(H^{\Omega_0}_q(t))-\dim\ker\Delta^{\Omega_0}_q)t^{s-1}\ dt+\sum_{i=1}^{N_q}\mu_i(\eps)^{-s}-\frac{N_q}{s\Gamma(s)}+o(1).
\end{gathered}\end{equation} 

\subsubsection{Small time contribution} 
Consider the short-time zeta function:
\[\frac{1}{\Gamma(s)}\int_0^{1}
(\tra(H^{\Omega_\epsilon}_q(t))-\dim\mathcal H^q(\Omega_\epsilon))t^{s-1}\ dt=\frac{1}{\Gamma(s)}\int_0^{1}
\tra(H^{\Omega_\epsilon}_q(t))t^{s-1}\ dt-\frac{\dim\mathcal H^q(\Omega_\epsilon)}{s\Gamma(s)}.
\]
We follow the method of Section 3.3, 3.4 and 3.5 in \cite{s2} to analyze the expansion in $\eps$ of
\[\frac{1}{\Gamma(s)}\int_0^{1}f_\eps(t)t^{s-1}\ dt \quad \textrm{with } f_\eps(t)=\tra(H^{\Omega_\epsilon}_q(t)). \]
The idea is to fix $b>0$ and to break the integral into integrals over the regions $A:=\{\sqrt t>b\epsilon\}$ and $B:=\{\sqrt t\leq b\epsilon\}$.
The region $A$ localizes near the face $R$ in $Q_0$ and the region $B$ localizes near $L$. We first write $f_\eps(t)=f_R^0(t)+\til{f}_{\eps}(t)$, where $f^{0}_R(t)$ is the leading coefficient of $f_\eps(t)$ at the face $R$ in $Q_0$;
from \eqref{decomptraHq} and Proposition \ref{htconve}, we have
\[f_R^0(t)=\tra(H^{\Omega_0}_q(t))-\dim\ker\Delta^q_{\Omega_0}+\dim\mathcal H^q(\Omega_\epsilon)+N_q.\]
Then we write 
\begin{equation}\label{smalltimedec}
\begin{split} 
\frac{1}{\Gamma(s)}\int_{b^2\eps^2}^{1}f_\eps(t)t^{s-1}dt=&
 \frac{1}{\Gamma(s)}\int_0^1(\tra(H^{\Omega_0}_q(t))-\dim\ker\Delta^q_{\Omega_0})t^{s-1}\ dt+\frac{\dim\mathcal H^q_{\Omega_\epsilon}+N_q}{s\Gamma(s)}\\
& -\frac{1}{\Gamma(s)}\int_0^{b^2\epsilon^2}f_R^0(t)t^{s-1}\ dt + \frac{1}{\Gamma(s)}\int_{b^2\epsilon^2}^1\til{f}_\eps(t)t^{s-1}dt.\end{split}\end{equation}
Using the small time expansion of $\tra(H_q^{\Omega_0}(t))$ due to Cheeger \cite{ch2}, 
we see that as $t\rightarrow 0$ 
\[f_R^0(t)=\sum_{k=0}^{n}a_{k}t^{(k-n)/2}+\mc{O}(t^\alpha), \,\, \alpha>0\] 
for some $a_k$. Again there is no $\log(t)$ coefficient due to the fact we work in odd dimension. This directly gives the $\eps$-expansion of the integral $\frac{1}{\Gamma(s)}\int_0^{b^2\epsilon^2}f_R^0(t)t^{s-1} dt$, and combining with \eqref{smalltimedec} andÊ \eqref{longtimecont}, we deduce that for $s$ near $0$, one has as $\eps\to 0$
\begin{equation}\label{zetaq1}
\begin{split}
\zeta^{\Omega_\eps}_q(s) = & \zeta_q^{\Omega_0}(s)+\sum_{i=1}^{N_q}\mu_i(\eps)^{-s}+ \frac{1}{\Gamma(s)}\Big(\int_{b^2\epsilon^2}^1\til{f}_\eps(t)t^{s-1}dt +\int_{0}^{b^2\epsilon^2}f_\eps(t)t^{s-1}dt\Big)\\
& - \frac{1}{\Gamma(s)}\sum_{k=0}^n\frac{2a_k(b\eps)^{k-n+2s}}{k-n+2s}+o(1).
\end{split}\end{equation}
Moreover we easily check (for instance see the details in \cite[Sec. 3.3]{s2}) that the $o(1)$ term is uniform in $C^1$ norm in the $s$ parameter near $s=0$. Now we use the expression \eqref{structureeq} of $f_\eps(t)$, the small time expansion of the local trace $\tra(H_q^{\Omega_0}(t z,z))$ and Proposition \ref{expansion} to express 
\[
\begin{split}
\int_{0}^{b^2\epsilon^2}f_\eps(t)t^{s-1}dt= &\int_{0}^{b^2\epsilon^2}\int_{\Omega_0}\chi_2(z)\tra(H_q^{\Omega_0}(t,z,z))dg_0(z)t^{s-1}dt\\
&+\eps^{2s}\int_{0}^{b^2}\int_{M}\chi_1(\eps z)\tra(H_q^{M}(\tau,z,z))dg_0(z)\tau^{s-1}d\tau+o(1)\\
= &\sum_{k=0}^n\frac{2\til{a}_k(b\eps)^{k-n+2s}}{k-n+2s}+ + \eps^{2s}\int_{0}^{b^2}({^R}\tra(H_q^M(\tau))+\til{R}(\eps,\tau))\tau^{s-1}d\tau+o(1)
\end{split}\]
for some $\til{a}_k\in \rr$. According to Proposition \ref{polyhomof69}, the $\til{R}(\eps,\tau)$ term is polyhomogeneous conormal in $(\eps,\tau)\in[0,1]^2$ and vanishes at $\eps=0$. By the argument of Section 3.6.1 of \cite{s2}, we see that $\frac{\eps^{2s}}{\Gamma(s)}\int_0^{b^2}\til{R}(\eps,\tau)\tau^{s-1}d\tau$ will be $o(1)$ in $\eps\to 0$ uniformly in $s$ near $0$. 

We also have from  \eqref{structureeq} that as $\eps\to 0$
\[\begin{gathered}\int_{b^2\epsilon^2}^1\til{f}_\eps(t)t^{s-1}dt=\int_{b^2\eps^2}^{1}
\Big(\int_{M}\chi_1(\eps z)\tra(H_q^{M}(\tfrac{t}{\eps^2},z,z))dg_0(z)-h(t)\Big)t^{s-1}dt+o(1),\\ 
\textrm{ where }h(t)=\lim_{\eps\to 0}\int_{M}\chi_1(\eps z)\tra(H_q^{M}(t/\eps^2,z,z))dg_0(z).\end{gathered}\]
Change variables to $\tau=t/\eps^2$; this becomes 
\[\int_{b^2\epsilon^2}^1\til{f}_\eps(t)t^{s-1}dt=\eps^{2s}\int_{b^2}^{\eps^{-2}}
\Big(\int_{M}\chi_1(\eps z)\tra(H_q^{M}(\tau,z,z))dg_0(z)-h(\eps^2\tau)\Big)\tau^{s-1}d\tau+o(1).\]
We now proceed as in Section 3.6.2 of \cite{s2}: using the expansion of Proposition \ref{expansion} at the face $F$ of $Q_0$, the polyhomogeneity statement of Proposition \ref{polyhomof69}, and Remark \ref{remimp}, we obtain that 
$^{R}{\rm Tr}(H^M_q(\tau))=f_\infty+o(1)$
for some $f_\infty\in \rr$, and 
\[h(t)=\sum_{k=0}^{n-1}l_{k}f_{k}(t)+f_{\infty}+ \lim_{\delta\to 0}\til{R}(\delta,t/\delta^2),
\]
with $\lim_{\delta\to 0}\til{R}(\delta,t/\delta^2)$ a well-defined polyhomogeneous function of $t$ which 
goes to $0$ as $t\to 0$. 
From this one gets 
\[\frac{1}{\Gamma(s)}\int_{b^2\epsilon^2}^1\til{f}_\eps(t)t^{s-1}dt=\frac{\eps^{2s}}{\Gamma(s)}\int_{b^2}^{\eps^{-2}}\Big({^R}{\rm Tr}(H^M_q(\tau))-
 f_\infty+\til{R}(\eps,\tau)-\lim_{\delta\to 0}\til{R}(\delta,\tfrac{\tau\eps^2}{\delta^2}) \Big)\tau^{s-1}
d\tau.\]
The term $\til{R}(\eps,\tau)-\lim_{\delta\to 0}\til{R}(\delta,\tfrac{\tau\eps^2}{\delta^2})$ is polyhomogeneous conormal 
on $Q_0$ (as a function of $(\eps,\sqrt{t}=\eps\sqrt{\tau})$) and vanishes at all boundary faces. It is easy to see, as in \cite{s2}, that it will contribute a uniform $o(1)$ term to the integral as $\eps\to 0$, with $C^1$ dependence in $s$ near $s=0$. As for the ${^R}{\rm Tr}(H^M_q(\tau))- f_\infty$ part (which goes to $0$ as $\tau\to \infty$), 
a straightforward computation gives as $\eps\to 0$
\[\begin{gathered}
\frac{\eps^{2s}}{\Gamma(s)}\int_{b^2}^{\eps^{-2}}({^R}{\rm Tr}(H^M_q(\tau))-f_\infty)\tau^{s-1}d\tau=\\
\frac{f_\infty (b\eps)^{2s}}{s\Gamma(s)}
+\frac{\eps^{2s}}{\Gamma(s)}\int_{b^2}^\infty {^R}{\rm Tr}(H^M_q(\tau))\tau^{s-1}d\tau+o(1),
\end{gathered}\]
with $C^1$ dependence in $s$ near $s=0$ for the remainder (again we refer to Section 3.6.2 of \cite{s2} for more details). We finally obtain 
\[\begin{split}
\zeta^{\Omega_\eps}_q(s) = & \zeta_q^{\Omega_0}(s)+\sum_{i=1}^{N_q}\mu_i(\eps)^{-s}
+ \sum_{k=0}^n\frac{2(\til{a}_k-a_k)(b\eps)^{k-n+2s}}{\Gamma(s)(k-n+2s)}+\frac{f_\infty (b\eps)^{2s}}{s\Gamma(s)}\\
&+\frac{\eps^{2s}}{\Gamma(s)}\int_{0}^\infty {^R}{\rm Tr}(H^M_q(\tau))\tau^{s-1}d\tau+o(1).
\end{split}\]
with $C^1$ dependence in $s$ near $s=0$ for the remainder. Fix $s<0$ small; since 
$\zeta^{\Omega_\eps}_q(s)$ is independent of $b$, we deduce directly that 
\[\zeta^{\Omega_\eps}_q(s) = \zeta_q^{\Omega_0}(s)+\sum_{i=1}^{N_q}\mu_i(\eps)^{-s}+
\frac{\eps^{2s}}{\Gamma(s)}\int_{0}^\infty {^R}{\rm Tr}(H^M_q(\tau))\tau^{s-1}d\tau+o(1).
\]
All these terms are $C^1$ in $s$ near $s=0$, thus we get as $\eps\to 0$: 
\[-\pl_s\zeta^{\Omega_\eps}_q(0)=-\pl_s\zeta^{\Omega_\eps}_q(0)+\sum_{i=1}^{N_q}\log \mu_i(\eps)
-2\log(\eps)(\zeta_q^M(0))-\pl_s(\zeta_q^M)(0)+o(1).\]
This completes the proof of Theorem \ref{conictorsion}.

\subsection{Cases with no small eigenvalues} The small eigenvalues in Theorem \ref{conictorsion} are the major remaining obstacle to a complete analysis of the analytic torsion under conic degeneration. However, in certain cases, it is possible to prove that $N_q=0$ for all $q$, which removes the undetermined term in Theorem \ref{conictorsion}:

\begin{lemma}\label{conditions} 
Suppose $n$ is odd and the cross-section $N$ is such that:

a) the cohomology $H^q(N)$ is trivial for all $q\in [1,n-2]$,

b) the spectrum of $\Delta_N$ on forms satisfies 
\[{\rm Sp}_{\Lambda^{(n-1)/2}}(\Delta_N)\cap [0,15/4)=\emptyset, \quad 
{\rm Sp}_{\Lambda^{(n-3)/2}}(\Delta_N|_{{\rm Im}\, d_N})\cap [0,7/4)=\emptyset.\] 
Then Theorem \ref{conictorsion} holds with $N_q=0$ for all $q$. (Note that condition b) is strictly stronger than the modified Witt condition). 

 \end{lemma}

\begin{proof}
We shall follow the discussion in section 5 of \cite{at}.
Consider the manifold $\Omega_{\eps_0}$ for $\eps_0>0$ small, 
and split it as $K\cup M_{\eps_0}$ (as in the Introduction) so that $\pl M_{\eps_0}=N=\pl K$.
Then there is a Mayer-Vietoris sequence:
\[\ldots\to H^q(\Omega_{\eps_0})\to H^q(K)\oplus H^q(M_{\eps_0})\to H^q(N)\to H^{q+1}(\Omega_{\eps_0})\to\ldots\]
Since $H^q(N)$ is trivial for all $q$ between $1$ and $n-1$, this sequence splits, and we see that for any $q$ with $2\leq q\leq n-1$,
\[H^q(\Omega_{\eps_0})\cong H^q(M_{\eps_0})\oplus H^q(K).\]

On the one hand, since $M_{\eps_0}$ is homeomorphic to $\bbar{M}$, we may apply the results of Hausel, Hunsicker, and Mazzeo \cite{hhm} to relate the $L^2$-cohomology of $\bbar{M}$ to the cohomology of $M_{\eps_0}$. In particular, for $q>n/2$, $\ker_{L^2}(\Delta_q^M)\simeq  H^q(M_{\eps_0})$, while for $q<n/2$, 
$\ker_{L^2}(\Delta_q^M)\cong H^q(M_{\eps_0},N)$. On the other hand, by condition a), the relative and absolute cohomology are isomorphic in all degrees greater than zero (from the usual long exact sequence in cohomology); so for all $q$ with $2\leq q\leq n-1$, $H^q(M_{\eps_0})\cong \ker_{L^2}(\Delta_q^M)$.

On the other hand, we can relate the cohomology of $K$ to $\dim\ker\Delta^q_{\Omega_0}$. First we need to compare the kernel of $\Delta^{\Omega_{\eps_0}}_q$ to the $L^2$ cohomology of $\Omega_0$. The problem is that there may be $L^2$ harmonic forms which are not Friedrichs if $\Delta_q^{\Omega_0}$ is not essentially self-adjoint near the cone tip. However we claim that  if the cross-section $N$ is such that the indicial set $\mc I(\Delta_q^M)$ has no elements in $(0,1)$, then the Hodge Laplacian in degree $q$ on $\Omega_0$ is essentially self-adjoint.
This fact is essentially due to Bruning-Seeley \cite{brs}, but we use the exposition in Loya-McDonald-Park \cite{lmp}: 
near the cone tip of $\Omega_0$, the metric is precisely $dr^2+r^2h_0$. Then there is a decomposition for the $q$-form bundle similar to the one near infinity on $M$, and we follow the normalization conventions of \cite{lmp}; each $q$-form $\phi$ near the cone point $r=0$ can be written as
\[\phi = r^{q+\frac{n-1}{2}}\phi_t+r^{q-1+\frac{n-1}{2}}dr\wedge\phi_n,\]
where $\phi_t\in\Lambda^q(N)$ and $\phi_n\in\Lambda^{q-1}(N)$. On this decomposition $\Lambda^{q}(N)\oplus\Lambda^{q-1}(N)$ of the $q$-form bundle near the cone point, the Laplacian is
\[-\partial_r^2+\frac{1}{r^2}A_q,\]
where, after some easy algebra,
\[A_q+\frac{1}{4}I=\left(\begin{array}{cc}
\Delta_N+(\ndemi-q-1)^2 & -2d_N\\
-2\delta_N & \Delta_N+(\ndemi-q+1)^2
\end{array}
\right).\]
Then as in \cite{lmp}, $\Delta^q_{\Omega_0}$ is essentially self-adjoint if $A_q$ 
has no eigenvalues in $(-1/4,3/4)$; i.e. if $A_q+\frac{1}{4}I$ has no eigenvalues in $(0,1)$. However, except for the signs of the off-diagonal elements, $A_q+\frac{1}{4}I$ is precisely the same matrix as we obtained in the calculation of the indicial roots of $P_b$ in \eqref{indicialop} (with $\la=0$).
We can perform the same decomposition to compute the eigenvalues. Since multiplying both off-diagonal elements by $-1$ changes neither the trace nor determinant and hence does not change the eigenvalues, the eigenvalues of $A_q+\frac{1}{4}I$ are precisely $\mc{I}(P_b)=\mc{I}(\Delta_q^M)$.

Using the form of the indicial set of $\Delta_q^M$ in \eqref{I1I2} and \eqref{i3}, we see that condition b) in Lemma \ref{conditions} above implies $\mc{I}(\Delta_q^M)\cap (0,1)=\emptyset$ for all $q$. Hence the space of harmonic forms in the Friedrichs domain of $\Delta_q^{\Omega_0}$ is isomorphic to the $L^2$ cohomology of $\Omega_0$.

However, by work of Cheeger \cite{ch3}, the Witt condition implies that the $L^2$ cohomology of $\Omega_0$ is isomorphic to the intersection cohomology of $\Omega_0$ with middle perversity. Further, the intersection cohomology of middle perversity of $\Omega_0$ is isomorphic to either the cohomology of $K$, the relative cohomology of $K$, or the image of the relative cohomology of $K$ in the absolute cohomology of $K$, depending on the degree \cite{ch3}. As before, condition a) implies that the relative and absolute cohomologies are isomorphic for all degrees between $2$ and $n-1$. So for $2\leq q\leq n-1$.
\[\dim\ker(\Delta_q^{\Omega_0})=\dim H^q(K).\] 
Putting everything together, we have that for all $q$ with $2\leq q\leq n-1$,
\[\dim H^q(\Omega_{\eps_0})=\dim\ker(\Delta^{\Omega_0}_q)+\dim \ker_{L^2}(\Delta_q^M),\]
and hence by \eqref{numbersmall}, $N_q=0$ for $2\leq q\leq n-1$. But since $N_q=N_{n-q}$ by duality and $N_0=0$ by \cite{s2}, $N_q=0$ for all $q$, which completes the proof.
\end{proof}

\textbf{Example.} The conditions of Lemma \ref{conditions} are satisfied, for instance, by taking any Riemannian compact manifold $N$ with metric $h_0$ which satisfies $H^{q}(N)=0$ for $q\in[1,n-2]$ and then rescaling: write $h_\la=\la^{2}h_0$. The manifold $(N,h_\la)$ satisfies the conditions of the Lemma for $\la>0$ small enough, since the spectrum on all $q$-forms with $q\in[1,n-2]$ scales by $\la^{-2}$. We also notice that the case of the canonical sphere $(N=\sph^{n-1},d\theta^2)$ with $n\geq 5$ satisfies the conditions of Lemma \ref{conditions}, since $\nu_0=n/2-1>1$ for all $q$.

\end{document}